%% file: zeta_pseudo_Anosov.tex
\documentclass[10pt,a4paper]{article}
\usepackage[utf8]{inputenc}
\usepackage{amsmath}
\usepackage{amsfonts}
\usepackage{amssymb}
\usepackage{amsthm}
\usepackage{enumitem}
\usepackage{xcolor}
\usepackage{soul}
\usepackage{mathtools}
\usepackage{hyperref}
\usepackage{graphicx}
\usepackage{tikz-cd}
\usepackage[capitalize]{cleveref}
\usepackage{caption}

\newtheorem{lemma}{Lemma}
\newtheorem{proposition}[lemma]{Proposition}
\newtheorem{corollary}[lemma]{Corollary}
\newtheorem{theorem}{Theorem}
\newtheorem*{conjecture}{Conjecture}
\theoremstyle{definition}
\newtheorem{definition}[lemma]{Definition}
\newtheorem{example}[lemma]{Example}
\theoremstyle{remark}
\newtheorem{remark}[lemma]{Remark}

\numberwithin{equation}{section}
\numberwithin{lemma}{subsection}

\newcommand{\set}[1]{\left\{#1\right\}}
\newcommand{\n}[1]{\left\|#1\right\|}
\newcommand{\R}{\mathbb{R}}
\newcommand{\C}{\mathbb{C}}
\newcommand{\Z}{\mathbb{Z}}
\newcommand{\orient}{\varepsilon}
\newcommand{\W}{{\mathbb C^m}}

\DeclareMathOperator{\tr}{tr}
\DeclareMathOperator{\im}{Im}
\DeclareMathOperator{\image}{im}
\DeclareMathOperator{\re}{Re}

\DeclareMathOperator{\lk}{lk}

\title{Zeta functions and the Fried conjecture for smooth pseudo-Anosov flows}
\author{Malo J\'ez\'equel\footnote{CNRS, Univ. Brest, UMR 6205, Laboratoire de Mathématiques de Bretagne Atlantique, France. email: malo.jezequel@math.cnrs.fr} \and Jonathan Zung\footnote{Georgia Institute of Technology, Atlanta GA, USA. email: jzung3@gatech.edu}}
\date{}

\begin{document}

\maketitle

\begin{abstract}
    To a transitive pseudo-Anosov flow $\varphi$ on a $3$-manifold $M$ and a representation $\rho$ of $\pi_1(M)$, we associate a zeta function $\zeta_{\varphi,\rho}(s)$ defined for $\re s \gg 1$, generalizing the Anosov case. For a class of ``smooth pseudo-Anosov flows'', we prove that $\zeta_{\varphi,\rho}(s)$ has a meromorphic continuation to $\mathbb{C}$. We also prove a version of the Fried conjecture for smooth pseudo-Anosov flows which, under some conditions on $\rho$, relates $\zeta_{\varphi,\rho}(0)$ to the Reidemeister torsion of $M$. Finally we prove a topological analogue of the Dirichlet class number formula. In order to deal with singularities, we use $C^\infty$ versions of the approaches of Rugh and Sanchez--Morgado, based on Markov partitions.
\end{abstract}

\tableofcontents
\section{Introduction}

\subsection{The Fried conjecture}
An old theme in topology is that the number of fixed points or closed orbits of a dynamical system is a topological invariant. For example, the Lefschetz fixed point theorem computes a homological trace of a continuous self-map of a topological space by counting its fixed points with appropriate weights. Fried proposed such a counting theorem for Anosov flows. Let $M$ be a closed 3-manifold and $\varphi$ a smooth transitive Anosov flow on $M$. Let $\rho$ be an acyclic unitary representation of $\pi_1(M)$. Now an Anosov flow has infinitely many closed orbits, so one cannot count them in a naive way. Instead, one may count them using zeta function regularization. We form the \emph{twisted Ruelle zeta function} via an Euler product over primitive closed orbits:
\begin{equation}\label{eqn:anosov_zeta_definition}
\zeta_{\varphi,\rho}(s) = \prod_{\gamma \text{ primitive closed orbit of } \varphi} \det(I-e^{-s T_\gamma} \Delta_\gamma \rho(\gamma))
\end{equation}
where $T_\gamma$ is the length of $\gamma$ and $\Delta_\gamma = \pm 1$ depending on the orientation of the unstable bundle over $\gamma$. The product converges when $\re s \gg 1$ and admits a meromorphic continuation to $\mathbb{C}$ \cite{giulietti.liverani.ea.AnosovFlowsDynamical}. Fried proposed a ``Lefschetz formula for flows'' stated in terms of this twisted zeta function \cite{fried_lefschetz,fried95}:
 \begin{conjecture}[Fried conjecture]\label{conj:fried}
	The number $|\zeta_{\varphi,\rho}(0)|^{-1}$ is a topological invariant of $M$, the $\rho$-twisted Reidemeister torsion of $M$.
\end{conjecture}

Fried proved his conjecture for the geodesic flow on surfaces of constant negative curvature \cite{fried.AnalyticTorsionClosed, fried_fuchsian}. Even though the current paper is mostly concerned about the case of hyperbolic flows, Fried conjecture has been investigated in other contexts. It is satisfied by the geodesic flow on the unit tangent bundle of a locally symmetric reductive manifold \cite{moscovici_stanton, shen_locally_symmetric} and a generalization to the case of locally symmetric orbifolds is also known \cite{shen_yu_orbifolds}. We refer to the survey \cite{shen_survey} for a detailed introduction to the Fried conjecture, and will from now on keep to the case of Anosov (and pseudo-Anosov) flows.

The proof of the existence of a meromorphic extension for \eqref{eqn:anosov_zeta_definition} has a long history. Ruelle established meromorphic extension to $\mathbb{C}$ for analytic Anosov flows with analytic strong stable or unstable foliation \cite{ruelle.ZetafunctionsExpandingMaps} and to a smaller domain for $C^1$ Anosov flows \cite{ruelle_zeta_holder}. The latter result has then been improved by Pollicott \cite{pollicott_zeta} and Haydn \cite{haydn_zeta}. They prove the existence of a meromorphic continuation for the zeta function associated to a $C^1$ Axiom A flow to a larger domain that is still smaller than $\mathbb{C}$ (but in this context we do not expect the existence of a meromorphic continuation to the whole $\mathbb{C}$ as explained in \cite{pollicott_zeta}). All these results for $C^1$ hyperbolic flows are obtained using symbolic dynamics, the interested reader may refer to the textbook \cite{parry_pollicott_asterisque} for an exposition of this approach. With the emergence of spaces of anisotropic distributions in the hyperbolic dynamical systems literature, it became possible to prove meromorphic continuation to $\mathbb{C}$ of zeta functions associated to large classes of flows. Rugh proved the existence of a meromorphic extension for the zeta function associated to an analytic Anosov flows in dimension 3 \cite{rugh96}, and his result was then extended to higher dimensions by Fried \cite{fried95}. More recently, Giulietti--Liverani--Pollicott established meromorphic extension of the zeta function in the case of smooth ($C^\infty$) Anosov flows \cite{giulietti.liverani.ea.AnosovFlowsDynamical}. Another proof was given by Dyatlov and Zworski, using microlocal analysis \cite{dyatlov.zworski.DynamicalZetaFunctions}. 

The different approaches to the meromorphic extension of \eqref{eqn:anosov_zeta_definition} mentioned above have been used to prove Fried's conjecture in different contexts. Using Ruelle's approach to the zeta function, Sanchez-Morgado \cite{morgado93} proved Fried conjecture first for analytic Anosov flows with analytic stable foliation, under the assumption of the existence of a periodic orbit $\gamma$ for the flow such that $1$ and $\Delta_\gamma$ are not eigenvalues for $\rho(\gamma)$ (we will call this the \emph{monodromy condition}). Using the approach of Rugh and Fried, he was then able to remove the assumption of analyticity on the stable foliation \cite{morgado96}. Finally, using Dyatlov-Zworski's approach, Dang et al. established the Fried conjecture for smooth Anosov flows in dimension 3 \cite{dang.guillarmou.ea.FriedConjectureSmall}, under an additional condition (the absence of resonance at zero) that is automatically satisfied when the flow $\varphi$ is volume preserving. Their proof is based on Sanchez-Morgado's result, but they are able to get rid of the monodromy condition if the first Betti number of the manifold $M$ is non-zero.

In this paper, we study the Fried conjecture for pseudo-Anosov flows. Before stating our results precisely, let us explain how to extend the definition of the Ruelle zeta function to pseudo-Anosov flows.

\subsection{Zeta function for pseudo-Anosov flows}\label{subsection:definition_zeta_pseudo_Anosov}
	Let $\varphi = (\varphi_t)_{t \in \mathbb{R}}$ be a smooth transitive pseudo-Anosov flow on a closed $3$-manifold $M$ and $\rho$ be a representation of the fundamental group of $M$ (the choice of the base point does not matter). We use the adjective ``smooth'' here to express the fact that we are not working with the standard definition of pseudo-Anosov flows that can be found in the literature (which is a topological notion). Indeed, the existence of a meromorphic continuation for the zeta function of a hyperbolic flow is deeply related to its regularity properties, so that we need to impose some smoothness assumptions on the flows we are working with (even if they have singular periodic orbits in general). We will work with flows that are smooth away from a finite number of singular orbits, and with a particular model near each singular orbit (see \S \ref{section:definition_pseudo_Anosov} for details). See \cite{agol_tsang} for a discussion of the notions of smooth and topological pseudo-Anosov flows (beware though that their notion of smooth pseudo-Anosov flows is not exactly the same as ours).
	
	 The precise class of flows we can work with is defined in \cref{section:definition_pseudo_Anosov} but let us give an informal description of the properties of $\varphi$. We assume that there are a finite number of primitive periodic orbits\footnote{Our convention is that a periodic orbit is an integral curve $\gamma : [0,T] \to M$ for $\varphi$ such that $\gamma(0) = \gamma(T)$. We identify two periodic orbits if they only differ by shifting the parametrization. We say that $\gamma$ is primitive if $\gamma_{|[0,T)}$ is injective. Moreover, if $k$ is a positive integer, we let $\gamma^k$ denote the periodic orbit $[0,kT] \to M$ for $\varphi$ obtained by running $k$ times through $\gamma$. Notice that the $\gamma^k$'s all have the same image, but they a priori have different free homotopy classes. We may sometimes identify a primitive periodic orbit and its image.}  $\gamma_1,\dots,\gamma_N$ for $\varphi$ such that $\varphi$ is smooth on $\widetilde{M} \coloneqq M \setminus \bigcup_{i = 1}^N \gamma_i$ and satisfies there a hyperbolic property similar to the definition of an Anosov flow. Moreover, we require that the orbits $\gamma_1,\dots,\gamma_N$ have neighbourhoods on which the flow $\varphi$ is conjugate to a local model of a ``pseudo-hyperbolic singular orbit'' (defined in \cref{subsubsection:pseudo_hyperbolic_orbit}). The local model is obtained as a suspension of a ``pseudo-hyperbolic fixed point'' constructed in the following way: take a linear hyperbolic fixed point and cut it along the unstable direction, you get then a hyperbolic map acting on a half-space, glue more than two of these half-spaces together (and maybe make the map permute the half-spaces). The resulting singular orbits still have some hyperbolic behaviour, but instead of having two unstable and two stable prongs (like a smooth hyperbolic periodic orbit) they have more. We shall now explain how to take into account these exceptional orbits in the definition of a twisted zeta function.
		
	For $\gamma$ a periodic orbit for $\varphi$, define the auxiliary function
	\begin{equation}\label{eq:auxiliary_function} 
	\xi_{\varphi,\rho,\gamma}(s) = \det(I-e^{-s T_\gamma }\rho(\gamma)),
	\end{equation}
	where $T_\gamma$ denotes the length of $\gamma$. Now suppose $\gamma$ is a primitive closed orbit. Suppose there are $m$ weak unstable half-leaves $\lambda_1^{u}, \dots, \lambda_m^{u}$ whose boundary is some multiple of $\gamma$. Let $r_i^{u}$ be the number of times $\lambda_i$ wraps around $\gamma$. See \cref{fig:pushoff}. Note that $\sum r_i^{u}$ is the number of prongs at $\gamma$. 
	\begin{figure}
		\centering
	\def\svgwidth{1.2in}
		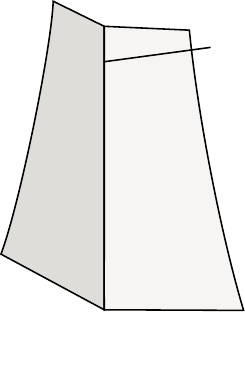
		\caption{The neighbourhood of a 3-prong singular orbit. The bottom of the picture glues to the top with a rotation of $2\pi/3$. Here, $m=1$ and $r_1^u=3$.}\label{fig:pushoff}
	\end{figure}
	We define the local zeta function at $\gamma$ as
	\begin{equation}\label{eq:definition_local_zeta_function}
	\zeta_{\varphi,\rho,\gamma}(s) = \frac{\prod_{i=1}^m\xi_{\varphi,\rho,\gamma^{r_i^{u}}}(s)}{\xi_{\varphi,\rho,\gamma}(s)}.
	\end{equation}
	
	In other words, $\lambda_i^{u}$ counts as a closed orbit of length $\gamma^{r_i^u}$, and we deduct one orbit which wraps once around $\gamma$. In the case where the first return map along $\gamma$ is orientation preserving, one may motivate this definition by the fact that the first return map along $\gamma$ can be perturbed to a number of hyperbolic fixed points counted in the numerator of \cref{eq:definition_local_zeta_function} and a single elliptic fixed point counted in the denominator; see \cref{fig:perturb}. A different kind of perturbation may be used in the orientation reversing case; see \cref{fig:perturb_reflection}.
	
	Let $\Delta$ be the orientation class of the strong unstable bundle, well defined in the complement of the singular orbits of $\varphi$. Let $\varepsilon$ be the orientation class of $TM$. Given a closed orbit $\gamma$, we define $\Delta_\gamma$ or $\varepsilon_\gamma$ to be the value of the corresponding class on $\gamma$ (in the first case we need $\gamma$ to be non-singular).
	
	\begin{remark}\label{remark:orientable half leaves}
		When $\gamma$ is non-singular and $\lambda$ is a half weak stable leaf whose boundary is $\gamma^r$, the tangent bundle $T\lambda$ is orientable. Since $TM = T\lambda \oplus T\lambda^\perp$ and $\Delta$ is the orientation class of $T\lambda^\perp$, we have $\Delta_{\gamma^r} = \varepsilon_{\gamma^r}$.
	\end{remark}
	
	\begin{remark}\label{remark:nonsingular_zeta}
		When $\gamma$ is nonsingular, \cref{eq:definition_local_zeta_function} simplifies to
		\begin{equation*}
		\zeta_{\varphi,\rho,\gamma}(s) = \det(I-\Delta_\gamma e^{-s T_\gamma} \rho(\gamma)).
		\end{equation*}
		Indeed, if $\Delta_\gamma = 1$ (which in this case just means that the unstable direction is orientable over $\gamma$), then $n = 2$ and $r_1^u = r_2^u = 1$, so that $$\zeta_{\varphi,\rho,\gamma}(s) = \frac{\xi_{\varphi,\rho,\gamma}(s)^2}{\xi_{\varphi,\rho,\gamma}(s)} = \det(I-e^{-s T_\gamma }\rho(\gamma)).$$ If $\Delta_\gamma = -1$ then $n = 1$ and $r_1^u = 2$ so that
		\begin{equation*}
		\zeta_{\varphi,\rho,\gamma}(s) = \frac{\xi_{\varphi,\rho,\gamma^2}(s)}{\xi_{\varphi,\rho,\gamma}(s)} = \frac{\det(I-e^{-2s T_\gamma }\rho(\gamma)^2)}{\det(I - e^{-sT_\gamma} \rho(\gamma))} = \det(I + e^{-s T_\gamma} \rho(\gamma)).
		\end{equation*}
		\end{remark}
	
	\begin{figure}
		\centering
		\includegraphics[width=0.75\textwidth]{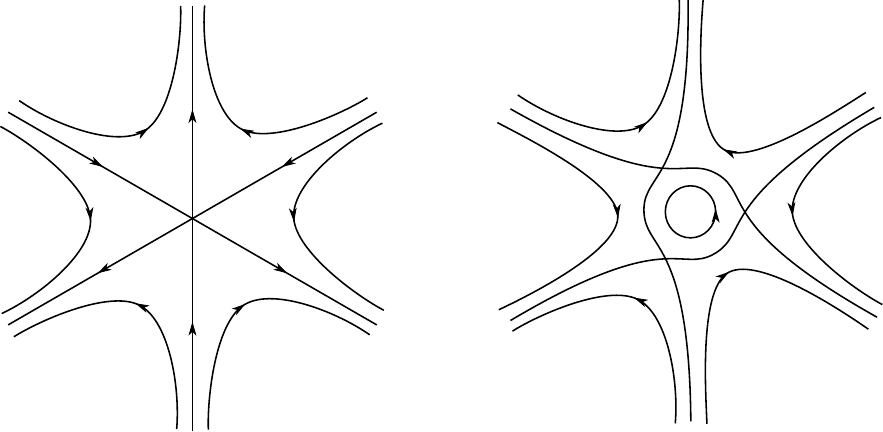}
		\caption{Perturbing a $k=3$ prong pseudo-Anosov singularity to a number of hyperbolic fixed points and an elliptic fixed point}\label{fig:perturb}
	\end{figure}
	
	\begin{figure}
		\centering
		\includegraphics[width=0.75\textwidth]{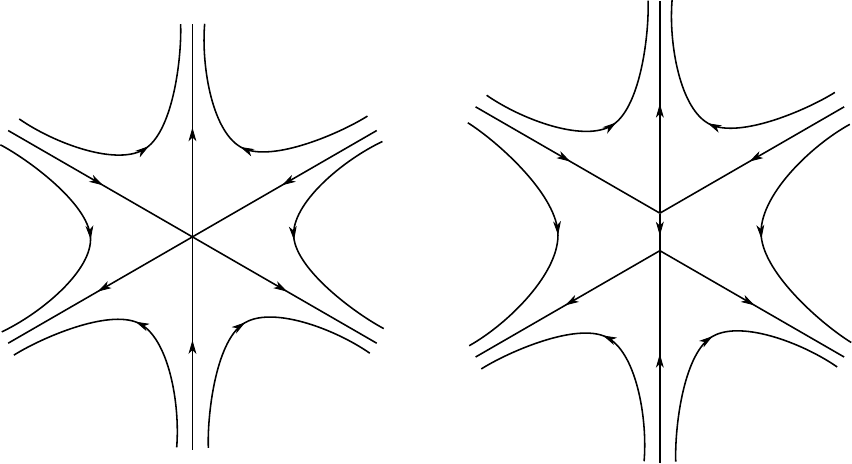}
		\caption{Consider the pseudo-Anosov diffeomorphism of the plane given by flowing along the vector field shown to the left and then reflecting in the vertical line. In this example, $\{r_1^s,r_2^s\}=\{1,2\}$ and $\{r_1^u,r_2^u\} = \{1,2\}$. The diffeomorphism can be perturbed to one with two hyperbolic fixed points shown on the left.}\label{fig:perturb_reflection}
	\end{figure}

	We now define the \emph{Ruelle zeta function} of $\varphi$ twisted by the representation $\rho$ by 
	\begin{equation}\label{eq:definition_twisted_zeta_function}
	\zeta_{\varphi,\rho}(s) = \prod_{\text{primitive closed orbits }\gamma} \zeta_{\varphi,\rho,\gamma}(s).
	\end{equation}
	By \cref{remark:nonsingular_zeta}, the contribution of regular periodic orbits is the same as in the Anosov case. We have just explained how to take into account singular orbits.
\subsection{Statement of results}

As in the Anosov case, the Euler product in \cref{eq:definition_twisted_zeta_function} converges when $\re s \gg 1$ (this follows from a bound on the number of periodic orbits that can be deduced from the symbolic representation of the flow used in \cref{subsection:symbolic_dynamics}), but we have more:

\begin{theorem}\label{theorem:continuation_zeta}
Let $\varphi = (\varphi_t)_{t \in \mathbb{R}}$ be a transitive smooth pseudo-Anosov flow on a $3$-dimensional closed manifold $M$. Let $\rho$ be a representation of the fundamental group of $M$. Then the zeta function $\zeta_{\varphi,\rho}$ has a meromorphic continuation to $\mathbb{C}$.
\end{theorem}

Adapting Sanchez-Morgado's approach \cite{morgado96}, we also prove a version of Fried's conjecture:

\begin{theorem}\label{theorem:fried_conjecture}
Let $\varphi = (\varphi_t)_{t \in \mathbb{R}}$ be a transitive smooth pseudo-Anosov flow on a $3$-dimensional closed orientable manifold $M$. Let $\rho$ be an acyclic unitary representation of the fundamental group of $M$. Assume that for each singular orbit $\gamma$, the number $1$ is not an eigenvalue of $\rho(\gamma)$. If there is no singular orbit, then assume that there is an element $g\in \pi_1(M)$ such that $1$ is not an eigenvalue of $\rho(g)$. Then 
\begin{equation*}
|\zeta_{\varphi,\rho}(0)|^{-1} = \tau_{\rho}(M),
\end{equation*}
where $\tau_{\rho}(M)$ denote the $\rho$-twisted Reidemeister torsion of $M$.
\end{theorem}

\Cref{theorem:fried_conjecture} is a particular case of a more general result (\cref{theorem:fried_conjecture_general}) that is also valid when $M$ is not orientable, but whose statement is slightly more complicated. Notice that \cref{theorem:fried_conjecture} extends previously available results \cite{morgado96,dang.guillarmou.ea.FriedConjectureSmall} in several ways. First of all, we leave the class of Anosov flows and are able to deal with certain singularities. But even in the class of Anosov flows, our result a priori has a larger scope. Indeed, our monodromy condition (there is $g\in \pi_1(M)$ such that $1$ is not an eigenvalue of $\rho(g)$) is weaker than the corresponding condition in \cite{morgado93,morgado96,dang.guillarmou.ea.FriedConjectureSmall} (there is a periodic orbit $\gamma$ such that $1$ and $\Delta_\gamma$ are not eigenvalues for $\rho(\gamma)$). Moreover, we do not require the ``no resonance at zero'' assumption from \cite{dang.guillarmou.ea.FriedConjectureSmall}.

Finally, we give a formula for the order of the first homology group of the double cover of an integral homology sphere branched over periodic orbits of an Anosov flow. Double branched covers are well studied in knot theory. Recall that the order of the first homology group of the double branched cover of $\mathbb{S}^3$ over a knot may be computed using the Alexander polynomial \cite[Corollary 9.2]{book_knot}. We give instead a formula in terms of dynamical quantities:

\begin{theorem}\label{theorem:dirichletclassnumber}
	Let $\varphi = (\varphi_t)_{t \in \mathbb{R}}$ be a transitive $C^\infty$ Anosov flow on an integer homology 3-sphere $M$. Let $\gamma_1,\dots,\gamma_n$ be a collection of closed orbits of $\varphi$. Let $\overline M$ be the double cover\footnote{See \Cref{subsection:branched_covers} for the description of $\overline{M}$.} of $M$ branched over $\bigcup_i \gamma_i$. Let $\overline \varphi$ be the pseudo-Anosov flow obtained by lifting $\varphi$ to $\overline M$. If $H_1(\overline M, \Z)$ is finite, then

	$$|H_1(\overline M,\Z)| = \frac 1 2 \lim_{s\to 0} \left|\frac{\zeta_{\overline \varphi,1}(s)} {\zeta_{\varphi,1}(s)} \right|$$
\end{theorem}
This last theorem is a topological analogue of the Dirichlet class number formula for quadratic extensions of number fields. We explain this analogy further in \cref{section:motivation_pseudo_Anosov}. When $H_1(\overline{M},\Z)$ is infinite, the quotient of zeta functions is expected to have a zero of order dictated by $b_1(\overline{M})$, but we do not prove this. Note that $H_1(\overline{M},\Z)$ is always finite when taking the branched double cover over a single orbit.

\subsection{Pseudo-Anosov flows in topology and arithmetic}\label{section:motivation_pseudo_Anosov}
For readers unfamiliar with pseudo-Anosov flows on 3-manifolds, let us explain some of the topological and arithmetic motivations for studying them. First of all, pseudo-Anosov flows are abundant on 3-manifolds. The mapping tori of most diffeomorphisms of positive genus surfaces support pseudo-Anosov flows. More generally, Gabai and Mosher explained how to build a pseudo-Anosov flow on any hyperbolic 3-manifold with $b_1>0$ \cite{gabai_mosher}. Starting from these examples, one may produce many more via surgery constructions \cite{fried.TransitiveAnosovFlows, goodman, shannon, agol_tsang}.

\Cref{theorem:fried_conjecture} additionally requires an acyclic representation satisfying a monodromy condition on the singular orbits. Let us explain how to construct some examples satisfying these conditions.

\begin{example}
Let $S$ be an orientable surface of genus $\geq 2$. Choose a pseudo-Anosov map $h:S\to S$. Let $M$ be the mapping torus of $h$ and let $\varphi$ be a suspension flow\footnote{The flow $\varphi$ is a smooth pseudo-Anosov flow provided $h$ is smooth away from its singular points, the singular points of $h$ may be described as in \Cref{definition:standard_fixed_point}, and the roof function satisfies the relevant regularity condition as in \cref{definition:roof_function}. This is true for instance if $h$ is a linear pseudo-Anosov map and the roof function is constant.} of $h$. Let $\alpha\in H^1(M,\R)$ be the cohomology class dual to the surface $S\subset M$. For any $t\in \R$, we have the one dimensional unitary representation $\rho(\gamma) = e^{it\alpha(\gamma)}$. For $i=0,1,2,3$, the first page of the Leray-Serre spectral sequence for the fibration over $S^1$ has rows of the form $$H_i(S) \xrightarrow{I-e^{it}h_*} H_i(S) \to H_i(M; \rho) \to H_i(S^1; \rho)$$
For $t$ not a multiple of $2\pi$, $H_i(S^1,\rho)=0$. For all but finitely many values of $t \mod 2\pi$, $I-e^{it} h_*$ has full rank and consequently $H_i(M;\rho)=0$. All the singular orbits have monodromy a power of $e^{it}$, so the monodromy condition is also satisfied when $t\notin 2 \pi \mathbb{Z}$.
\end{example}
\begin{example}
One can modify the previous example to get non suspension flows. Fix $t=2\pi p/q$ for some $p/q$ satisfying the conditions in the previous example. Take any nonsingular orbit $\gamma$ of $\varphi$ and perform a Fried--Goodman surgery with slope $1/N$. The representation survives the Dehn surgery whenever $N$ is a multiple of $q$. The monodromy along the nonsingular orbits is unchanged. The set of $N$ for which $H_1(M; \rho)\neq 0$ has the form $q\Z \cap X$ for some algebraic set $X$; this can be seen by writing down the maps in the twisted chain complex as a function of $N$ and noting that the vanishing of $H_1$ is equivalent the algebraic condition that a map has full rank. Note that $0\not \in X$ because $N=0$ corresponds with the trivial surgery. Therefore, $X$ is a finite set and $H_1(M;\rho)$ is almost always trivial. It is easy to check that $H_0(M;\rho)=H_3(M;\rho)=0$. A 3-manifold has trivial Euler characteristic, so $H_2(M;\rho)=0$ whenever all the other homology groups vanish. Therefore, there are infinitely many choices of $N$ yielding 3-manifolds with acyclic representations satisfying the monodromy condition.
\end{example}

There is a long history of studying dynamical determinants for pseudo-Anosov flows. The earliest work takes place in the setting of pseudo-Anosov maps \cite{birman.brinkmann.ea.PolynomialInvariantsPseudoAnosov,mcmullen.PolynomialInvariantsFibered}. More recently Landry--Minsky--Taylor and Parlak studied related constructions in the setting of pseudo-Anosov flows without perfect fits \cite{landry.minsky.ea.PolynomialInvariantVeering, parlak.TautPolynomialAlexander}. These invariants encode information about the growth rates of the number of orbits in various homology classes, just as the domain of convergence of the Euler product defining the Ruelle zeta function is determined by the growth rate of a weighted count of orbits of $\varphi$.

Another motivation for studying zeta functions of pseudo-Anosov flows is the primes and knots analogy:
\begin{center}
	\begin{tabular}{c | c}
	 number field $K$ & 3-manifold $M$ with a flow \\ 
	 prime $p$ in the ring of integers $\mathcal O_K$ & primitive closed orbit $\gamma$ \\
	 Frobenius automorphism of $\mathcal O_K/p$ & generator of $\pi_1(\gamma)$\\
	 rank of the unit group & $b_1(M)$\\
	 class number $h_K$ & $|H_1(M,\Z)_{tors}|$\\
	 Legendre symbol & linking number\\
	 Artin L-function & Ruelle zeta function\\
	 field extension ramified over $p_1,\dots,p_n$ & covering branched over $\gamma_1,\dots,\gamma_n$\\
	\end{tabular}
	\end{center}
Given the last analogy, it is natural to consider a class of flows for which we can take ramified covers over closed orbits. Pseudo-Anosov flows constitute such a class of flows. This idea was pursued by McMullen who proved a version of the Chebotarev density theorem for pseudo-Anosov flows \cite{mcmullen}. We will prove a topological analogue of the Dirichlet class number formula. The Dirichlet class number formula computes the class number of a quadratic field extension in terms of special values of Dedekind zeta functions. In the case of an imaginary quadratic extension $K=\mathbb Q[\sqrt {d}]$, $d<0$, the class number formula can be written as

\begin{equation*}
	h_K =\frac {w}{2}\lim_{s\to 0} \frac{\zeta_K(s)}{\zeta(s)},
\end{equation*}
where $\zeta_K(s)$ is the Dedekind zeta function for $K$, $\zeta(s)$ is the Riemann zeta function, and $w$ is the number of roots of unity (of any order) in $K$. $$w=\begin{cases}
	4 \qquad d=-1\\
	6 \qquad d =-3\\
	2 \qquad \textup {otherwise}
\end{cases}$$

\Cref{theorem:dirichletclassnumber} is the analogous statement for double branched covers of 3-manifolds. There, $M$ plays the role of $\mathbb Q$ and $\overline M$ plays the role of $K$. The orbits $\gamma_1,\dots,\gamma_n$ correspond to the primes which ramify in $K$. The condition that $b_1(\overline M)=0$ is analogous to the condition that the rank of the unit group is zero; this is always the case for imaginary quadratic extensions.

\subsection{Strategy of proof}

A modern approach to the study of zeta functions associated to hyperbolic flows is to relate them to spectral properties of transfer operators acting on spaces of anisotropic distributions (see \cite{giulietti.liverani.ea.AnosovFlowsDynamical,dyatlov.zworski.DynamicalZetaFunctions}, this is also the approach in \cite{dang.guillarmou.ea.FriedConjectureSmall}). We will only partly adhere to this approach, and rather follow the strategy from \cite{rugh96}: using a Markov partition (see \cref{section:Markov_partition}), we study a transfer operator associated to a hyperbolic family of open maps instead of the transfer operator for the flow itself. To study the transfer operator associated to this hyperbolic family of open maps, we use the spaces of anisotropic distributions designed by Baladi and Tsujii in \cite{baladi_tsujii_determinant}. Consequently, we get a $C^\infty$ version of the approach from \cite{rugh96} (this reference only deals with analytic flows because it relies on the anisotropic spaces from \cite{rugh92}). 

This approach using a Markov partition has at least two important advantages. First, we avoid the difficult task of constructing a space of anisotropic distributions adapted to a pseudo-Anosov flow (notice that such a construction has been made for linear pseudo-Anosov maps \cite{pseudoAnosovmap}) with good enough properties to be able to study zeta functions. Instead, we can work with spaces already present in the literature \cite{baladi_tsujii_determinant}. This slightly less intrinsic approach is also why we can remove the ``no resonance at zero'' assumption from \cite{dang.guillarmou.ea.FriedConjectureSmall}. The second advantage is that Markov partitions for Anosov and pseudo-Anosov flows in dimension 3 are nicer than for Anosov flows in higher dimensions: the rectangles of the partition do look like actual rectangles (see for instance \cite{chernov_markov}). Consequently, a well-chosen Markov partition induces a topological description of the underlying manifold that is convenient for computing Reidemeister torsion. We have a combinatorial object (a finite graph) that can be used to describe both the dynamics of the flow and the topology of the manifold, and consequently to relate these two things. This is Sanchez-Morgado's approach \cite{morgado93,morgado96}, that we adapt slightly here (which allows us for instance to weaken the monodromy condition). It is well-known that the symbolic model associated to a Markov partition produces a slight error when counting periodic orbits. But, once again, the low dimension allows us to build a Markov partition for which this error is explicit, and can thus be corrected.

\subsection{Organization of the paper}
In \cref{section:definition_pseudo_Anosov}, we introduce the class of smooth pseudo-Anosov flows we will work with. In \cref{section:stable_unstable_manifolds} and \cref{section:Markov_partition}, we establish the existence of stable and unstable manifolds and Markov partitions for smooth pseudo-Anosov flows. These arguments are more or less standard, but we are not aware of the existence of a careful treatment of this question in the setting of smooth pseudo-Anosov flows in the literature\footnote{See for instance \cite{brunella_markov} for a construction of Markov partitions for pseudo-Anosov flows. We cannot use the Markov partitions constructed this way here, because this approach is merely topological. We shall instead work with nice Markov partitions whose construction uses the smoothness assumptions on our flows.}. In \cref{section:zeta_continuation}, we relate the Ruelle zeta function to a dynamical determinant defined in terms of a Markov partition. This is where we prove \cref{theorem:continuation_zeta}. Finally, in \cref{section:rtorsion} we relate this dynamical determinant to the Reidemeister torsion and prove Theorems \ref{theorem:fried_conjecture} and \ref{theorem:dirichletclassnumber}.

\subsection*{Acknowledgement}

The authors would like to thank Semyon Dyatlov for introducing them and suggesting this collaboration. The second author would like to thank Michael Landry, Sam Taylor, and Anna Parlak for explaining their closely related work on the taut polynomial, which served as motivation for the present work. Part of this work was done while the authors were working at MIT. The first author benefits from the support of the French government “Investissements d’Avenir” program integrated to France 2030, bearing the following reference ANR-11-LABX-0020-01.

\section{Definition of smooth pseudo-Anosov flows}\label{section:definition_pseudo_Anosov}

We start by defining the class of flows we will work with. It is standard that the existence of a meromorphic continuation for the zeta function associated to an Anosov flow is deeply connected with the regularity properties of this flow. Hence, even if we will work with flows that have singularities, we need to impose some smoothness assumption in the way these singularities are described. This is why we use the term ``smooth pseudo-Anosov flow'', even though these flows may have singularities. 

To define these flows, we first introduce a local model for ``pseudo-hyperbolic orbits'' in \cref{subsection:local_model} and then use it to get a global definition in \cref{subsection:definition_pseudo_Anosov}.

\subsection{Local model}\label{subsection:local_model}

We first define a class of ``pseudo-hyperbolic fixed point'' for $2$-dimensional maps in \cref{subsubsection:pseudohyperbolic_fixed_point} and then explain how to take a suspension in order to get a model of pseudo-hyperbolic orbit in \cref{subsubsection:pseudo_hyperbolic_orbit}.

\subsubsection{Standard pseudo-hyperbolic fixed point (map case)}\label{subsubsection:pseudohyperbolic_fixed_point}

We start by recalling a common definition.

\begin{definition}\label{definition:hyperbolic_fixed_point}
Let $U$ be an open neighbourhood of $0$ in $\mathbb{R}^2$ and $\mathbf{F} : U \to \mathbb{R}^2$ be a $C^\infty$ map such that $\mathbf{F}(0) = 0$. We say that $0$ is a hyperbolic fixed point for $\mathbf{F}$ if $D \mathbf{F}(0)$ is invertible and has an eigenvalue of modulus strictly less than one and another of modulus strictly larger than one\footnote{We are not interested here in the case of merely expanding or contracting fixed points, so we exclude them from the definition.}. 
\end{definition}

It is standard \cite[Theorem 6.2.3]{katok_hasselblatt_book} that under the assumption of \cref{definition:hyperbolic_fixed_point} there are $C^\infty$ curves $W^{\textup{s}}_{\mathbf{F}}$ and $W^{\textup{u}}_{\mathbf{F}}$ (called respectively local stable and unstable manifolds of $\mathbf{F}$) that intersect transversally at zero such that $\mathbf{F}$ preserves $W^{\textup{s}}_{\mathbf{F}}$ and induces a contraction on $W^{\textup{s}}_{\mathbf{F}}$, and $\mathbf{F}^{-1}$ preserves $W^{\textup{u}}_{\mathbf{F}}$ and induces a contraction on $W^{\textup{u}}_{\mathbf{F}}$. The following definition refers to the closed upper half-plane:
\begin{equation*}
\overline{\mathbb{H}} = \set{ z \in \mathbb{C} : \im z \geq 0}.
\end{equation*}

\begin{definition}
Let $U$ be an open neighbourhood of $0$ in $\mathbb{R}^2$ and $\mathbf{F} : U \to \mathbb{R}^2$ be a $C^\infty$ map such that $\mathbf{F}(0) = 0$.  We say that $\mathbf{F}$ is a half hyperbolic fixed point bounded by the unstable direction if:
\begin{itemize}
\item $0$ is a hyperbolic fixed point for $\mathbf{F}$;
\item $\mathbf{F}(U \cap \overline{\mathbb{H}}) \subseteq \overline{\mathbb{H}}$;
\item the local unstable manifold of $\mathbf{F}$ is contained in $\mathbb{R} \times \set{0}$ and the local stable manifold of $\mathbf{F}$ is contained in $\set{0} \times \mathbb{R}$.
\end{itemize}
We define similarly a half hyperbolic fixed point bounded by the stable direction by interchanging stable and unstable manifolds in the last point. In this definition, even though $\mathbf{F}$ is defined on a neighbourhood of $0$, we are mostly interested in the restriction of $\mathbf{F}$ to the upper half-plane.
\end{definition}

After some changes of coordinates, a hyperbolic fixed point may be described by using two half hyperbolic fixed points (hence the name, see \cref{example:regular} below). We will now introduce a class of singular fixed points obtained by gluing a priori more than two half hyperbolic fixed points. To do so, let $n\geq 2$ be an integer and define the closed angular sectors:
\begin{equation*}
\mathfrak{A}_k^n = \set{ r e^{i \theta}: r \in \mathbb{R}_+, \theta \in [\pi k /n, \pi (k+2)/n]} \textup{ for } k \in \mathbb{Z} / 2n \mathbb{Z}.
\end{equation*}

\begin{figure}[h]
	\centering
	\def\svgwidth{\textwidth}
	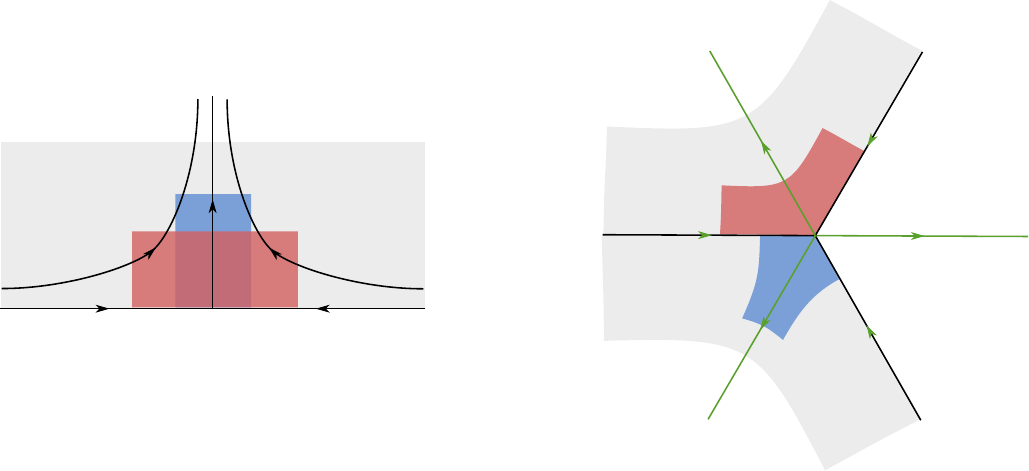
	\caption{A standard $3$-pronged hyperbolic fixed point}\label{fig:sectors}
\end{figure}

\begin{definition}\label{definition:standard_fixed_point}
Let $\phi : \mathbb{R}^2 \to \mathbb{R}^2$. We say that $\phi$ is a standard $n$-pronged hyperbolic fixed point (see \cref{fig:sectors}) if:
\begin{enumerate}[label=(\roman*)]
\item $\phi(0) = 0$;
\item $\phi$ is a homeomorphism from $\mathbb{R}^2$ to itself and induces a $C^\infty$ diffeomorphism from $\mathbb{R}^2 \setminus \set{0}$ to itself; \label{item:global_assumption}
\item there is an open neighbourhood $\mathcal{V}$ of $0$ in $\overline{\mathbb{H}}$, and for each $k \in \mathbb{Z}/ 2n \mathbb{Z}$ there is a map $\psi_k :\mathcal{V} \to \mathfrak{A}_k^n,$ such that $\psi_k(0) =0$, $\psi_k$ is a homeomorphism onto its image and $\psi_k$ restricted to $\mathcal{V} \setminus \set{0}$ is a diffeomorphism onto its image\footnote{This item does not involve $\phi$, it is just meant to introduce the $\psi_k$'s that appear in the last two items.};
\item there is an open neighbourhood $\mathcal{U}$ of $0$ in $\mathbb{R}^2$ such that for every $k \in \mathbb{Z}/2n \mathbb{Z}$, we have $\phi(\mathcal{U} \cap \mathfrak{A}_k^n) \subseteq \mathfrak{A}_{\sigma(k)}^n$ for some $\sigma(k) \in \mathbb{Z}/ 2n\mathbb{Z}$ and the map $\psi_{\sigma(k)}^{-1} \circ \phi \circ \psi_k$ defined on a neighbourhood of $0$ in $\overline{\mathbb{H}}$ extends to a half hyperbolic fixed point\footnote{In particular, $\psi_{\sigma(k)}^{-1} \circ \phi \circ \psi_k$ is smooth.}, bounded by the unstable direction if $k$ is even, by the stable direction if $k$ is odd\footnote{The parity condition on $k$ encodes the alternance between stable and unstable manifolds.}; \label{item:lambda_mu}
\item for every $k \in \mathbb{Z}/2n \mathbb{Z}$, the map $\psi_{k+1}^{-1} \circ \psi_k : \psi_k^{-1}( \psi_k(\mathcal{V}) \cap \psi_{k+1}(\mathcal{V})) \to \psi_{k+1}^{-1}( \psi_k(\mathcal{V}) \cap \psi_{k+1}(\mathcal{V}))$ is smooth. \label{item:smooth_rotation}
\end{enumerate}
\end{definition}

\begin{remark}
In \cref{definition:standard_fixed_point}, we only care about the properties of $\phi$ near $0$. However, for later convenience, we imposed some global assumptions on $\phi$ in \ref{item:global_assumption}. This is not restrictive. Indeed, if $\phi : U \to \mathbb{R}^2$ defined on an neighbourhood of $0$ in $\mathbb{R}^2$ satisfies \cref{definition:standard_fixed_point} with \ref{item:global_assumption} replaced by ``$\phi$ is a homeomorphism on its image and $\phi$ restricted to $U \setminus \set{0}$ is a diffeomorphism on its image'', then one can find $\tilde{\phi}$ satisfying \cref{definition:standard_fixed_point} as it is stated and that coincides with $\phi$ on a neighbourhood of zero.
\end{remark}

\begin{remark}\label{remark:prong_permutation}
\cref{definition:standard_fixed_point} involves a map $\sigma $ from $\mathbb{Z}/ 2n \mathbb{Z}$ to itself. It cannot be any self-map of $\mathbb{Z}/2n \mathbb{Z}$. Indeed, the fact that $\phi$ is a homeomorphism forces $\sigma$ to be a permutation. Moreover, for every $k \in \mathbb{Z}/2n \mathbb{Z}$ $0$ belongs to the closure of the interior of $\mathfrak{A}_k^n \cap \mathfrak{A}_{k+1}^n$. Consequently, $\mathfrak{A}_{\sigma(k)}^n$ and $\mathfrak{A}_{\sigma(k+1)}^n$ must have the same property, which imposes $\sigma(k+1) = \sigma(k) \pm 1$. Here, the sign $\pm 1$ has to be the same for every $k$ (the opposite would contradict the injectivity of $\sigma$). It follows that there is $\epsilon \in \set{\pm 1}$ and $a \in \mathbb{Z}/2n \mathbb{Z}$ even such that $\sigma : k \mapsto \epsilon k + a$. Notice that $\epsilon = 1$ if and only if $\phi$ preserves orientation.
\end{remark}

\begin{example}\label{example:regular}
Let $U$ be an open neighbourhood of $0$ in $\mathbb{R}^2$ and $\phi : U \to \mathbb{R}^2$ be a $C^\infty$ map with a hyperbolic fixed point at $0$. Then, after a change of variables near $0$, we may assume that the local unstable manifold of $\phi$ is contained in $\mathbb{R} \times \set{0}$ and that the local stable manifold of $\phi$ is contained in $\set{0} \times \mathbb{R}$. Define $\sigma : \mathbb{Z}/4 \mathbb{Z} \to \mathbb{Z}/ 4 \mathbb{Z}$ by
\begin{itemize}
\item $\sigma(0) = 0$ and $\sigma(2) = 2$ if $\phi^{-1}$ is orientation-preserving on the local unstable manifold of $\phi$, and $\sigma(0) = 2, \sigma(2) = 0$ otherwise;
\item  $\sigma(1) = 1$ and $\sigma(3) = 3$ if $\phi$ is orientation preserving on its local stable manifold, $\sigma(1) = 3, \sigma(3) = 1$ otherwise.
\end{itemize}
With this definition, for every $k \in \mathbb{Z}/ 4 \mathbb{Z}$, the map $\phi$ sends a neighbourhood of $0$ in $\mathfrak{A}_k^2$ to a neighbourhood of $0$ in $\mathfrak{A}_{\sigma(k)}^2$. 

Now, for $k \in \mathbb{Z}/ 4 \mathbb{Z}$, let $\psi_k$ be the rotation of angle $k \pi /2$. With this definition, we find that for every $k \in \mathbb{Z}/4 \mathbb{Z}$, the composition $\psi_{\sigma(k)} \circ \phi \circ \psi_{k}^{-1}$ defines a half hyperbolic fixed point (bounded by the stable direction if $k$ is odd, by the unstable direction if $k$ is even). Consequently, $\phi$ coincides near $0$ with a standard $2$-pronged hyperbolic fixed point.
\end{example}

\begin{example}\label{example:ramified}
Let $\phi$ be as in \cref{example:regular}. As above, make a change of variables to ensure that the local unstable manifold of $\phi$ is contained in $\mathbb{R} \times \set{0}$ and that the local stable manifold of $\phi$ is contained in $\set{0} \times \mathbb{R}$. We will reuse the notations from \cref{example:regular}. Let $m \geq 2$. Identifying $\mathbb{R}^2$ with $\mathbb{C}$, the map $p : z \mapsto z^m$ defines a ramified cover of $\mathbb{R}^2$ by itself. Let then $\tilde{\phi}$ be a continuous lift of $\phi$ defined on a neighbourhood of zero (that is we have $p \circ \tilde{\phi} = \phi \circ p$). 

Let $b: \mathbb{Z}/4m \mathbb{Z} \to \mathbb{Z} / 4 \mathbb{Z}$ be the natural projection. For $k \in \mathbb{Z}/ 4m \mathbb{Z}$, let $\psi_k : \overline{\mathbb{H}} \to \mathfrak{A}_k^{2m}$ be the composition of a rotation of angle $k \pi/2$ and the inverse of the homeomorphism induced by $p$ between $\mathfrak{A}_k^{2m}$ and $\mathfrak{A}_{b(k)}^{2}$. 

For every $k \in \mathbb{Z}/ 4 m \mathbb{Z}$, the map $\tilde{\phi}$ sends a neighbourhood of zero in $\mathfrak{A}_k^{2m}$ to a neighbourhood of zero in $\mathfrak{A}_{\tilde{\sigma}(k)}^{2m}$ for some $\tilde{\sigma}(k)$ such that $b(\tilde{\sigma}(k)) = \sigma(b(k))$. Near $0$ in $\overline{\mathbb{H}}$ the map $\tilde{\psi}_{\tilde{\sigma}(k)}^{-1} \circ \tilde{\phi} \circ \tilde{\psi}_k$ and $\psi_{\sigma(b(k))}^{-1} \circ \phi \circ \psi_{b(k)}$ coincides. Hence, $\tilde{\psi}_{\tilde{\sigma}(k)}^{-1} \circ \tilde{\phi} \circ \tilde{\psi}_k$ extends to a half hyperbolic fixed point according to \cref{example:regular}. It follows that $\tilde{\phi}$ coincides near zero with a standard $2m$-pronged hyperbolic fixed point near zero.
\end{example}

\begin{example}\label{example:fixed_point}
Let us mention two further examples of standard $n$-pronged hyperbolic fixed points. The first one is used to define smooth pseudo-Anosov flows in \cite{agol_tsang}. Fix some $\lambda > 1$. Notice that, for $k \in \mathbb{Z}/ 2 n \mathbb{Z}$ even, identifying $\mathbb{R}^2$ and $\mathbb{C}$, the map $z \mapsto z^{\frac{n}{2}}$ induces a homeomorphism between $\mathfrak{A}_k^n$ and the lower or the upper half-plane. We define $\phi$ on $\mathfrak{A}_k^n$ by imposing that 
\begin{equation*}
\phi(z) \in \mathfrak{A}_k^n, \phi(z)^{\frac{n}{2}} = \begin{bmatrix}
\lambda & 0 \\ 0 & \lambda^{-1}
\end{bmatrix} z^{\frac{n}{2}}
\end{equation*}
for $z \in \mathfrak{A}_k^n$. One can then check that $\phi$ is well-defined (there is no ambiguity in the definition of $\phi$ for points on the boundary of an angular sector) and satisfies \cref{definition:standard_fixed_point}. One can then compose $\phi$ on the left by a rotation of angle a multiple of $2 \pi / n$ to get more examples.

The second class of examples is obtained by the same procedure as above but uses the map $z \mapsto z^{\frac{n}{2}} / |z^{\frac{n}{2}-1}|$ instead of $z \mapsto z^{\frac{n}{2}}$.
\end{example}

For the class of $n$-pronged fixed point introduced in \cref{definition:standard_fixed_point}, the role of the local unstable manifold is played by a neighbourhood of $0$ in the set 
\begin{equation*}
\mathfrak{L}^n = \set{ r e^{\frac{2i k \pi}{n}}: r \in [0,\infty), k = 0,\dots, n - 1}.
\end{equation*}
Indeed:

\begin{lemma}\label{lemma:ramified_unstable}
Let $\phi$ be a standard $n$-pronged hyperbolic fixed point and $U$ be a neighbourhood of $0$ in $\mathbb{R}^2$. Then there is a neighbourhood $V \subseteq U$ of $0$ in $\mathbb{R}^2$ such that:
\begin{itemize}
\item $\phi^{-1}(\mathfrak{L}^n \cap V) \subseteq \mathfrak{L}^n \cap V$;
\item if $x \in \mathfrak{L}^n \cap V$ then $\phi^{-m}(x) \underset{m \to + \infty}{\to} 0$;
\item if $x \in \mathbb{R}^2$ is such that $\phi^{-m}(x) \in V$ for every $m \geq 0$ then $x \in \mathfrak{L}^n$.
\end{itemize}
\end{lemma}

\begin{proof}
Applying the Hartman--Grobman Theorem \cite[Theorem 6.3.1]{katok_hasselblatt_book} and using that all the germs of contracting fixed points are $C^0$-conjugated in dimension $1$, we may construct homeomorphisms $\tilde{\psi}_k$ (from a neighbourhood of zero in $\overline{\mathbb{H}}$ to a neighbourhood of zero in $\mathfrak{A}_k^n$) such that for $(x_1,x_2)$ near $0$ in $\overline{\mathbb{H}}$ and $k \in \mathbb{Z}/ 2n \mathbb{Z}$ even, we have
\begin{equation*}
\tilde{\psi}_{\sigma(k)}^{-1} \circ \phi \circ \tilde{\psi}_k (x_1,x_2) = (2 x_1,x_2/2).
\end{equation*}
Then, for $\delta > 0$ small enough, we can set
\begin{equation*}
V = \bigcup_{k \in \mathbb{Z}/ 2n \mathbb{Z} \textup{ even}} \tilde{\psi}_k([ - \delta,\delta] \times [0,\delta]).
\end{equation*}
\end{proof}

Similarly, one can prove that the role of the local stable manifold of $\phi$ is played by a neighbourhood of zero in 
\begin{equation*}
\widetilde{\mathfrak{L}}^n = \set{ r e^{\frac{(2 k+1)i \pi}{n}}: r \in [0,\infty), k = 0,\dots, n - 1}.
\end{equation*}

\subsubsection{Standard pseudo-hyperbolic orbit (flow case)}\label{subsubsection:pseudo_hyperbolic_orbit}

We explain now how to take a suspension of the maps from \cref{definition:standard_fixed_point} in order to get a model of singular hyperbolic orbit.

\begin{definition}\label{definition:roof_function}
Let $n \geq 2$. Let $\phi$ be a standard $n$-pronged hyperbolic fixed point. We say that $r : \mathbb{R}^2 \to (0,\infty)$ is an admissible roof function for $\phi$ if:
\begin{enumerate}[label=(\roman*)]
\item $r$ is continuous;
\item $r$ is bounded below by a positive number, and is bounded above;
\item $r$ is smooth on $\mathbb{R}^2 \setminus \set{0}$;
\item using the notation from \cref{definition:standard_fixed_point}, for every $k \in \mathbb{Z}/2n \mathbb{Z}$, the map $r \circ \psi_k$ is $C^\infty$ on $\mathcal{V}$.
\end{enumerate}
\end{definition}

\begin{definition}\label{definition:suspension}
Let $n \geq 2$. Let $\phi$ be a standard $n$-pronged hyperbolic fixed point and $r$ an admissible roof function for $\phi$. Define the suspension manifold of $\mathbb{R}^2$ over $\phi$ with roof function $r$:
\begin{equation*}
N_{\phi,r} = (\mathbb{R}^2 \times \mathbb{R}) / \mathbb{Z},
\end{equation*}
where the action of $\mathbb{Z}$ on $\mathbb{R}^2 \times \mathbb{R}$ is given by
\begin{equation*}
n \cdot (x,t) = (\phi^n (x), t - S_n r (x)),
\end{equation*}
where
\begin{equation*}
S_n r (x) = \begin{cases} \sum_{k = 0}^{n-1} r \circ \phi^k(x) & \textup{ if } n \geq 0, \\ - \sum_{k = 1}^{-n} r \circ \phi^{-k} (x) & \textup{ if } n \leq -1.\end{cases}
\end{equation*}
We let $\Phi^{\phi,r} = (\Phi_t^{\phi,r})_{t \in \mathbb{R}}$ denote the associated suspension flow, i.e. for fixed $t \in \mathbb{R}$, the map $\Phi_t^{\phi,r}$ on $N_{\phi,r}$ factorizes $(x,t') \mapsto (x,t'+t)$. We will denote by $\pi_{\phi,r}$ the canonical projection $\mathbb{R}^2 \times \mathbb{R} \to N_{\phi,r}$ and by $\gamma_{\phi,r}$ the image by $\pi_{\phi,r}$ of $\set{0} \times \mathbb{R}$ (which is a periodic orbit for $\Phi^{\phi,r}$).
\end{definition}

\begin{remark}\label{remark:proper}
Since $r$ is bounded below, the action of $\mathbb{Z}$ on $\mathbb{R}^2 \times \mathbb{R}$ in \cref{definition:suspension} is free and  properly continuous. It follows that the quotient space $N_{\phi,r}$ is a topological manifold (when endowed with the quotient topology). Moreover, since $\phi$ induces a diffeomorphism from $\mathbb{R}^2 \setminus \set{0}$ to itself, we find that $N_{\phi,r}^* \coloneqq N_{\phi,r} \setminus \gamma_{\phi,r}$ has a structure of smooth manifold (which makes the restriction of $\pi_{\phi,r}$ to $ (\mathbb{R}^2 \setminus \set{0})\times \mathbb{R}$ a submersion).
\end{remark}

\begin{definition}\label{definition:unstable_manifolds_model}
Let $n \geq 2$. Let $\phi$ be a standard $n$-pronged hyperbolic fixed point and $r$ be an admissible roof function for $\phi$. Let then $V$ be a small neighbourhood of zero as in \cref{lemma:ramified_unstable} and set $\widetilde{V} = \set{(x,t) \in V \times \mathbb{R} : 0 \leq t < r(x)}$. We define then the local unstable manifold of $\gamma_{\phi,r}$ as
\begin{equation*}
W^{\textup{lu}}_{\phi,r} \coloneqq \pi_{\phi,r}( (\mathfrak{L}^n \times \mathbb{R}) \cap \widetilde{V} ) \subseteq N_{\phi,r}.
\end{equation*}
Beware that the set $W^{\textup{lu}}_{\phi,r}$ is not a manifold in general. This definition of $W^{\textup{lu}}_{\phi,r}$ makes it depend on some choices, but its germ near $\gamma_{\phi,r}$ does not, and that is what we care about.

The stable manifold $W^{\textup{ls}}_{\phi,r}$ of $\gamma_{\phi,r}$ is defined similarly.
\end{definition}

The flows from \cref{definition:suspension} will be used to model singular orbits for the class of smooth pseudo-Anosov flows that we introduce in \cref{subsection:definition_pseudo_Anosov} below. In preparation for future proofs, we list here some properties of the class of flows from \cref{definition:suspension}. We start with a direct consequence of \cref{lemma:ramified_unstable}.

\begin{lemma}\label{lemma:actually_unstable}
Let $n \geq 2$. Let $\phi$ be a standard $n$-pronged hyperbolic fixed point and $r$ be an admissible roof function for $\phi$. For every neighbourhood $U$ of the $\gamma_{\phi,r}$ in $N_{\phi,r}$, there is a neighbourhood $V$ of $\gamma_{\phi,r}$ in $N_{\phi,r}$ such that:
\begin{itemize}
\item for every $t \geq 0$, we have $\Phi^{\phi,r}_{-t}(W^{\textup{lu}}_{\phi,r} \cap V) \subseteq W^{\textup{lu}}_{\phi,r} \cap V$;
\item if $x \in W^{\textup{lu}}_{\phi,r}$ then $\Phi_{-t}^{\phi,r}(x)$ converges to $\gamma_{\phi,r}$ as $t$ goes to $+ \infty$;
\item if $x \in N_{\phi,r}$ is such that $\Phi_{-t}^{\phi,r}(x) \in V$ for every $t \geq 0$ then $x \in W^{\textup{lu}}_{\phi,r}$.
\end{itemize}
\end{lemma}

We will need the following technical facts for the proof of \cref{lemma:knot}, a crucial step for the description of the local unstable manifolds of a smooth pseudo-Anosov flow (\cref{proposition:unstable_manifolds}). Lemma \ref{lemma:unstable_chambers} may be interpreted as follows: near an orbit $\gamma_{\phi,r}$ as in Definition \ref{definition:suspension}, two points close to each other, but in different sectors bounded by unstable manifolds, will eventually separate under the action of the flow in negative time, unless they are both on the unstable manifold.

\begin{lemma}\label{lemma:unstable_chambers}
Let n $\geq 2$. Let $\phi$ be a standard $n$-pronged hyperbolic fixed point and $r$ an admissible roof function for $\phi$. Let $U$ be a neighbourhood of $\gamma_{\phi,r}$ in $N_{\phi,r}$ and $C$ be an upper bound for $r$. Let $c > 0$ be such that $c < \inf r$. Let $d$ be any distance inducing the topology of $N_{\phi,r}$. There are $\epsilon > 0$ and a neighbourhood $V \subseteq U$ of $\gamma_{\phi,r}$ such that for every $x,y \in V$ if there are $-10C \leq a < b \leq 10C$ such that $b - a \leq c $ and $k,k' \in \mathbb{Z}/ 2n \mathbb{Z}$ even such that $k \neq k'$, $x \in \pi_{\phi,r} (\mathfrak{A}_{k}^n \times [a,b])$ and $y \in \pi_{\phi,r} (\mathfrak{A}_{k'}^n \times [a,b])$ and $d(x,y) < \epsilon$ then
\begin{itemize}
\item there is $t_0 \geq 0 $ such that $d(\Phi^{\phi,r}_{-t_0}(x),\Phi^{\phi,r}_{-t_0}(y)) \geq \epsilon$ and $\Phi^{\phi,r}_{-t} (x) \in U$ and $ \Phi^{\phi,r}_{-t} (y) \in U$ for every $t \in [0,t_0]$
\end{itemize}
or
\begin{itemize}
\item $x$ and $y$ belong to $W_{\phi,r}^{\textup{lu}}$.
\end{itemize}
\end{lemma}

\begin{proof}
As in the proof of \cref{lemma:ramified_unstable}, apply the Hartman--Grobman Theorem \cite[Theorem 6.3.1]{katok_hasselblatt_book} to construct homeomorphisms $\tilde{\psi}_k$ (from a neighbourhood of zero in $\overline{\mathbb{H}}$ to a neighbourhood of zero in $\mathfrak{A}_k^n$) such that for $(x_1,x_2)$ near $0$ in $\overline{\mathbb{H}}$ and $k \in \mathbb{Z}/ 2n \mathbb{Z}$ even, we have
\begin{equation*}
\tilde{\psi}_{\sigma(k)}^{-1} \circ \phi \circ \tilde{\psi}_k (x_1,x_2) = (2 x_1,x_2/2).
\end{equation*}

Write then $x = \pi_{\phi,r} (z,\tau)$ and $y = \pi_{\phi,r}(z',\tau')$ with $z \in \mathfrak{A}_k^n, z' \in \mathfrak{A}_{k'}^n$ and $\tau,\tau' \in [a,b]$. Write $z = \psi_k(x_1,x_2)$ and $z' = \psi_{k'}(y_1,y_2)$. Since $\pi_{\phi,r |\mathbb{R}^2 \times [a,b]} $ is a homeomorphism on its image, we find that the distance between $z$ and $z'$ is bounded by some quantity that goes to $0$ as $\epsilon$ goes to $0$ (this is uniform in $a$ and $b$ because they stay within a bounded set and $b-a \leq c < \inf r$). Hence, $z$ and $z'$ must be close to the boundary of $\mathfrak{A}_{k}^n$ and $\mathfrak{A}_{k'}^n$ respectively. Thus, we find that $x_2$ and $y_2$ are bounded by some quantity that goes to $0$ with $\epsilon$.

Take $\delta > 0$ very small, and define
\begin{equation*}
V_0 = \bigcup_{\substack{p \in \mathbb{Z}/ 2n \mathbb{Z}\\ p \textup{ even}}} \tilde{\psi}_{p}([ - \delta,\delta]\times[0,\delta]).
\end{equation*}
Let then $V$ be the neighbourhood of the singular orbits of $\Phi^{\phi,r}$ given by $V = \pi_{\phi,r}( V_0 \times [-C,2C])$. By taking $\delta$ small enough, we ensure that $V$ is contained in the interior of $U$. Introduce the times
\begin{equation*}
t_0 = \inf\set{t \geq 0 : \forall t' \in [0,t], d(\Phi^{\phi,r}_{-t'} x , \Phi^{\phi,r}_{-t'} y) < \epsilon} \in (0,+\infty]
\end{equation*}
and
\begin{equation*}
t_1 = \inf\set{ t \geq 0 : \forall t' \in [0,t], \Phi^{\phi,r}_{-t'} (x) \in V \textup{ and } \Phi^{\phi,r}_{-t'} (y) \in V} \in (0,+\infty].
\end{equation*}
Let then $t_2 = \min (t_0,t_1)$.

For $t \in [0,t_2)$, let $m(t)$ be the smallest integer such that $(\phi^{-m(t)} z, \tau + t - S_{-m(t)} r(z)) \in V_0 \times [- C, 2C]$. The integer $m(t)$ exists because $\Phi^{\phi,r}_{-t} (x) \in V$ and $r$ is bounded from below. Since $d(\Phi^{\phi,r}_{-t} x, \Phi^{\phi,r}_{-t} y) \leq \epsilon$, we find that there is $\tilde{m}(t) \in \mathbb{Z}$ such that the distance between $(\phi^{-m(t)} z, \tau + t - S_{-m(t)} r(z))$ and $(\phi^{-\tilde{m}(t)} z', \tau' + t - S_{-\tilde{m}(t)} r(z'))$ is bounded by some quantity that goes to $0$ with $\epsilon$. We can consequently make this quantity way smaller than the infimum of $r$, which ensures, using also that the difference between $\tau$ and $\tau'$ goes to $0$ with $\epsilon$, that $m(t) = \tilde{m}(t)$ for $t \in [0,t_2)$. Thus, for every $t \in [0,t_2)$, the distance between $\phi^{-m(t)} (z)$ and $\phi^{-m(t)} (z')$ is bounded by some quantity that goes to $0$ with $\epsilon$. Notice that
\begin{equation*}
\phi^{-m(t)} z = \psi_{\sigma^{-m(t)}(k)}(2^{-m(t)} x_1, 2^{m(t)} x_2)
\end{equation*}
and
\begin{equation*}
\phi^{-m(t)} z' = \psi_{\sigma^{-m(t)}(k')}(2^{-m(t)} y_1, 2^{m(t)} y_2).
\end{equation*}

As above, we see that $\phi^{-m(t)} (z)$ and $\phi^{-m(t)} (z')$ must be close to the boundaries of $\mathfrak{A}_{\sigma^{-m(t)}(k)}^n$ and $\mathfrak{A}_{\sigma^{-m(t)}(k')}^n$ respectively, which proves that $2^{m(t)} x_2$ and $2^{m(t)} y_2$ are bounded for $t \in [0,t_2)$ by some quantity that goes to $0$ with $\epsilon$. Since $r$ is bounded, if $t_2 = + \infty$, we must have $m(t) \underset{t \to + \infty}{\to} + \infty$, and thus $x_2 = y_2 = 0$, that is $x$ and $y$ belong to $W^{\textup{lu}}_{\phi,r}$. 

We need to deal now with the case $t_2 < + \infty$. Let us prove that in that case, we must have $t_0 < t_1$ (which ends the proof of the lemma, we are in the first case). Notice that $2^{-m(t)} |x_1|$ and $2^{-m(t)} |y_1|$ are decreasing functions of $t$ on $[0,t_2)$. Hence, they are bounded on $[0,t_2)$ by their values at $0$, which are less than $\delta$. Thus, if $\Phi^{\phi,r}_{-t} (x)$ or $\Phi^{\phi,r}_{-t} (y)$ is to leave $V$ when $t$ crosses $t_2$, then $2^{m(t)} x_2$ or $2^{m(t)} y_2$ must become larger than some positive number proportional to $\delta$ as $t$ approaches $t_2$. But these two quantities are bounded by some quantity that goes to $0$ with $\epsilon$. Hence, by making this quantity way smaller than $\delta$ we make $t_0 > t_1$ impossible.
\end{proof}

\subsection{Definition of smooth pseudo-Anosov flows.}\label{subsection:definition_pseudo_Anosov}

We will first introduce a class of flows (\cref{definition:admissible_singularities}) with singularities as in the local model we just introduced (see \cref{definition:suspension}). After discussing some basic properties of these flows, we add a global hyperbolicity assumption to define the class of smooth pseudo-Anosov flows (\cref{definition:pseudo_Anosov}) that is the object of this paper.

\begin{definition}\label{definition:admissible_singularities}
Let $M$ be a $3$-dimensional closed topological manifold and $\varphi = (\varphi_t)_{t \in \mathbb{R}}$ be a flow on $M$. We say that $(\varphi_t)_{t \in \mathbb{R}}$ is a smooth flow with pseudo-hyperbolic singularities if it satisfies the following properties:
\begin{enumerate}[label=(\roman*)]
\item $(\varphi_t)_{t \in \mathbb{R}}$ is a continuous flow with no fixed points;
\item there are periodic orbits $\gamma_1,\dots,\gamma_N$ for $(\varphi_t)_{t \in \mathbb{R}}$ and a structure of smooth manifold on $\widetilde{M} \coloneqq M \setminus \bigcup_{j = 1}^N \gamma_j$ such that $(\varphi_t)_{t \in \mathbb{R}}$ is $C^\infty$ on $\widetilde{M}$;
\item for each $j \in \set{1,\dots,N}$, there is $n_j \geq 2$, a standard $n_j$-pronged hyperbolic fixed point $\phi_j$, an admissible roof function $r_j$ for $\phi_j$ and a homeomorphism $f_j: V_j \to U_j$ from a neighbourhood of $\gamma_j$ in $M$ to a neighbourhood of $\gamma_{\phi_j,r_j}$ in $N_{\phi_j,r_j}$ that induces a diffeomorphism from $V_j \setminus \gamma_j \subseteq \widetilde{M}$ to $U_j \cap N_{\phi_j,r_j}^*$;
\item for each $j \in \set{1,\dots,N}$, if $x \in V_j$ and $t \in \mathbb{R}$ then
\begin{equation*}
\set{\varphi_{t'}(x) : t' \in [0,t]} \subseteq V_j \textup{ if and only if } \set{\Phi^{\phi_j,r_j}_{t'}(f_j(x)) : t' \in [0,t]} \subseteq U_j,
\end{equation*}
in which case $f_j (\varphi_t(x)) = \Phi^{\phi_j,r_j}_t(f_j(x))$.
\end{enumerate}
\end{definition}

From now on, we fix a flow $\varphi = (\varphi_t)_{t \in \mathbb{R}}$ with pseudo-hyperbolic singularities on a manifold $M$ and use the notations from \cref{definition:pseudo_Anosov}. In particular, the manifold $\widetilde{M}$ will always be endowed with the smooth structure from this definition. Let $j \in \set{1,\dots,N}$. We will introduce the dependence on $j$ in the notation from \cref{definition:standard_fixed_point} in the following way:
\begin{itemize}
\item the permutation $\sigma$ will be called $\sigma_j$;
\item the homeomorphisms $\psi_k,k \in \mathbb{Z}/2n_j \mathbb{Z},$ will be called $\psi_k^{(j)}, k \in \mathbb{Z}/2 n_j \mathbb{Z}$;
\item by choosing the neighbourhoods $\mathcal{U}$ and $\mathcal{V}$ of $0$  small enough, we may choose them independent of $j$. 
\end{itemize}
We introduce the constant
\begin{equation*}
r_{\min} = \inf_{j = 1,\dots,N} \inf_{x \in \mathbb{R}^2} r_j(x) > 0 \textup{ and } r_{\max} = \sup_{j = 1,\dots,N} \sup_{x \in \mathbb{R}^2} r_j(x) <+ \infty.
\end{equation*}
By taking $\mathcal{U}$ small enough, we may assume that
\begin{equation*}
\pi_{\phi_j,r_j}(\mathcal{U} \times [ \tau - \frac{r_{\min}}{3}, \tau + \frac{r_{\min}}{3}]) \subseteq U_j,
\end{equation*}
for every $j \in \set{1,\dots,N}$ and $\tau \in [-r_{\max}, r_{\max}]$. Up to taking $\mathcal{V}$ smaller, we assume that $\psi_k^{(j)}(\mathcal{V}) \subseteq \mathcal{U}$ for every $j \in \set{1,\dots,N}$ and $k \in \mathbb{Z} / 2 n_j \mathbb{Z}$.

\begin{definition}\label{definition:unstable_singular_orbits}
Let $j \in \set{1,\dots,N}$. Let us assume, as we may, that the neighbourhood $V$ in \cref{definition:unstable_manifolds_model} has been chosen so that $W_{\phi_j,r_j}^{\textup{lu}} \subseteq U_j$. Then, we define the local unstable and stable manifolds of $\gamma_j$ as
\begin{equation*}
W_j^{\textup{lu}} \coloneqq f_j^{-1}(W^{\textup{lu}}_{\phi_j,r_j}) \textup{ and } W_j^{\textup{ls}} \coloneqq f_j^{-1}(W_{\phi_j,r_j}^{\textup{ls}}).
\end{equation*}
Let us also define the full unstable and stable manifolds as
\begin{equation*}
W_j^{\textup{u}} \coloneqq \bigcup_{t \geq 0} \varphi_t(W_j^{\textup{lu}}) \textup{ and } W_j^{\textup{s}} \coloneqq \bigcup_{t \geq 0} \varphi_t(W_j^{\textup{ls}}). 
\end{equation*}
We will also use the following notation
\begin{equation*}
W_{\textup{sing}}^{\textup{lu}} \coloneqq \bigcup_{j = 1}^N W_j^{\textup{lu}}, W_{\textup{sing}}^{\textup{ls}} \coloneqq \bigcup_{j = 1}^N W_j^{\textup{ls}}, W_{\textup{sing}}^{\textup{u}} \coloneqq \bigcup_{j = 1}^N W_j^{\textup{u}} \textup{ and } W_{\textup{sing}}^{\textup{s}} \coloneqq \bigcup_{j = 1}^N W_j^{\textup{s}}.  
\end{equation*}
\end{definition}

We study in \cref{section:local_unstable_manifold} stable and unstable manifolds for the class of smooth pseudo-Anosov flows that we are about to define. However, since the singular orbits are described by a very specific model, it was possible to define their stable and unstable manifolds already in a relevant way, as expressed by the following result.

\begin{lemma}\label{lemma:true_local_weak_unstable}
Let $x \in M$ and $j \in \set{1, \dots ,N}$. Then $x \in W_j^{\textup{u}}$ if and only if $\varphi_{-t} (x)$ converges to $\gamma_j$ as $t$ goes to $+ \infty$.
\end{lemma}

\begin{proof}
If $x \in W_j^{\textup{u}}$, then by definition there is $t' \geq 0$ such that $\varphi_{-t'}(x) \in W_j^{\textup{lu}} = f_j^{-1}(W^{\textup{lu}}_{\phi_j,r_j})$. Since $W^{\textup{lu}}_{\phi_j,r_j}$ is backward invariant by $\Phi^{\phi_j,r_j}$, it follows from \cref{lemma:actually_unstable} that $\varphi_{-t} (x)$ converges to $\gamma_j$ as $t$ goes to $+ \infty$.

Reciprocally, assume that $\varphi_{-t} (x)$ converges to $\gamma_j$ as $t$ goes to $+ \infty$. Then, there is $t_0 \geq 0$ such that for every $t \geq t_0$, we have $\varphi_{-t}(x) \in f_j^{-1}(V)$, where $V$ is as in \cref{lemma:actually_unstable}. Hence, $f_j(\varphi_{-t_0} x) \in W_{\textup{lu}}^{r_j,\phi_j}$, and thus $x \in W_j^{\textup{u}}$.
\end{proof}

This result suggests to introduce the following notation:

\begin{definition}\label{definition:full_manifolds_periodic}
Let $\gamma$ be a periodic orbit for $\varphi$. We define the stable and unstable manifolds of $\gamma$ as
\begin{equation*}
W^{\textup{s}}(\gamma) = \set{ x \in M : d(\varphi_t(x), \gamma) \underset{t \to + \infty}{\to} 0}
\end{equation*}
and
\begin{equation*}
W^{\textup{u}}(\gamma) = \set{ x \in M : d(\varphi_t(x), \gamma) \underset{t \to - \infty}{\to} 0}.
\end{equation*}
\end{definition}

We will get more information about these sets in \cref{section:local_unstable_manifold} below. Our analysis will rely in an essential way on the existence of specific (a priori non-smooth) coordinates near the singular orbits of $\varphi$ in which the action of the flow has a particularly neat form (in particular, it is smooth).

\begin{definition}
Let $j \in \set{1,\dots,N},k \in \mathbb{Z}/ 2 n_j \mathbb{Z}$ and $\tau \in [-r_{\max},r_{\max}]$. Define the map
\begin{equation*}
\begin{array}{ccccc}
\Upsilon_{j,k,\tau} &:&  \mathcal{V} \times ] - \frac{r_{\min}}{3}, \frac{r_{\min}}{3} [ &\to & V_j \\
 & & (x,t) & \mapsto & f_j^{-1} \pi_{\phi_j,r_j} (\psi_k^{(j)}(x), \tau + t).
\end{array}
\end{equation*}
We call $\Upsilon_{j,k,\tau}$ a ``half-space parametrization''.
\end{definition}

\begin{remark}\label{remark:flow_in_half_space_chart}
Let $j \in \set{1,\dots,N}$ and $x$ be a point on $\gamma_j$. Choose $k \in \mathbb{Z}/ 2 n_j \mathbb{Z}$ and $\tau \in [-r_{\max},r_{\max}]$ such that $\Upsilon_{j,k,\tau}(0) = x$. Let $T_0 \in \mathbb{R}$. Let $m \in \mathbb{Z}$ be such that 
\begin{equation*}
\tilde{\tau} \coloneqq \tau + T_0 - m r_j(0) \in [-r_{\max},r_{\max}]
\end{equation*}
and notice that $\varphi_{T_0} (x) = \Upsilon_{j, \sigma_j^{m}(k) , \tilde{\tau}}(0)$.

Assume for instance that $k$ is even and notice that the composition of maps $\smash{(\psi_{\sigma_j^m(k)}^{(j)})^{-1} \circ \phi^m_j \circ \psi_k^{(j)}}$ is well-defined on a neighbourhood of zero. If $m \neq 0$, it is the restriction of a half hyperbolic fixed point $\mathbf{F}$ (bounded by the unstable direction if $m > 0$, by the stable  direction if $m< 0$). Indeed, it follows from \ref{item:lambda_mu} in \cref{definition:standard_fixed_point} that $\mathbf{F}$ is the product of $|m|$ half-hyperbolic fixed point with the same stable and unstable manifolds near zero. When $m = 0$, let just $\mathbf{F}$ be the identity map.

Consequently, for $x$ near $0$ in $\mathcal{V}$, we have
\begin{equation*}
\begin{split}
\Phi^{\phi_j,r_j}_{T_0} \pi_{\phi_j,r_j} (\psi_k^{(j)} (x), \tau) & = \pi_{\phi_j,r_j} ( \psi_k^{(j)} (x), \tau + T_0) \\
     & = \pi_{\phi_j,r_j}(\phi_j^m \circ \psi_k^{(j)}(x), \tau + T_0 - S_m r_j \circ \psi_k^{(j)} (x)) \\
     & = \pi_{\phi_j,r_j}(\psi_{\sigma_j^m(k)}^{(j)} \circ \mathbf{F} (x), \tau + T_0 - S_m r_j \circ \psi_k^{(j)} (x)).
\end{split}
\end{equation*}
Finally, we find that for $(x,t)$ close enough to $0$ in $\mathcal{V} \times ] - \frac{r_{\min}}{3}, \frac{r_{\min}}{3}[$, we have
\begin{equation*}
\varphi_{T_0}(\Upsilon_{j,k,\tau} (x,t)) = \Upsilon_{j,\sigma_j^m(k),\tilde{\tau}} (
\mathbf{F} (x), R(x) + t),
\end{equation*}
where
\begin{equation*}
R(x) = \tau + T_0 - \tilde{\tau} - S_m r_j \circ \psi_k^{(j)} (x).
\end{equation*}
Noticing that
\begin{equation*}
S_m r_j  \circ \psi_k^{(j)} (x) = \begin{cases} \sum_{p = 0}^{m-1} r_j \circ \psi_{\sigma_j^p(k)}^{(j)} \circ (\psi_{\sigma_j^p(k)}^{(j)})^{-1} \circ \phi^p_j \circ \psi_k^{(j)} (x) & \textup{ if } m \geq 0, \\
-\sum_{p = 1}^{-m} r_j \circ \psi_{\sigma_j^{-p}(k)}^{(j)} \circ (\psi_{\sigma_j^{-p}(k)}^{(j)})^{-1} \circ \phi^p_j \circ \psi_k^{(j)} (x) & \textup{ if } m < 0, \end{cases}
\end{equation*}
we find that $R$ is $C^\infty$ on a neighbourhood of $0$ in $\mathcal{V}$.

Finally, notice that if $T_0 > 2 r_{\max}$, then we must have $m \geq \frac{T_0 - 2 r_{\max}}{r_{\max}} > 0$, and thus $\mathbf{F}$ is a half hyperbolic fixed point bounded by the unstable direction. Consequently, we have
\begin{equation}\label{eq:derivative_half_unstable}
D\mathbf{F}(0) = \begin{bmatrix}
\lambda & 0 \\ 0 & \mu 
\end{bmatrix}
\end{equation}
with $\lambda,\mu \in \mathbb{R}$ such that $|\lambda| > 1$ and $\mu \in (0,1)$. More precisely, we have
\begin{equation*}
D \mathbf{F}(0) = D ((\psi_{\sigma_j^m(k)}^{(j)})^{-1} \circ \phi_j \circ \psi_{\sigma_j^{m-1}(k)}^{(j)}) (0) \cdots D((\psi_{\sigma_j(k)}^{(j)})^{-1} \circ \phi_j \circ \psi_k^{(j)})(0)
\end{equation*}
and all the matrices of the product also have the form \eqref{eq:derivative_half_unstable}. 

Hence, if we let $\lambda_{\min}$ be the minimum of the modulus of the unstable eigenvalues of the matrices 
\begin{equation*}
D( (\psi_{\sigma_{j'}(k')}^{j'})^{- 1} \circ \phi_{j'} \circ \psi_{k'}^{(j')})(0), j' \in \set{1,\dots,N}, k' \in \mathbb{Z}/ 2 n_{j'} \mathbb{Z} \textup{ even},
\end{equation*}
and $\mu_{\max}$ be the maximum of their stable eigenvalues, we find that $\lambda$ and $\mu$ in \cref{eq:derivative_half_unstable} satisfies $|\lambda| \geq \lambda_{\min} > 1$ and $\mu \leq \mu_{\max} < 1$.
\end{remark}

Now that we understand the flow $\varphi$ near the singular orbits, we add a global hyperbolicity assumption and define the class of smooth pseudo-Anosov flows that we will study in the rest of the paper.

\begin{definition}\label{definition:pseudo_Anosov}
Let $\varphi = (\varphi_t)_{t \in \mathbb{R}}$ be a flow with pseudo-hyperbolic singularities on $M$. We will use the notations defined above. Fix any smooth Riemannian metric on $\widetilde{M}$. We say that $(\varphi_t)_{t \in \mathbb{R}}$ is a smooth pseudo-Anosov flow if it satisfies the following properties:
\begin{enumerate}[label=(\roman*)]
\item \label{item:hyperbolic_splitting} for every $x \in \widetilde{M}$, there is a splitting $T_x M = E_x^u \oplus E_x^s \oplus E_x^0$ of the tangent space to $M$ at $x$;
\item for every $x \in \widetilde{M}$, the direction $E_x^0$ is spanned by the generator $X(x)$ of the flow;
\item for every $x \in \widetilde{M}$ and $t \in \mathbb{R}$, we have $D\varphi_t (x) E_x^u = E_{\varphi_t(x)}^u$ and $D\varphi_t (x) E_x^s = E_{\varphi_t(x)}^s$;
\item \label{item:hyperbolic} for every neighbourhood $U$ of the singular orbits of $\varphi$, there are constants $C, \nu > 0$ such that for every $t \geq 0, x \in \widetilde{M} \setminus U, v_u \in E_{\varphi_t(x)}^u$ and $v_s \in E_x^s$, if $\varphi_t(x) \in \widetilde{M} \setminus U$ then we have
\begin{equation*}
| D\varphi_t(x) \cdot v_s| \leq C e^{- \nu t} |v_s| \textup{ and } | D\varphi_{-t}(\varphi_t(x)) \cdot v_u| \leq C e^{- \nu t} |v_u|;
\end{equation*}
\item for every $j \in \set{1,\dots,N}, k \in \mathbb{Z}/ 2 n_j \mathbb{Z}$ and $\tau \in [ - r_{\max},r_{\max}]$, there is a neighbourhood $U$ of $0$ in $\mathcal{V} \times (- r_{\min}/3, r_{\min}/3)$ such that the distribution of lines $\widetilde{E}^u$ and $\widetilde{E}^s$ defined on $U \setminus (\set{0} \times ( - r_{\min}/ 3, r_{\min}/3))$ by $\widetilde{E}^u_x = D \Upsilon_{j,k,\tau} (x)^{-1} E^u_{\Upsilon_{j,k,\tau}(x)}$ and $\widetilde{E}^s_x = D \Upsilon_{j,k,\tau} (x)^{-1} E^s_{\Upsilon_{j,k,\tau}(x)}$ are continuous and have continuous extensions to $U$ which are tangent respectively to $\mathbb{R} \times \set{0} \times (- r_{\min} /3, r_{\min}/3)$ and $\set{0} \times [0,+ \infty) \times (-r_{\min}/3,r_{\min}/3)$ if $k$ is even (the converse when $k$ is odd). \label{item:a_priori_direction}
\end{enumerate}
\end{definition}

\begin{remark}
Notice that a $C^\infty$ flow with no fixed point on a smooth manifold $M$ is a smooth flow with pseudo-hyperbolic singularities (take $N= 0$ in \cref{definition:admissible_singularities}). Consequently, a $C^\infty$ Anosov flow on a $3$-dimensional manifold is also a smooth pseudo-Anosov flow according to \cref{definition:pseudo_Anosov}.
\end{remark}

\begin{remark}
Item \ref{item:a_priori_direction} in \cref{definition:pseudo_Anosov} expresses the fact that the hyperbolicity of the flow on $\widetilde{M}$ and the hyperbolicity near the singular orbits are compatible. Notice that at $(0,0,0)$ the directions $\widetilde{E}^u$ and $\widetilde{E}^s$ do not need to be exactly $\mathbb{R} \times \set{0} \times \set{0}$ or $\set{0} \times \mathbb{R}_+ \times \set{0}$ since our local model for singular hyperbolic orbits is a suspension for an a priori non-constant roof function.
\end{remark}

\begin{remark}
Since we fix a Riemannian metric on the a priori non-compact manifold $\widetilde{M}$, one could be worried about a potential degeneracy of the Riemannian metric near the singular orbits. However, notice that our definition do not include a smooth structure on $M$, so that it would not make sense to require a Riemannian metric on $M$. Actually, we will not work with the Riemannian metric near the singular orbits, as we will use the local model to measure hyperbolicity of the flow there. Notice indeed that in \ref{item:hyperbolic}, we only require hyperbolicity measured with the Riemannian metric away from the singular orbits. 
\end{remark}

\begin{example} 
Let $M$ be an integral homology $3$-sphere and $\varphi$ be an Anosov flow on $M$. Let $\overline{M}$ be a ramified cover of $M$ over finitely many periodic orbits $\gamma_1,\dots,\gamma_n$ of $\varphi$, constructed as in \Cref{subsection:branched_covers}, and $\overline{\varphi}$ a continuous lift of $\varphi$ to $\overline{M}$. Then, $\overline{\varphi}$ is a smooth pseudo-Anosov flow. Indeed, it follows from \cref{example:ramified} that the singular orbits of $\bar{\varphi}$, that are just the lifts of $\gamma_1,\dots,\gamma_n$, are conjugated to the required local model. To define the stable and unstable directions of $\bar{\varphi}$, just lifts the stable and unstable directions of $\varphi$. Away from the singular orbits, $\bar{\varphi}$ looks locally like $\varphi$. Near the singular orbits, the same is true in half-space parametrization. Hence, \cref{definition:pseudo_Anosov} may be checked using the hyperbolicity of $\varphi$ and the continuity of its stable and unstable manifolds.
\end{example}

\begin{example}
Let $T$ be a linear pseudo-Anosov map on a half-translation surface $\Sigma$. Let $r : \Sigma \to (0,+ \infty)$ be a map smooth away from the singular points of $T$ and satisfying the regularity assumptions from \cref{definition:admissible_singularities} near the singular point (for instance $r$ can be chosen constant). Let then $M$ be the suspension of $\Sigma$ over $T$ with roof function $r$, and $\varphi$ the associated suspension flow. One may check that $\varphi$ satisfies \cref{definition:pseudo_Anosov}. As explained in \cref{section:motivation_pseudo_Anosov}, in that case we can find a representation $\rho$ of the fundamental group of $M$ such that \cref{theorem:fried_conjecture} applies.
\end{example}

\begin{example}
More examples of pseudo-Anosov flows should be obtained by making a smooth perturbation of the generator of a pseudo-Anosov flow away from the singular orbits (the hyperbolicity can probably be obtained using a cone-field criterion). One could also apply techniques used to produce Anosov flows (such as Goodman surgery) to produce more examples. Finally, we expect that every topological pseudo-Anosov flow is orbit equivalent to a smooth pseudo-Anosov flow (see \cite{agol_tsang} for a discussion of this question based on the work of Shannon \cite{shannon}).
\end{example}

\section{Local stable and unstable manifolds of smooth pseudo-Anosov flows}\label{section:stable_unstable_manifolds}

In order to construct Markov partitions adapted to the class of smooth pseudo-Anosov flows under study, we need to investigate the stable and unstable manifolds for such flows. There will only be few differences with the Anosov case (the interested reader may refer for instance to \cite{fisher_hasselblatt_book} for a discussion of this case). However, we are not aware of a reference dealing with the case of pseudo-Anosov flow that would include a discussion of the specific regularity properties near the singular orbits that we need in order to get a meromorphic continuation of the zeta function. Let us fix for this section a smooth pseudo-Anosov flow $\varphi = (\varphi_t)_{t \in \mathbb{R}}$ on a manifold $M$, using the notations from \cref{subsection:definition_pseudo_Anosov}. We will use the notations introduced in \cref{section:definition_pseudo_Anosov}. Fix any distance $d$ on $M$ compatible with the topology of the manifold.

\begin{definition}
For $\epsilon > 0$ and $x \in M$ the local (strong) stable and unstable manifolds $W^s_\epsilon(x)$ and $W^u_\epsilon(x)$ are the connected components of $x$ in
\begin{equation*}
\set{ y \in M : d(\varphi_t(x), \varphi_t(y)) \underset{ t \to + \infty}{\to} 0 \textup{ and } \forall t \geq 0: d(\varphi_t(x),\varphi_t(y)) < \epsilon}
\end{equation*}
and
\begin{equation*}
\set{ y \in M : d(\varphi_t(x), \varphi_t(y) \underset{ t \to - \infty}{\to} 0 \textup{ and } \forall t \leq 0: d(\varphi_t(x),\varphi_t(y)) < \epsilon}
\end{equation*}
respectively.
\end{definition}

The main result that we will use concerning local stable and unstable manifolds is the following. The proposition is stated for unstable manifolds, but of course a similar statement holds for stable manifolds. While it may seem technical, \cref{proposition:unstable_manifolds} just asserts that the situation is very similar to the case of smooth Anosov flows, except that near the singular orbits one should use coordinates given by half-space parametrizations from \cref{remark:flow_in_half_space_chart} instead of smooth coordinates in order to describe the local unstable manifolds. The only true difference happens for points that are near a singular orbits in its weak unstable manifold: in that case the local unstable manifold crosses the singular orbit and is not a manifold anymore.

\begin{proposition}\label{proposition:unstable_manifolds}
Let $\delta_0 > 0$. Then there are $\epsilon_0 > 0, \delta \in (0,\delta_0)$ and open subsets $U_{\textup{sing}}$ and $U_{\textup{reg}}$ in $M$ such that for every $\epsilon \in (0,\epsilon_0)$ we have
\begin{enumerate}[label=(\roman*)]
\item \label{item:covering} $M = U_{\textup{sing}} \cup U_{\textup{reg}}$;
\item \label{item:standard_away_singular} for every $x \in U_{\textup{reg}}$, the set $W_\epsilon^{u}(x)$ is a $C^\infty$ submanifold of dimension $1$ of $M$ (uniformly in $x$) tangent to the unstable direction;
\item \label{item:sing_close_singular} for every $x \in U_{\textup{sing}}$, there is $j \in \set{1,\dots, N}, k \in \mathbb{Z}/ 2n_j \mathbb{Z}$ and $\tau \in [-r_{\max}/2, r_{\max}/2]$ such that  $x$ belongs to the intersection $$\Upsilon_{j,k,\tau}([ - \frac{\delta}{2},\frac{\delta}{2}] \times [0,\frac{\delta}{2}] \times [ -\frac{\delta}{2},\frac{\delta}{2}]) \cap \Upsilon_{j,k+1,\tau}([ - \frac{\delta}{2},\frac{\delta}{2}] \times [0,\frac{\delta}{2}] \times [ - \frac{\delta}{2},\frac{\delta}{2}]);$$
\item \label{item:unstable_parallel} if $x \in U_{\textup{sing}},j \in \set{1,\dots,N}, k \in \mathbb{Z}/ 2 n_j \mathbb{Z}$ and $\tau \in [-r_{\max}/2,r_{\max}/2]$ are such that $x \in \Upsilon_{j,k,\tau}([ - \delta/2,\delta/2] \times [0,\delta/2] \times [ - \delta/2,\delta/2])$ and $k$ is even, then there is an interval $I \subseteq [ - \delta,\delta]$, and a $C^\infty$ function $F : I \to \mathbb{R}^2$ such that
$$([- \delta,\delta] \times [0,\delta]  \times [ - \delta,\delta])  \cap \Upsilon_{j,k,\tau}^{-1}(W_\epsilon^u(x)) = \set{(y,F(y)): y \in I},$$ and for every $y \in I$ the vector $(1, F'(y))$ spans the pullback of the unstable direction by $\Upsilon_{j,k,\tau}$ at $x$ (eventually extended by continuity);
\item \label{item:normal_parallel} under the assumption of the previous point, if in addition $x \notin W_{\textup{sing}}^{\textup{lu}}$, then $W_\epsilon^u(x)$ is contained in $\Upsilon_{j,k,\tau}([- \delta,\delta] \times [0,\delta] \times [- \delta,\delta])$;
\item \label{item:unstable_perpendicular} if $x \in U_{\textup{sing}},j \in \set{1,\dots,N}, k \in \mathbb{Z}/ 2 n_j \mathbb{Z}$ and $\tau \in [-r_{\max},r_{\max}]$ are such that $x \in \Upsilon_{j,k,\tau}([ - \delta/2,\delta/2] \times [0,\delta/2] \times [ - \delta/2,\delta/2])$ and $k$ is odd, then there is an interval $I \subseteq [0,\delta]$, two $C^\infty$ functions $f,g : I \to \mathbb{R}$ such that $$([- \delta,\delta] \times [0,\delta]  \times [ - \delta,\delta])  \cap \Upsilon_{j,k,\tau}^{-1}(W_\epsilon^u(x)) = \set{(f(y),y,g(y)) : y \in I},$$ and for every $y \in I$, the vector $(f'(y), 1, g'(y))$ spans the pullback of the unstable direction by $\Upsilon_{j,k,\tau}$ at $y$ (eventually extended by continuity);
\item \label{item:uniform_graph} the functions $F,f$ and $g$ in \ref{item:unstable_parallel} and \ref{item:unstable_perpendicular} are $C^\infty$ uniformly in $x$, and the size of $I$ is uniformly bounded below (depending on $\epsilon$).
\end{enumerate}
\end{proposition}

\begin{remark}
Let us help the reader understand the meaning of \cref{proposition:unstable_manifolds}. We cover $M$ by two open sets. On the first one ($U_{\textup{reg}}$) the local unstable manifolds behave exactly as in the Anosov case according to \ref{item:standard_away_singular}. The other set ($U_{\textup{sing}}$) is an arbitrarily small neighbourhood of the singular orbits (as \ref{item:sing_close_singular} imposes that it is contained in the image of a small neighbourhood of zero by half-space parametrizations). The item \ref{item:unstable_parallel} means that if you take a half-space parametrization that makes the unstable direction almost parallel to the first axis (this is what it means for $k$ to be even), then in this parametrization the local unstable curve becomes a smooth graph over the first axis. By taking $k$ odd, we have the same picture up to a $\pi/2$ rotation (this is \ref{item:unstable_perpendicular}). The item \ref{item:normal_parallel} asserts that if $x$ is near a singular orbit then its local unstable will remain in the range of a half-space parametrization (and thus be a smooth curve) unless $x$ is exactly on the local unstable manifold of a singular orbit (in which case you may have a singularity at the singular orbit).

The proof of \cref{proposition:unstable_manifolds} is relatively straightforward: we will use a graph transform argument (see for instance \cite[Theorem 6.2.8]{katok_hasselblatt_book}). The only difficulty is that we need to use smooth coordinates away from the singular orbits and half-space parametrizations near the singular orbits. We must consequently check that the unstable manifolds given in coordinates by the graph transform argument do not leave the domain of the half-space parametrizations when we use them, so that these unstable manifolds correspond to actual subsets of $M$. This is achieved by using half-space parametrizations that are bounded by the local weak unstable manifolds of the singular orbits (which we know how to define a priori): this way, an unstable manifold that intersects the boundary of the domain of a half-space parametrization must be contained in this boundary, and in particular in the domain of the parametrization.
\end{remark}

Finally, we will also need to know that the local unstable manifold $W_\epsilon^u(x)$ depends continuously on $x$ in the following sense:

\begin{proposition}\label{proposition:continuity_unstable_manifolds_regular}
Let $\epsilon > 0$ be small enough. Let $x_0 \in \widetilde{M}$. Let $\kappa : U \to V$ be a $C^\infty$ diffeomorphism from a neighbourhood of $0$ in $\mathbb{R}^3$ and a neighbourhood of $x_0$ in $M$ such that $\kappa(0) = x_0$ and $D \kappa(0) (\mathbb{R} \times \set{0})= E_{x_0}^u$. Then, there is an open interval $I$ in $\mathbb{R}$ containing $0$, an open neighbourhood $W$ of $x_0$ in $M$, and for every $x \in W$ a $C^\infty$ function $F_x : I \to \mathbb{R}^2$ such that:
\begin{enumerate}[label=(\roman*)]
\item for every $x$ in $W$ the set $\set{(t,F_x(t)) : t \in I}$ is contained in $U$ and the image $\kappa(\set{(t,F_x(t)) : t \in I})$ is the intersection of $\kappa((I \times \mathbb{R}^2) \cap U)$ with $W_\epsilon^u(x)$;
\item the maps $(t,x) \mapsto F_x(t)$ and $(t,x) \mapsto F_x'(t)$ are continuous from $I \times W$ to $\mathbb{R}^2$.
\end{enumerate}
\end{proposition}

\begin{proposition}\label{proposition:continuity_unstable_manifolds_singular_parallel}
Let $\epsilon > 0$. Let $j \in \set{1,\dots,N}, k \in \mathbb{Z}/ 2 n_j \mathbb{Z}$ even and $\tau \in [-r_{\max},r_{\max}]$. Then, there are $\delta, \delta' > 0$ and for each $x \in (- \delta,\delta) \times [0,\delta) \times (-\delta,\delta)$ a $C^\infty$ function $F_x$ from $(-\delta,\delta)$ to $\mathbb{R}^2$ such that
\begin{enumerate}[label=(\roman*)]
\item for every $x \in (-\delta,\delta) \times [0,\delta) \times(-\delta,\delta)$ the set $\set{(t, F_x(t)) : t \in (-\delta,\delta}$ is contained in $(-\delta,\delta) \times [0,\delta') \times (-\delta',\delta')$ and $\Upsilon_{j,k,\tau}(\set{(t, F_x(t)) : t \in I})$ is the intersection of $\Upsilon_{j,k,\tau}((-\delta,\delta) \times [0,\delta') \times (-\delta',\delta'))$ with $W_\epsilon^u (\Upsilon_{j,k,\tau}(x))$;
\item the maps $(t,x) \mapsto F_x(t)$ and $(t,x) \mapsto F_x'(t)$ are continuous from $(-\delta,\delta) \times ((-\delta,\delta) \times [0,\delta) \times (- \delta,\delta))$ to $\mathbb{R}^2$.
\end{enumerate}
\end{proposition}

\begin{proposition}\label{proposition:continuity_unstable_manifolds_singular_perpendicular}
Let $\epsilon > 0$. Let $j \in \set{1,\dots,N}, k \in \mathbb{Z}/ 2 n_j \mathbb{Z}$ odd and $\tau \in [-\tau_{\max},\tau_{\max}]$. Then, there are $\delta, \delta' > 0$ and for each $x \in (- \delta,\delta) \times [0,\delta) \times (-\delta,\delta)$ two $C^\infty$ functions $f_x,g_x$ from $[0,\delta)$ to $\mathbb{R}$ such that
\begin{enumerate}[label=(\roman*)]
\item for every $x \in (- \delta,\delta) \times [0,\delta) \times(-\delta,\delta)$ the set $\set{(f_x(t),t, g_x(t)) : t \in I}$ is contained in $(- \delta',\delta') \times [0,\delta) \times (-\delta',\delta')$ and $\Upsilon_{j,k,\tau}(\set{(f_x(t),t, g_x(t)) : t \in [0,\delta)})$ is the intersection of $\Upsilon_{j,k,\tau}((- \delta',\delta') \times [0,\delta) \times (-\delta',\delta'))$ with $W_\epsilon^u (\Upsilon_{j,k,\tau}(x))$;
\item the map $(t,x) \mapsto (f_x(t), g_x(t),f_x'(t), g_x'(t))$ is continuous from $[0,\delta) \times (- \delta,\delta) \times [0,\delta) \times (- \delta,\delta)$ to $\mathbb{R}^4$.
\end{enumerate}
\end{proposition}

The section is structured as follows. We start with some preliminary results in \cref{subsection:preliminary} and then prove \cref{proposition:unstable_manifolds} in \ref{section:local_unstable_manifold}. Finally, we prove Propositions \ref{proposition:continuity_unstable_manifolds_regular}, \ref{proposition:continuity_unstable_manifolds_singular_parallel} and \ref{proposition:continuity_unstable_manifolds_singular_perpendicular} in \cref{subsection:continuity_unstable_manifold}.

\subsection{Preliminary results on smooth pseudo-Anosov flows}\label{subsection:preliminary}

Before writing a graph transform argument to prove \cref{proposition:unstable_manifolds}, we need to establish the continuity of the unstable direction away from the singular orbits (\cref{lemma:continuity_unstable}). We also need some result on the local weak unstable of the singular orbits in half-space parametrizations (\cref{lemma:local_weak_unstable_half_space}).

\begin{lemma}\label{lemma:continuity_unstable}
The directions $E^u_x$ and $E^s_x$ depend continuously on $x \in \widetilde{M}$.
\end{lemma}

\begin{proof}
Let us deal for instance with the case of the unstable direction. First of all, notice that one may cover a neighbourhood of the singular orbits by the ranges of a finite number of half-space parametrizations. It follows then from item \ref{item:a_priori_direction} of Definition \ref{definition:pseudo_Anosov} that $E^u_x$ depends continuously on $x$ in some neighbourhood $V$ of the singular orbits (with the singular orbits removed). Since the unstable direction is invariant under the action of the flow, it follows that it is continuous at every point on the unstable manifold of a singular orbit.

So pick $x \in \widetilde{M} \setminus V$ that does not belong to the unstable manifold of a singular orbit. Let $(x_n,v_n)_{n \geq 0}$ be a sequence of points in $T \widetilde{M}$ such that $v_n \in E_{x_n}^u$ for every $n \geq 0$. Assume that $(x_n,v_n)_{n \geq 0}$ converges to $(x,v)$ for some $v \in T_xM$. We want to prove that $v \in E_x^u$ (since we can do the same with the stable direction, which complements the unstable direction, it implies that the unstable direction depends continuously on $x$).

Since $x$ does not belong to the unstable manifold of a singular orbit, there are a neighbourhoods $\widetilde{V}$ of the singular orbits and a sequence $(t_j)_{j \geq 0}$ going to $+ \infty$ such that $\varphi_{-t_j}(x) \in \widetilde{M} \setminus \widetilde{V}$ for every $j \geq 0$. Let $U$ be an open neighbourhood of the singular orbits such that $\overline{U}$ is contained in the interior of $\widetilde{V}$. Let $C$ and $\nu$ be the constant given by \ref{item:hyperbolic} (from \cref{definition:pseudo_Anosov}) applied to $U$.

Let $j \geq 0$. For $n$ large enough, the continuity of $\varphi$ implies that $x_n$ and $\varphi_{-t_j} (x)$ belongs to $M \setminus U$. Hence, we have
\begin{equation*}
|D\varphi_{-t_j}(x_n) \cdot v_n | \leq C e^{- \nu t_j } |v_n|.
\end{equation*}
Since $\varphi_t$ is smooth on $\widetilde{M}$, it follows that
\begin{equation}\label{eq:contraction_limit}
|D\varphi_{-t_j}(x) \cdot v | \leq C e^{- \nu t_j } |v|.
\end{equation}
Now, decompose $v$ as $v = v_u+ v_s + v_0$ where $v_u \in E_x^u, v_s \in E_x^s$ and $v_0 \in E_x^0$. Write first that for $j \geq 0$
\begin{equation}\label{eq:remove_vs}
|D\varphi_{-t_j}(x) \cdot v | \geq C^{-1} e^{\nu t_j} |v_s| - C e^{- \nu t_j}|v_u| - |D\varphi_{-t_j}(x) \cdot v_0|.
\end{equation}
By compactness of $\widetilde{M} \setminus U$, we find that $|D\varphi_{-t_j}(x) \cdot v_0|$ is bounded uniformly in $j \geq 0$ (since $X$ is invariant under the action of $\varphi$). Thus, $v_s = 0$, since otherwise \cref{eq:remove_vs} would contradict \cref{eq:contraction_limit}. We find now that
\begin{equation*}
|D\varphi_{-t_j}(x) \cdot v | \geq |D\varphi_{-t_j}(x) \cdot v_0|  - C e^{- \nu t_j}|v_u|.
\end{equation*}
This would again contradict \cref{eq:contraction_limit} if $v_0 \neq 0$ (using that $X$ is bounded away from zero on $\widetilde{M} \setminus U$). Thus, $v_0=0$, and $v \in E_x^u$.
\end{proof}

\begin{lemma}\label{lemma:local_weak_unstable_half_space}
There is $\delta > 0$ such that for every $j \in \set{1,\dots,N}$, every even $k \in \mathbb{Z}/ 2 n_j \mathbb{Z}$ and every $\tau \in [ - r_{\max}, r_{\max}]$ we have
\begin{equation*}
\Upsilon_{j,k,\tau}^{-1}(W_{\textup{sing}}^{\textup{lu}}) \cap ([ - \delta,\delta] \times [0,\delta] \times [ - \delta,\delta]) = [- \delta,\delta] \times \set{0} \times [- \delta,\delta].
\end{equation*}
\end{lemma}

\begin{proof}
First, use that $\Upsilon_{j,k,\tau}$ is continuous uniformly in $j,k,\tau$ to find that $\delta$ small ensures that
\begin{equation*}
\Upsilon_{j,k,\tau}^{-1}(W_{\textup{sing}}^{\textup{lu}}) \cap ([ - \delta,\delta] \times [0,\delta] \times [ - \delta,\delta]) = \Upsilon_{j,k,\tau}^{-1}(W_{j}^{\textup{lu}}) \cap ([ - \delta,\delta] \times [0,\delta] \times [ - \delta,\delta]).
\end{equation*}
Indeed, it follows from \cref{lemma:true_local_weak_unstable} that $\gamma_j$ does not intersect the local unstable manifold of another singular orbit. The result follows then from the definition of the local unstable manifold of $\gamma_j$.
\end{proof}

\subsection{Construction of local unstable manifolds}\label{section:local_unstable_manifold}

We can now start the proof of \cref{proposition:unstable_manifolds}. In \cref{subsubsection:family_charts}, we construct some specific charts suited to describe the action of the flow $\varphi$. In \cref{subsubsection:graph_transform}, we use these charts to set up a graph transform argument. Finally, we sum up our findings in \cref{subsubsection:proof_unstable_manifolds} and prove \cref{proposition:unstable_manifolds}.

\subsubsection{A family of charts}\label{subsubsection:family_charts}

Let $\mathcal{A}$ be the set of $(j,k,\tau)$ where $j \in \set{1,\dots,N}$, $k \in \mathbb{Z}/ 2n_j \mathbb{Z}$ is even and $\tau \in [-r_{\max},r_{\max}]$. For $\alpha > 0$, we let $D_\alpha = [ - \alpha,\alpha ] \times [0,\alpha ] \times [ - \alpha,\alpha ]$. Let us fix some $\alpha_0$ such that $D_{2 \alpha_0} \subseteq \mathcal{V} \times (- r_{\min}/3,r_{\min}/3)$. By taking $\alpha_0$ small enough, we ensure that for every $(j,k,\tau) \in \mathcal{A}$, there are continuous functions $y_{j,k,\tau}^u,y_{j,k,\tau}^s, t^u_{j,k,\tau}$ and $t^{s}_{j,k,\tau}$ from $\mathcal{V}$ to $\mathbb{R}$ such that for every $(z,t) \in D_{\alpha_0} \setminus (\set{(0,0)} \times [ - \alpha_0,\alpha_0])$ the lines $D \Upsilon_{j,k,\tau} (z,t)^{-1} E^u_{\Upsilon_{j,k,\tau}(z,t)}$ and $D \Upsilon_{j,k,\tau} (z,t)^{-1} E^s_{\Upsilon_{j,k,\tau}(z,t)}$ are spanned respectively by $(1,y^u_{j,k,\tau}(z), t^u_{j,k,\tau}(z))$ and $(y^s_{j,k,\tau}(z), 1, t^s_{j,k,\tau}(z))$. This is a consequence of item \ref{item:a_priori_direction} in \cref{definition:pseudo_Anosov}, which also implies that $y^u_{j,k,\tau}$ and $y^s_{j,k,\tau}$ are identically zero respectively on $[- \alpha_0,\alpha_0] \times \set{0}$ and on $\set{0} \times [0,\alpha_0]$. We also assume that $\alpha_0$ is small enough so that the sets
\begin{equation*}
\bigcup_{\substack{ k \in \mathbb{Z}/2 n_j \mathbb{Z} \\ \tau \in [ - r_{\max},r_{\max}]}} \Upsilon_{j,k,\tau}(D_{2 \alpha_0}), \quad j = 1,\dots,N
\end{equation*}
are pairwise disjoint.

Choose some time $T_0 > 2 r_{\max}$. It follows from \cref{remark:flow_in_half_space_chart} that there is $\alpha_1 \in (0,\alpha_0)$ such that for every $(j,k,\tau) \in \mathcal{A}$ there are $k',\tau'$ such that $(j,k',\tau') \in \mathcal{A}$ and the map $\varphi_{-T_0}$ sends $\Upsilon_{j,k,\tau}(D_{\alpha_1})$ inside  $\Upsilon_{j,k',\tau'}(D_{\alpha_0})$. Moreover, there is a half hyperbolic fixed point bounded by the unstable direction $\mathbf{F}_{j,k,\tau}$, and a $C^\infty$ map $R_{j,k,\tau} : [ - \alpha_1,\alpha_1] \times [0,\alpha_1] \to \mathbb{R}$ such that 
\begin{equation}\label{eq:local_model_near_singular}
\varphi_{-T_0} \circ \Upsilon_{j,k,\tau}(x,t) = \Upsilon_{j,k',\tau'}(\mathbf{F}^{-1}_{j,k,\tau}(x),t + R_{j,k,\tau}(x)) 
\end{equation}
for every $(x,t) \in D_{\alpha_1}$. In addition, we can impose that $\tau ' \in [ - r_{\max}/2,r_{\max}/2]$ and that $R_{j,k,\tau}(0) = 0$.

If $(j,k,\tau) \in \mathcal{A}$ and $z = (z_1,z_2) \in [- \alpha_0,\alpha_0] \times [0 ,\alpha_0]$, define the affine map on $\mathbb{R}^3$:
\begin{equation*}
\Xi_{j,k,\tau,z} :  (x_1,x_2,t) \mapsto (\xi_{j,k,\tau,z}(x_1,x_2), t + \eta_{j,k,\tau,z}(x_1,x_2) )
\end{equation*}
where
\begin{equation*}
\xi_{j,k,\tau,z}(x_1,x_2) = (z_1 + x_1 + x_2 y^s_{j,k,\tau}(z), z_2 + x_2 + x_1 y^u_{j,k,\tau}(z))
\end{equation*}
and
\begin{equation*} 
\eta_{j,k,\tau,z}(x_1,x_2) = x_1 t_{j,k,\tau}^u (z) + x_2 t_{j,k,\tau}^s(z).
\end{equation*}
Let $k'$ and $\tau'$ be as in \cref{eq:local_model_near_singular} and assume that $z$ is close to $0$. Let $z' = \mathbf{F}_{j,k,\tau}^{-1}(z)$ and $\tau'' = \tau' + R_{j,k,\tau}(z)$. Then, for $(x,t)$ near $0$ in $\Xi_{j,k,\tau,z}^{-1}(\overline{\mathbb{H}} \times \mathbb{R})$ we have
\begin{equation}\label{eq:explicit_transition}
\Xi_{j,k',\tau'',z'}^{-1} \circ \Upsilon_{j,k',\tau''}^{-1} \circ \varphi_{-T_0} \circ \Upsilon_{j,k,\tau,z} \circ \Xi_{j,k,\tau,z}(x,t) = (\mathbf{F}_{j,k,\tau,z}^{-1}(x), t + R_{j,k,\tau,z}(x)),
\end{equation}
where
\begin{equation*}
\mathbf{F}_{j,k,\tau,z} =  \xi_{j,k,\tau,z}^{-1} \circ \mathbf{F}_{j,k,\tau} \circ \xi_{j,k,\tau'',z'}
\end{equation*}
and
\begin{equation*}
R_{j,k,\tau,z} = \eta_{j,k,\tau,z} + R_{j,k,\tau} \circ \xi_{j,k,\tau,z} - \eta_{j,k,\tau'',z'} \circ \mathbf{F}_{j,k,\tau,z}^{-1} - R_{j,k,\tau}(z).
\end{equation*}

Since the stable and unstable directions are stable under the action of the flow, we find that $D \mathbf{F}_{j,k,\tau,z}(0)$ is diagonal:
\begin{equation*}
D \mathbf{F}_{j,k,\tau,z}(0) = \begin{bmatrix}
\lambda_{j,k,\tau,z} & 0 \\ 0 & \mu_{j,k,\tau,z}
\end{bmatrix}.
\end{equation*}
Moreover, since $\Xi_{j,k,\tau,z}$ and $\Xi_{j,k,\tau'',z'}$ depend continuously on $z$ and are the identity when $z = 0$, we find that for $z$ small enough we have
\begin{equation}\label{eq:hyperbolicity_singular}
|\lambda_{j,k,\tau,z}| > (1 + \lambda_{\min})/2 \textup{ and } \mu_{j,k,\tau,z} \in (0, (1+ \mu_{\max})/2).
\end{equation}

Let now $\alpha_2 \in (0,\alpha_1)$ be small enough so that there is $\varpi > 0$ such that, for every $(j,k,\tau) \in \mathcal{A}$ and $z \in [-\alpha_2,\alpha_2] \times [0,\alpha_2]$, the relation \eqref{eq:explicit_transition} holds on $[- \varpi,\varpi]^3 \cap \Xi_{j,k,\tau,z}^{-1}(\overline{\mathbb{H}} \times \mathbb{R})$ and the estimate \eqref{eq:hyperbolicity_singular} is satisfied.

By continuity of $\varphi_{-T_0}$, we find that there is $\alpha_3 \in (0,\alpha_2)$, such that for every $(j,k,\tau) \in \mathcal{A}$, the set $\varphi_{-T_0} \circ \Upsilon_{j,k,\tau}(D_{2 \alpha_0} \setminus (( - \alpha_2, \alpha_2) \times [0,\alpha_2) \times [ - 2 \alpha_0,2 \alpha_0]))$ does not intersect $\bigcup_{(\tilde{j},\tilde{k},\tilde{\tau}) \in \mathcal{A}} \Upsilon_{\tilde{j},\tilde{k},\tilde{\tau}}(D_{\alpha_3})$. We also impose that $\alpha_3$ is small enough so that for every $(j,k,\tau) \in \mathcal{A}$, we have
\begin{equation}\label{eq:separation_orbites}
\varphi_{T_0} (\Upsilon_{j,k,\tau}(D_{\alpha_3})) \cap \bigcup_{\substack{(\tilde{j},\tilde{k},\tilde{\tau}) \in \mathcal{A} \\ \tilde{j} \neq j}} \Upsilon_{\tilde{j},\tilde{k},\tilde{\tau}}(D_{2 \alpha_0}) = \emptyset.
\end{equation}

Notice that there is a neighbourhood $U$ of $D_{2 \alpha_0} \setminus ( \set{(0,0)} \times [ - 2 \alpha_0,2 \alpha_0])$ in $\mathbb{R}^3$ such that for every $(j,k,\tau) \in \mathcal{A}$ the map $\Upsilon_{j,k,\tau}$ extends to a $C^\infty$ diffeomorphism from $W$ to some subset of $M$. Up to taking $\varpi$ smaller, we may assume that for every $(j,k,\tau) \in \mathcal{A}$ and $z = (z_1,z_2) \in [ - \alpha_0,\alpha_0] \times [0,\alpha_0]$ such that $|z_1| \geq \alpha_3$ the set $\Xi_{j,k,\tau,z}([- \varpi,\varpi]^3)$ is contained in $U$.

Let us now define for every $(j,k,\tau) \in \mathcal{A}$ and $z \in [- \alpha_0,\alpha_0] \times [0 ,\alpha_0]$ the ``parametrization'' $\kappa_{j,k,\tau,z} :  P_{j,k,\tau,z} \to M$ by 
\begin{equation*}
\kappa_{j,k,\tau,z} : (y_1,y_2,t)  \mapsto \Upsilon_{j,k, \tau} (\Xi_{j,k,\tau,z}(y_1,y_2,t)),
\end{equation*}
where
\begin{equation*}
P_{j,k,\tau,z} = \begin{cases} [ - \varpi,\varpi]^{3} & \textup{ if } |z_1| \geq \alpha_3 \\ [ - \varpi,\varpi]^3 \cap \Xi_{j,k,\tau,z}^{-1}( \overline{\mathbb{H}} \times \mathbb{R}) & \textup{ otherwise, } \end{cases}
\end{equation*}
We will also need the notation
\begin{equation*}
\widetilde{P}_{j,k,\tau,z} = [ - \varpi,\varpi]^3 \cap \Xi_{j,k,\tau,z}^{-1}( \overline{\mathbb{H}} \times \mathbb{R}).
\end{equation*}

Let us list some properties of the maps we just constructed.

\begin{lemma}\label{lemma:properties_chart_near_singular}
\begin{enumerate}[label =(\roman*)]
\item \label{item:clear} For every $(j,k,\tau) \in \mathcal{A}$ and $z \in [- \alpha_0,\alpha_0] \times [0,\alpha_0]$, we have $\kappa_{j,k,\tau,z}(0) = \Upsilon_{j,k, \tau}(z,0).$ Moreover, $ D \kappa_{j,k,\tau,z}(0) \cdot (1,0,0) \in E^u_{\kappa_{j,k,\tau,z}(0)}$ and $D \kappa_{j,k,\tau,z}(0) \cdot (0,1,0) \in E^s_{\kappa_{j,k,\tau,z}(0)}$.
\item \label{item:coordinates_local_weak_unstable} For every $(j,k,\tau) \in \mathcal{A}$ and $z \in [ - \alpha_0,\alpha_0] \times [0,\alpha_0]$, we have
\begin{equation*}
\Xi_{j,k,\tau,z}^{-1}(\mathbb{R} \times \set{0} \times \mathbb{R}) \cap P_{j,k,\tau,z} = \kappa_{j,k,\tau,z}^{-1}(W_{\textup{sing}}^{\textup{lu}}).
\end{equation*}
\item \label{item:no_abrupt_change} Let $(j,k,\tau), (\tilde{j},\tilde{k},\tilde{\tau}) \in \mathcal{A}$ and $z,\tilde{z} \in [- \alpha_0,\alpha_0] \times [0,\alpha_0]$. Assume that $\varphi_{-T_0} (\kappa_{j,k,\tau,z}(0)) = \kappa_{\tilde{j},\tilde{k},\tilde{\tau},\tilde{z}}(0)$ and $\tilde{z} \in [ - \alpha_3,\alpha_3] \times [0, \alpha_3]$. Then $j = \tilde{j}$ and $z \in [- \alpha_2,\alpha_2] \times [0,\alpha_2]$.
\end{enumerate}
\end{lemma}

\begin{proof}
The point \ref{item:clear} follows directly from the definition.

Item \ref{item:coordinates_local_weak_unstable} is a consequence of \cref{lemma:local_weak_unstable_half_space}.

It remains to prove \ref{item:no_abrupt_change}. The fact that $j = \tilde{j}$ follows from \cref{eq:separation_orbites}. For the other points, just notice that $(z,0) \in D_{2 \alpha_0}$ and $\varphi_{-T_0}( \Upsilon_{j,k,\tau}(z,0)) \in \Upsilon_{\tilde{j},\tilde{k},\tilde{\tau}}(D_{\alpha_3})$. By our definition of $\alpha_3$, it follows then that $z \in [ - \alpha_2,\alpha_2] \times [0,\alpha_2]$.
\end{proof}

The set $V_0 \coloneqq \set{ \kappa_{j,k,\tau,z}(0) : (j,k,\tau) \in \mathcal{A}, z \in [ - \alpha_0,\alpha_0] \times [0,\alpha_0]}$ is a neighbourhood of the singular orbits of $\varphi$ in $M$, as follows from \ref{item:clear} in \cref{lemma:properties_chart_near_singular}. For every $x \in M \setminus V_0$ choose a parametrization $\kappa_x : [- \varpi,\varpi]^3 \to M$ (up to making $\varpi$ smaller) such that the following properties hold:
\begin{enumerate}[label=(\roman*)]
\item for every $x \in M \setminus V_0$, we have $\kappa_x(0) = x$;
\item for every $x \in M \setminus V_0$ and $w \in [- \varpi,\varpi]^3$, we have $D\kappa_x(w) \cdot (0,0,1) = X(\kappa_x(w))$ (where $X$ is the generator of $\varphi$);
\item for every $x \in M \setminus V_0$, we have $D\kappa_x(0) \cdot (1,0,0) \in E_x^u$ and $ D\kappa_x(0) \cdot (0,1,0) \in E_x^s$.
\item the $\kappa_x$'s are uniformly $C^\infty$.
\end{enumerate}

Let us call $\mathcal{C}$ the set of all parametrizations of the form $\kappa_x$ for $x \in M \setminus V_0$ and of the form $\kappa_{j,k,\tau,z}$ for $(j,k,\tau) \in \mathcal{A}$ and $z \in ([ - \alpha_0,\alpha_0] \times [0,\alpha_0]) \setminus ([ - \alpha_3, \alpha_3] \times [0,\alpha_3])$. Notice that there is a neighbourhood $U_0$ of the singular orbits that does not intersect $\set{\kappa(0) : \kappa \in \mathcal{C}}$. Let $C_0$ and $\nu$ be the constants given by \ref{item:hyperbolic} in \cref{definition:pseudo_Anosov} for this neighbourhood $U_0$ of the singular orbits.

Let us now define the constant
\begin{equation*}
B_1 = \sup_{\kappa \in \mathcal{C}} \sup_{ w \in \set{(1,0,0),(0,1,0)}} \max( |D\kappa(0) \cdot w|,|D\kappa(0) \cdot w|^{-1}).
\end{equation*}

Choose then a positive integer $L$ large enough such that 
\begin{equation*}
\mathfrak{h} \coloneqq \max \left(\left(\frac{1 + \lambda_{\min}}{2} \right)^{-1}, \frac{1+\mu_{\max}}{2}, e^{- \nu T_0}\right) C_0^{\frac{1}{L}} B_1^{\frac{2}{L}} < 1.
\end{equation*}

\subsubsection{Graph transform argument}\label{subsubsection:graph_transform}

Let us now fix a point $x_0 \in M$. We will construct a local unstable manifold for $x_0$. For every $n \leq 0$, define $x_n = \varphi_{n T_0} (x_0)$. For each $n \leq 0$, we want to define a map $\kappa_n$ that will be either one of the $\kappa_{j,k,\tau,z}$'s or one of the $\kappa_x$'s from \cref{subsubsection:family_charts}. This definition will be made by induction:
\begin{itemize}
\item if $x_0 \notin V_0$, we just take $\kappa_0 = \kappa_{x_0}$, otherwise there are $(j,k,\tau) \in \mathcal{A}$ and $z \in [- \alpha_0,\alpha_0] \times [0,\alpha_0]$ such that $\kappa_{j,k,\tau,z}(0) = x_0$, and we set $\kappa_{0} = \kappa_{j,k,\tau,z}$;
\item assume that $\kappa_0,\kappa_{-1},\dots, \kappa_n$ are defined and let us define $\kappa_{n-1}$. We distinguish two cases. The first case is when $\kappa_n = \kappa_{j,k,\tau,z}$ for some $(j,k,\tau) \in \mathcal{A}$ and $z \in [-\alpha_2,\alpha_2] \times [0, \alpha_2]$. In that case, we let $j,k',\tau''$ and $z'$ be as in \cref{eq:explicit_transition} and we set $\kappa_{n-1} = \kappa_{j,k',\tau'',z'}$. We say that $n-1$ is of type II. Otherwise, we define $\kappa_{n-1}$ as we did for $\kappa_0$, and we say that $n-1$ is of type I.
\end{itemize}

Notice that if $n$ is of type I then both $\kappa_n$ and $\kappa_{n+1}$ belong to $\mathcal{C}$. For $\kappa_{n+1}$, it follows directly from the definition of type I, and for $\kappa_{n}$ it follows from \ref{item:no_abrupt_change} in \cref{lemma:properties_chart_near_singular}. Let us now define a family of $C^\infty$ maps $(F_n)_{n \leq -1}$ from $[- \delta_0,\delta_0]^3$ to $\mathbb{R}^3$ for some small $\delta_0 > 0$ (that can be chosen independent on $x_0$) and a sequence of sets $(H_n)_{n \leq 0}$:
\begin{itemize}
\item If $n$ is of type I, then both $\kappa_n$ and $\kappa_{n+1}$ are in $\mathcal{C}$ and thus define $C^\infty$ diffeomorphisms (uniformly in $x$ and $n$) between a neighbourhood of zero in $\mathbb{R}^3$ (of uniform size in $n$ and $x$) and a subset of $M$, and that $x_n$ and $x_{n+1}$ are uniformly away from the singular orbits of the flow. In that case, we just let $F_n$ be defined an a neighbourhood of zero by $F_n = \kappa_{n+1}^{-1} \circ \varphi_{T_0} \circ \kappa_{n}$, and we set $H_{n+1} = \mathbb{R}^3$.
\item If $n$ is of type II, then it follows from the construction of $\kappa_n$ that there are $(j,k,\tau),(j,k',\tau'') \in \mathcal{A}$ and $z,z' \in [ - \alpha_0,\alpha_0] \times [0,\alpha_0]$ such that $\kappa_n = \kappa_{j,k',\tau',z'}$ and $\kappa_{n+1} = \kappa_{j,k,\tau,z}$. Moreover, the relation \eqref{eq:explicit_transition} holds on $[ - \varpi, \varpi]^3$ intersected with $\Xi_{j,k,\tau,z}(\overline{\mathbb{H}} \times \mathbb{R})$. We define then $F_n$ by $F_n(x,t) = (\mathbf{F}_{j,k,\tau,z}(x),t - R_{j,k,\tau,z}(\mathbf{F}_{j,k,\tau,z}(x))$ and define $H_{n+1} = \Xi_{j,k,\tau,z}^{-1}(\overline{\mathbb{H}} \times \mathbb{R})$. 
\end{itemize}
Notice that $F_n$ always extends as a $C^\infty$ map on a neighbourhood of $0$ in $\mathbb{R}^3$ (uniformly in $x$ and $n$). Moreover, the domain of $\kappa_{n+1}$ always contains a neighbourhood of $0$ in $H_{n+1}$ and for $(x,t)$ there we have $\varphi_{-T_0} \circ \kappa_{n+1}(x,t) = \kappa_{n} \circ F_n^{-1}(x,t)$. Moreover, if $H_{n+1} \neq \mathbb{R}^3$, it follows from point \ref{item:coordinates_local_weak_unstable} in \cref{lemma:properties_chart_near_singular} that the image of $\partial H_{n+1}$ by $\kappa_{n+1}$ is the intersection of $W_{\textup{sing}}^{\textup{lu}}$ with the range of $\kappa_{n+1}$.

We want to apply a graph transform argument to $(F_n)_{n \leq -1}$, so we start by checking that they satisfy a hyperbolicity estimates.

\begin{lemma}\label{lemma:derivative_graph_transform}
For every $n \leq -1$, there are $a_n,b_n \in \mathbb{R}$ such that
\begin{equation}\label{eq:derivative_flow_coordinates}
DF_n(0) = \left[ \begin{array}{ccc}
a_n & 0 & 0 \\
0 & b_n & 0 \\
0 & 0 & 1
\end{array} \right].
\end{equation}
Moreover, for every $n \leq -L$,
\begin{equation*}
\prod_{p = 0}^{L-1} |a_{n+p}| \geq \mathfrak{h}^{-L} \textup{ and } \prod_{p = 0}^{L-1} |b_{n+p}| \leq \mathfrak{h}^L.
\end{equation*}
\end{lemma}

\begin{proof}
It follows from the invariance of the unstable, stable and neutral directions (eventually continuously extended at $0$ in a half-space chart) under the action of the flow that the matrix of $DF_n(0)$ is diagonal. Since in addition $\kappa_n$ and $\kappa_{n+1}$ are flow boxes, we find that $DF_n(0)$ indeed has the form \eqref{eq:derivative_flow_coordinates}.

Let us estimate now $A_n = \prod_{p = 0}^{L-1} |a_{n+p}|$. The other product may be estimated similarly. We split the orbit of $\varphi_{T_0}$ from $x_n$ to $x_{n+L}$ in three. Let $L_1$ be the largest integer such that $n, n +1,\dots, n + L_1 - 1$ are of type II ($L_1 = 0$ if $n$ is of type I). We let then $L_2$ be the largest integer such that $n + L_1 + L_2 - 1$ is of type I. Finally, we set $L_3 = L - L_2 - L_1$. The idea of the proof is that hyperbolicity follows from the local model when the orbit stays close to the singular orbits and from \ref{item:hyperbolic} in \cref{definition:pseudo_Anosov} for long orbits with endpoints away from the singular orbits. Hence, we decomposed our orbit in three (potentially empty) parts on which we can prove hyperbolicity: the extremal parts are close to a singular orbit and the middle part has its endpoints away from the singular orbits.

If $p \in \set{0,\dots, L_1 - 1} \cup \set{L_1 + L_2,\dots, L - 1}$, then we know that $n+p$ is of type II. Consequently $a_{n+p} = \lambda_{j,k,\tau,z}$ and $b_{n+p} = \mu_{j,k,\tau,z}$ for some $(j,k,\tau) \in \mathcal{A}$ and $z \in [ - \alpha_2,\alpha_2] \times [0,\alpha_2]$. Hence, it follows from \cref{eq:hyperbolicity_singular} that $|a_p| \geq (1 + \lambda_{\min})/2$.

Notice that $D(F_{n+ L_1 + L_2 - 1} \circ \dots \circ F_{n+ L_1}) (0) = D \kappa_{n + L_1 + L_2}(0)^{-1} \circ D \varphi_{L_2 T_0}\circ D \kappa_{n + L_1}(0)$. Thus
\begin{equation*}
\begin{split}
\left|\prod_{p = L_1}^{L_1 + L_2 - 1} a_p \right| & = |D \kappa_{n + L_1 + L_2}(0)^{-1} \circ D \varphi_{L_2 T_0}  (\kappa_{n+L_1}(0)) \circ D \kappa_{n + L_1}(0) \cdot (1,0,0)| \\
    & \geq C_0^{-1} e^{\nu L_2 T_0} B_1^{-2}.
\end{split}
\end{equation*}
Here, we used the fact that $\kappa_{n+L_1}$ and $\kappa_{n+L_1+L_2}$ belongs to $\mathcal{C}$ to get that the images of $0$ by these maps are not in $U_0$ and apply the hyperbolicity estimate for the flow.

Hence, we have
\begin{equation*}
\begin{split}
A_n & = \prod_{p = 0}^{L_1 - 1} a_{n+p} \left| \prod_{p = L_1}^{L_1 + L_2 - 1} a_{n+p}\right| \prod_{p = L_1 + L_2}^{L-1} a_{n+p} \\
     & \geq \left( \frac{1+ \lambda_{\min}}{2}\right)^{L_1 + L_3} C_0^{-1} B_1^{-2} e^{\nu L_2 T_0} \geq \mathfrak{h}^{-L} >1.
\end{split}
\end{equation*}
\end{proof}

\begin{remark}\label{remark:trick_hyperbolicity}
One could wonder if the proof of Lemma \ref{lemma:derivative_graph_transform} implies some uniform hyperbolicity estimate on $\widetilde{M}$. This is not the case : $\widetilde{M}$ being a non-compact manifold, we could adjust the Riemannian metric on $\widetilde{M}$ to prevent the existence of a hyperbolicity estimate. What happens in Lemma \ref{lemma:derivative_graph_transform} is that we get a hyperbolicity estimate by working with smooth charts away from the singular orbits (hence on a compact subset of $\widetilde{M}$) and in the local model near the singular orbits. The key point to get this estimate is that in Definition \ref{definition:pseudo_Anosov} we require uniform hyperbolicity estimates for orbits whose endpoints lie in any given compact subsets of $\widetilde{M}$, but the orbits themselves are allowed to get arbitrarily close to the singular orbits. Hence, any orbit may be split into a finite number of parts (actually three) that either lie within a neighbourhood of the singular orbits (which allows us to get hyperbolicity from the local model directly) or have their endpoints away from the singular orbits (so that we also have a hyperbolicity estimates for those as explained before).
\end{remark}

Then, let $\chi : \mathbb{R}^3 \to [0,1]$ be a $C^\infty$ function identically equal to $1$ on $[-1,1]^3$ and vanishing outside of $[-2,2]^3$ and define for some $\delta_1 >0$ very small the map
\begin{equation}\label{eq:extension_diffeo}
\widetilde{F}_n : z \mapsto \chi(\delta_1^{-1} z) F_n(z) + (1 - \chi(\delta_1^{-1} z)) D F_n(0) \cdot z.
\end{equation}
Then, for $\delta_1$ small enough (uniform in $n$ and $x$), $\widetilde{F}_n$ is a diffeomorphism from $\mathbb{R}^3$ to itself that coincides with $F_n$ on $[-\delta_1,\delta_1]^3$.

Moreover, provided $\delta_1$ is small enough, $\widetilde{F}_n$ is $C^1$ close to $DF_n(0)$ and consequently, one can apply a graph transform argument to construct unstable manifolds for the family of maps $(\widetilde{F}_n)_{n \leq - 1}$. See for instance the statement of the Hadamard--Perron Theorem in \cite[Theorem 6.2.8]{katok_hasselblatt_book}. Technically speaking, the version of the Hadamard--Perron Theorem from this reference does not apply directly here, let us explain how to bypass these technicalities:
\begin{itemize}
\item the $\widetilde{F}_n$ are only defined for $n \leq -1$, but it does not matter since we only are constructing unstable manifolds (we could set $\widetilde{F}_n = DF_{-1}(0)$ for $n \geq 0$ if we really want to have it defined);
\item with the notation from the reference, we have $\lambda = 1$, so that as stated \cite[Theorem 6.2.8]{katok_hasselblatt_book} does not give the smoothness of the unstable manifolds, but see the beginning of the step 5 of the proof (p. 254);
\item the family of maps $(\widetilde{F}_n)_{n \leq -1}$ does not satisfy directly the required hyperbolicity conditions: we need to iterate in \cref{lemma:derivative_graph_transform} to witness the hyperbolicity. But this issue is not dynamically relevant. One can reduce to the setting of \cite[Theorem 6.2.8]{katok_hasselblatt_book} for instance by using an $n$-dependent norm (but with a dependence on $n$ that is periodic of period $L$). Another solution is to apply the result to the families of maps $(F_{nL + p} \circ \dots \circ F_{nL + p - L + 2} \circ F_{nL +p - L +1})_{n \leq - 1}$ for $p = 0,\dots, L-1$.
\end{itemize}
From the Hadamard--Perron Theorem, we get for each $n \leq 0$ a $C^\infty$ curve $\widetilde{W}_n$ such that
\begin{itemize}
\item $\widetilde{W}_n$ is a graph over the unstable direction
\begin{equation*}
\widetilde{W}_n = \set{(x, f_n(x), g_n(x)) : x \in \mathbb{R}},
\end{equation*}
where $f_n$ and $g_n$ are $C^\infty$ (locally uniformly in $x_0$ and $n$);
\item $\widetilde{F}_n(\widetilde{W}_n) = \widetilde{W}_{n+1}$ for $n \leq - 1$;
\item there is $\varrho < 1$ such that for every $z \in \widetilde{W}_n$ we have $$\n{\widetilde{F}_{n -L + 1}^{-1} \circ \dots \circ \widetilde{F}_{n-1}^{-1} \circ \widetilde{F}_n^{-1} z} \leq \varrho^L \n{z};$$
\item if $z \in \mathbb{R}^3$ and $m \leq 0$ are such that there is $C > 0$ such that $$\n{\widetilde{F}_{m-n}^{-1} \circ \dots \circ \widetilde{F}_{m-1}^{-1}(z)} \leq C \varrho^{n}$$ for every $n \geq 0$, then $z \in \widetilde{W}_m$.
\end{itemize}

We want now to use these manifolds to understand the local unstable manifold of $x_0$. Let us pick some very small $\delta_2 \in (0,\delta_1)$ and define $W_n = \widetilde{W}_n \cap [- \delta_2,\delta_2]^3$ for $n \leq 0$. Notice that $W_n$ is contained in the domain where $F_n$ and $\widetilde{F}_n$ coincide.

\begin{lemma}\label{lemma:manifold_makes_sense}
Provided $\delta_2$ is small enough (uniformly in $x_0$), then for every $n \leq 0$ the manifold $W_n$ is contained in $H_n$ (and thus in the domain of $\kappa_n$).
\end{lemma}

\begin{proof}
It follows from \cref{eq:derivative_flow_coordinates} that for every $n \leq 0$, the vector $(1,0,0)$ is tangent to $W_n$ at $0$. Hence, provided $\delta_2$ is small enough, the manifold $W_n$ is connected. Consequently, if there is $n \leq 0$ such that $W_n$ is not contained in $H_n$, then there is a point $y_n \in W_n \cap \partial H_n$. Define for $m < n$ the point $y_m = \widetilde{F}_m^{-1} \circ \dots \circ \widetilde{F}_{n-1}^{-1} y_n$. It follows then from the properties of the $\widetilde{W}_m$'s that there is a constant $C > 0$ (that does not depend on $x_0$ nor $n$) such that for every $m \leq n$ we have $y_m \in [- C \delta_2,C \delta_2]^3$ and that $y_m \underset{m \to - \infty}{\to} 0$. Notice also that, since we clearly have $H_n \neq \mathbb{R}^3$, we must have $\kappa_n(y_n) \in W_{\textup{sing}}^{\textup{lu}}$.

Let us prove by induction that for every $m \leq n$ the point $y_m$ belongs to $ H_m$. This is clearly true for $m = n$. Now, assume that $y_m,\dots,y_n$ belong respectively to $H_m,\dots,H_n$ and let us prove that $y_{m-1}$ belongs to $H_{m-1}$. If $H_{m-1} = \mathbb{R}^3$, this is clear, so assume that $H_{m-1} \neq \mathbb{R}^3$. Since $y_m,\dots,y_n$ belong respectively to $H_m,\dots,H_n$, and are all in a small neighbourhood of zero, we find that $y_{m-1}$ belongs to the domain of $\kappa_{m-1}$ and that $\kappa_{m-1}(y_{m-1}) = \varphi_{(m-1-n) T_0} \circ \kappa_n(y_n) \in W_{\textup{sing}}^{\textup{lu}}$. Here we used that the local unstable manifold of the singularities is backward invariant by the flow. Using point \ref{item:coordinates_local_weak_unstable} in \cref{lemma:properties_chart_near_singular}, we find that $y_{m-1} \in \partial H_{m-1} \subseteq H_{m-1}$.

Thus, we know that for every $m \leq n$, we have $y_m \in H_m$, and it implies, provided $\delta_2$ is small enough, that, for every $m \leq n$, the point $y_m$ belongs to the domain of $\kappa_m$ with $\kappa_m(y_m) = \varphi_{(m-n) T_0} \circ \kappa_{n}(y_n)$. Using $\kappa_n(y_n) \in W_{\textup{sing}}^{\textup{lu}}$, the limit $y_m \underset{m \to - \infty}{\to} 0$, the fact that the $\kappa_m$'s are uniformly continuous and \cref{lemma:actually_unstable}, we find that $x_n \in W_{\textup{sing}}^{\textup{lu}}$. Since $H_n \neq \mathbb{R}^3$, we must consequently have $0 \in \partial H_n$. It implies that $\kappa_n$ is of the form $\kappa_{j,k,\tau,z}$ with $z \in [ - \alpha_3,\alpha_3] \times \set{0}$. It follows that $y_{j,k,\tau}^u(z) = 0$ (because the unstable direction of the flow is supposed to be tangent to the local weak unstable of the local model for the singular orbits), and thus that $\partial H_n = \mathbb{R} \times \set{0} \times \mathbb{R}$. One can then prove by induction using the formula for the $\widetilde{F}_m^{-1}$'s near $0$ that for every $m < n$, the integer $m$ is of type II and $\partial H_m = \mathbb{R} \times \set{0} \times \mathbb{R}$. Notice then that the maps $(\widetilde{F}_m)_{m < n}$ let the plane $\mathbb{R} \times \set{0} \times \mathbb{R}$ invariant and that their restrictions to this plane satisfy the hyperbolicity conditions required to construct a family of unstable manifolds by a graph transform argument (we just removed one directions that was not expanded). This produces a new family of unstable manifolds for the $\widetilde{F}_m$'s with $m < n$, but these new manifolds are contained in $\mathbb{R} \times \set{0} \times \mathbb{R}$. These new unstable manifolds have to coincide with the old ones (for instance because the constant $\varrho$ appearing in these two constructions can be chosen to be the same). In particular, we find that $W_n \subseteq \partial H_n \subseteq H_n$, a contradiction.
\end{proof}

Let us now consider the ``manifold'' $W = \kappa_0(W_0)$, which is well-defined thanks to \cref{lemma:manifold_makes_sense}. Notice that $W$ is smooth unless it contains a point on a singular orbit, and that in that case it is smooth in a half-space parametrization. By comparing $W$ to the local unstable manifold of $x_0$, we will prove most of \cref{proposition:unstable_manifolds}. Let us start by proving that $W$ is tangent to the unstable direction (away from the singular orbits and in half-space parametrizations).

\begin{lemma}\label{lemma:manifold_tangent_unstable}
Provided $\delta_2$ is small enough, for every $n \leq 0$, the curve $W_n$ is tangent to the direction $\widetilde{E}_{u,n}$ obtained by pulling-back $E_u$ by $\kappa_{n}$ (and eventually extended by continuity).
\end{lemma}

\begin{proof}
Pick a point $y \in W_n$, and define for $m \leq 0$ the point $y_m = \widetilde{F}_{n+ m}^{-1} \circ \dots \circ \widetilde{F}_{n-1}^{-1} y$. Notice that for every $p \leq 0$, the point $y_{n+p}$ belongs to $\widetilde{W}_{n+p} \cap [- C\delta_2,C\delta_2]^3$, for some $C > 0$ uniform in $x_0,y,n$ and $p$ (as follows from the properties of the $\widetilde{W}_m$'s). 

For every $p \leq 0$, let us consider the map $G_p = \widetilde{F}_{n-1 + p L} \circ \dots \circ \widetilde{F}_{n + (p-1) L}$. From the relation $G_p ( \widetilde{W}_{n + (p-1) L}) = \widetilde{W}_{n+ p L}$, we deduce that for every $x \in \mathbb{R}$ we have
\begin{equation*}
G_p(x, f_{n+ (p-1)L} (x) , g_{n+ (p-1) L} (x)) = (h_p(x), f_{n+ pL}(h_p(x)), g_{n+ pL}(h_p(x)))
\end{equation*}
for some $C^\infty$ map $h_p : \mathbb{R} \to \mathbb{R}$ such that $h_p(0) = 0$, that is smooth and expanding (uniformly in $p$). Thus, we have for every $x \in \mathbb{R}$
\begin{equation*}
\begin{split}
& h_p'(h_p^{-1}(x)) (1, f_{n + p L}'(x) , g_{n + p L}'(x)) \\ & \qquad \qquad = D G_p (h_p^{-1}(x), f_{n + (p+1)L} (h_p^{-1}(x)), g_{n + (p+1)L}^{-1}(x)) \\ & \qquad \qquad \qquad \qquad \qquad \qquad \qquad \qquad \qquad \qquad \cdot (1, f_{n + (p+1)L}'(x), g_{n + (p+1)L}'(x))
\end{split}
\end{equation*}
Writing $y_p = (z_p, f_{n + p L} (z_p) , g_{n + p L}(z_p))$, we find iterating the relation above that
\begin{equation}\label{eq:iteration_hyperbolic}
(1 , f_n'(z_0), g_n'(z_0) ) = \frac{D G_1( y_{-1}) \circ \dots \circ D G_{p}(y_{p}) \cdot (1, f_{n+ p L }'(z_p), g_{n + p L}'(z_p))}{h_1'(z_{-1}) \dots h_{p }'(z_p)}.
\end{equation}

Let then $(1, v_p,w_p)$ and $(\tilde{v}_p, 1, \tilde{w}_p)$ span respectively $\widetilde{E}_{u,n + p L}$ and $\widetilde{E}_{s,n+ p L}$ at the point $y_p \in E_{n + p L} \cap W_{n + p L}$ (when $\delta_2$ is small enough, we can normalize these vectors in this way). Write 
\begin{equation*}
(1, f_{n+ p L }'(z_p), g_{n + p L}'(z_p)) = a_p (1, v_p, w_p) + b_p (\tilde{v}_p, 1, \tilde{w}_p) + c_p (1,0,0)
\end{equation*}
and notice that the coefficients $a_p, b_p$ and $c_p$ must be bounded uniformly in $p$ (because of the hyperbolic splitting, \cref{lemma:continuity_unstable} and the assumption that the stable and unstable directions are continuous in half-space parametrizations). It follows then from \cref{eq:iteration_hyperbolic} that
\begin{equation*}
(1, f_n'(z_0), g_n'(z_0)) = \frac{\tilde{a}_p ( 1 , v_0, w_0) + \tilde{b}_p (\tilde{v}_0,1,\tilde{w}_0) + c_p(1,0,0)}{h_0'(z_{-1}) \dots h_{p + 1}'(z_p)},
\end{equation*}
for some coefficients $\tilde{a}_p,\tilde{b}_p$. Working as in the proof of \cref{lemma:derivative_graph_transform}, we deduce from the hyperbolicity of the flow that $\tilde{b}_p$ decays exponentially fast with $p$. We must consequently have $\tilde{a}_p/ h_0'(z_{-1}) \dots h_{p + 1}'(z_p) \underset{ p \to - \infty}{\to} 1$, and it follows that $(1, f_n'(z_0), g_n'(z_0)) = (1,v_0,w_0)$.
\end{proof}

\begin{lemma}\label{lemma:knot}
Assume that $\delta_2$ is small enough. Let $V$ be a neighbourhood of the singular orbits of the flow. There is $\epsilon_1 > 0$ (that does not depend on $x_0$) such that for every $\epsilon \in (0,\epsilon_1)$ the intersection of $W_\epsilon^u(x_0)$ with $\kappa_0([- \delta_2,\delta_2]^3 \cap H_0)$ is a neighbourhood of $x$ in $W$ (of uniform size). Moreover, if $x_0$ does not belong to $V \cap W_{\textup{sing}}^{\textup{lu}}$ then $W_\epsilon^u(x)$ is contained in $\kappa_0([- \delta_2,\delta_2]^3 \cap H_0)$ (and is thus a neighbourhood of $x_0$ in $W$).
\end{lemma}

\begin{proof}
It follows from the properties of the $\widetilde{W}_n$'s that for every $\epsilon > 0$ the set $W_\epsilon^u(x_0)$ contains a neighbourhood of $x_0$ in $W$. Let us work around the reciprocal inclusion.

Take a point $y \in W_\epsilon^u(x)$ and assume that $y$ belongs to $\kappa_0([-\delta_2,\delta_2]^3 \cap H_0)$ or that $x_0$ does not belong to $V \cap W_{\textup{sing}}^{\textup{lu}}$. We want to prove first that for every $n \leq 0$, the point $\varphi_{nT_0}(y)$ belongs to $\kappa_n([- \delta_2,\delta_2]^3 \cap H_n)$. By contradiction, assume that there is $n \leq 0$ such that $\varphi_{nT_0} (y) \notin \kappa_n([- \delta_2,\delta_2]^3 \cap H_n)$. By taking $\epsilon$ small enough, we can make sure that it can only happen when $H_n \neq \mathbb{R}^3$ (in particular $n-1$ is of type II). In that case, $\kappa_n([- \delta_2,\delta_2]^3 \cap H_n)$ is the intersection of a neighbourhood of $x_n$ (of uniform size in $n$ and $x_0$) with $\Upsilon_{j,k,\tau}(D_{\alpha_0})$ for some $(j,k,\tau) \in \mathcal{A}$. Hence, by taking $\epsilon$ small enough, we ensure that $\varphi_{nT_0}(y) \notin \Upsilon_{j,k,\tau}(D_{\alpha_0})$. It follows then from \cref{lemma:unstable_chambers} that the points $x_n$ and $\varphi_{nT_0} y$ belong to $W_j^{\textup{lu}}$. We can then apply \cref{lemma:unstable_chambers} again but with time reversed (replacing the unstable direction by the stable direction, we use the fact that $x_n$ and $\varphi_{nT_0} (y)$ belong to the local unstable manifold of the singular orbit but not to the same prong to check the assumption that they are in different odd-indexed angular sector). If $x_n$ and $\varphi_{nT_0} (y)$ belong to the local stable manifold of $\gamma_j$, then they must belong to the singular orbit itself and we contradict $\varphi_{nT_0} (y) \notin \Upsilon_{j,k,\tau}(D_{\alpha_0})$. Hence, we find that there is some time $t_0$ such that $\varphi_{t_0} (x_n)$ and $\varphi_{nT_0 + t_0}( y)$ are at distance more than $2 \epsilon$ (which imposes $t_0 \geq - n T_0$ since $y \in W_\epsilon^u(x)$) and that $\varphi_t(x_n)$ and $\varphi_{nT_0 + t}(y)$ remain in an arbitrarily small neighbourhood of the singular orbits for $t \in [0,t_0]$. Provided that $\epsilon$ is small enough, we find that $x_0,y \in W_j^{\textup{lu}}$ and that the integers $-1,-2,\dots, n$ are all of type II. Consequently, $y \in \kappa_0([- \delta_2,\delta_2]^3 \cap H_0)$, and we can write $y = \kappa_0(a,b,c)$ but since $ x,y \in W_j^{\textup{lu}}$, we must have $b = 0$. Thus, we find by induction, as in the proof of \cref{lemma:manifold_makes_sense}, that $\varphi_{nT_0}(y) \in \kappa_n([- \delta_2,\delta_2]^3 \cap H_n)$, a contradiction.

Now that we know that $\varphi_{nT_0}(y) \in \kappa_n([-\delta_2,\delta_2]^3 \cap H_n)$ for every $n \leq 0$, let us write $\varphi_{nT_0}(y) = \kappa_n(a_n,b_n,c_n)$ with $(a_n,b_n,c_n) \in [-\delta_2,\delta_2]^3 \cap H_n$. Let $p$ denote the projection $\mathbb{R}^3 \to \mathbb{R}^2$ obtained by removing the last coordinate. Consider then the family of maps $(\mathcal{F}_n)_{n \leq -1}$ from $\mathbb{R}^2$ to $\mathbb{R}^2$ defined by $\mathcal{F}_n(a,b) = p(\widetilde{F}_n(a,b,0))$. Using \cref{lemma:derivative_graph_transform}, we find that we can apply a graph transform argument again \cite[Theorem 6.2.8]{katok_hasselblatt_book} to construct a family of unstable manifolds $(\widehat{W}_n)_{n \leq 0}$ for the family $(\mathcal{F}_n)_{n \leq - 1}$. Moreover, using the relation $\widetilde{F}_n(a,b,c+t) = \widetilde{F}_n(a,b,c)$, which is valid for every $n \leq -1$ and $(a,b,c,t) \in \mathbb{R}^4$, we find that $\widehat{W}_n = p(\widetilde{W}_n)$ for every $n \leq 0$. The advantage of the $\mathcal{F}_n$'s is that, since we removed the flow direction, we get a family of truly hyperbolic maps ($\lambda < 1$ and $\mu > 1$ in the notation of \cite[Theorem 6.2.8]{katok_hasselblatt_book}), which implies that if $z \in \mathbb{R}^2$ is such that the sequence $(\mathcal{F}_n^{-1} \circ \dots \circ \mathcal{F}_{-1}^{-1}(z))_{n \leq 0}$ is bounded, then $z \in \widehat{W}_0$. Since we have $\mathcal{F}_n(a_n,b_n) = (a_{n+1}, b_{n+1})$, we find that $(a_0,b_0) \in \widehat{W}_0$, and thus $\kappa_0^{-1}(y) \in p(\widetilde{W}_0)$. Using in addition that $\kappa_0^{-1}(y) \in H_0 \cap [- \delta_2,\delta_2]^3$, we find that there is $t' \in [ - 2 \delta_2,2 \delta_2]$ such that $\varphi_{t'}(y) \in W$. The fact that $t' \neq 0$ would contradict $d(\varphi_t (x),\varphi_t(y)) \underset{t \to - \infty}{\to} 0$. Hence, $t' = 0$ and $y \in W$.
\end{proof}

\subsubsection{Proof of \cref{proposition:unstable_manifolds}}\label{subsubsection:proof_unstable_manifolds}

With the analysis from \cref{subsubsection:graph_transform}, we have most of the tools to prove \cref{proposition:unstable_manifolds}.

With the notation from \cref{subsubsection:graph_transform}, notice that if $x_0$ is away from the singular orbits, then the map $\kappa_0$ is uniformly smooth near $0$. Item \ref{item:standard_away_singular} follows then from Lemmas \ref{lemma:manifold_tangent_unstable} and \ref{lemma:knot}, where we can take for $U_{\textup{reg}}$ any open subset of $M$ whose closure does not intersect the singular orbits, up to making $\epsilon$ smaller. Hence, we can take for $U_{\textup{sing}}$ an arbitrarily small neighbourhood of the singular orbits and still have \ref{item:covering}. The smallness of $U_{\textup{sing}}$ ensures that \ref{item:sing_close_singular} holds.

Let us now explain why \ref{item:unstable_parallel} and \ref{item:normal_parallel} hold. If $j \in \set{1,\dots,N}, k \in \mathbb{Z}/2 n_j \mathbb{Z}$ and $\tau \in [ - r_{\max}/2, r_{\max}/2]$ are such that $k$ is even and $x \in \Upsilon_{j,k,\tau}(D_{\delta/2})$, write $x = \Upsilon_{j,k,\tau}(z,t')$. Notice then that, provided $\delta$ is small enough, $(j,k,\tau + t') \in \mathcal{A}$ and, in the construction from \cref{subsubsection:graph_transform}, we can take $\kappa_0 = \kappa_{j,k,\tau + t',z}$. The points \ref{item:unstable_parallel} and \ref{item:normal_parallel} are then deduced from Lemmas \ref{lemma:manifold_tangent_unstable} and \ref{lemma:knot} again (along with the uniformity from \ref{item:uniform_graph}).

Finally, item \ref{item:unstable_perpendicular} follows from the previous points using item \ref{item:smooth_rotation} from \cref{definition:standard_fixed_point}.

\subsection{Continuity property of local unstable manifolds}\label{subsection:continuity_unstable_manifold}

We close this section by proving the continuity properties of local unstable manifold that we stated in Propositions \ref{proposition:continuity_unstable_manifolds_regular}, \ref{proposition:continuity_unstable_manifolds_singular_parallel} and \ref{proposition:continuity_unstable_manifolds_singular_perpendicular}. These properties are needed in particular to define the Bowen bracket in \cref{subsection:bowen_bracket}.

\begin{proof}[Proof of \cref{proposition:continuity_unstable_manifolds_regular}]

The first point follows from \cref{proposition:unstable_manifolds} (we need to take $\delta_0$ small enough to ensure that $x_0 \in U_{\textup{reg}}$) and the implicit function theorem. The main point is to prove the continuity of $(t,x) \mapsto (F_x(t), F_x'(t))$. Since $D\kappa(0) (\mathbb{R} \times \set{0}) = E_{x_0}^u$, for $x$ in a neighbourhood of $x_0$, there is $G(x) \in \mathbb{R}^2$ such that $(1,G(x))$ spans $D \kappa^{-1}(x) \cdot E_x^u$. Moreover, it follows from \cref{lemma:continuity_unstable} that the function $G$ is continuous. Since the unstable manifolds are tangent to the unstable direction, we find that for $x$ close to $x_0$ and $t$ close to $0$ we have, writing $\kappa^{-1}(x) = (t_x,y_x)$,
\begin{equation}\label{eq:cauchy_problem_unstable}
\begin{cases}
F_x'(t) = G(\kappa(t,F_x(t))); \\
F_x(t_x) = y_x.
\end{cases}
\end{equation}
From the differential equation satisfied by $F_x$, we deduce that we only need to prove the continuity of $(t,x) \mapsto F_x(t)$. Moreover, it follows from \cite[Theorem 2.1, chapter V]{hartman_book} that we only need to prove that for $x$ near $x_0$, the function $F_x$ is the only local solution to the Cauchy problem \eqref{eq:cauchy_problem_unstable}.

So, let $f_x$ be another solution to this Cauchy problem. Then, with $J$ a small open interval in $\mathbb{R}$ containing $t_x$, we find that $\kappa(\set{(t,f_x(t)): t \in J})$ is a small curve in $M$ that is tangent to the unstable direction and contains the point $x$. It follows then from the mean value theorem that a neighbourhood of $x$ in this curve must belong to $W_\epsilon^u(x)$, which forces $f_x = F_x$ near $t_x$. One needs to use half-space parametrizations to deal with times $t$ such that $\varphi_{-t}(x)$ is near a singular orbit. To do so, notice that it follows from Lemma \ref{lemma:derivative_graph_transform} that the unstable direction is also contracted by the flow in large backward time in half-space parametrizations.
\end{proof}

The proofs of Propositions \ref{proposition:continuity_unstable_manifolds_singular_parallel} and \ref{proposition:continuity_unstable_manifolds_singular_perpendicular} are based on the same argument as the proof of \cref{proposition:continuity_unstable_manifolds_regular}. One just needs to work in half-space parametrizations from the beginning.

\begin{remark}
One could deduce the continuity of the dependence of $W_\epsilon^u(x)$ on $x$ (Propositions \ref{proposition:continuity_unstable_manifolds_regular}, \ref{proposition:continuity_unstable_manifolds_singular_parallel} and \ref{proposition:continuity_unstable_manifolds_singular_perpendicular}) directly from the graph transform argument that is used to prove \cref{proposition:unstable_manifolds}. The main idea is that, once the smoothness and the hyperbolicity of the maps appearing in this argument are fixed, we only need a finite number of them to get the unstable manifolds up to a fixed precision (of course, the number that is needed increases if we ask for a better accuracy) and perturbing a little these maps do not change much the resulting manifold. However, we were not aware of a reference detailing this argument in a way that would have allowed us to quote it directly here, and thus we used a more indirect but shorter argument.
\end{remark}

\section{Markov partitions for smooth pseudo-Anosov flows}\label{section:Markov_partition}

In this section, we explain how to construct a Markov partition for a transitive smooth pseudo-Anosov flow $\varphi$ on a manifold $M$. The flow $\varphi$ and the manifold $M$ are fixed for the whole section. We will use notations from \cref{section:definition_pseudo_Anosov}. In particular, the singular orbits of $\varphi$ will be denoted by $\gamma_1,\dots,\gamma_N$. This is also a standard topic in the Anosov case, see for instance \cite{bowen_markov}.

We start by introducing a Bowen bracket in \cref{subsection:bowen_bracket}. Then, we define in \cref{subsection:rectangles} the class of rectangles that appear in the Markov partition that we construct in \cref{subsection:Markov_partition}. 

\subsection{Bowen bracket}\label{subsection:bowen_bracket}

A first consequence of \cref{proposition:unstable_manifolds} is that we can define the Bowen bracket as detailed in the following lemma. The main difference with the Anosov flow case is that if $x$ and $y$ are two points close to a singular orbit, then $[x,y]$ may not be defined, even if $x$ and $y$ are very close see \cref{fig:bowen_bracket}.

\begin{figure}[h]
	\centering
\def\svgwidth{2in}
	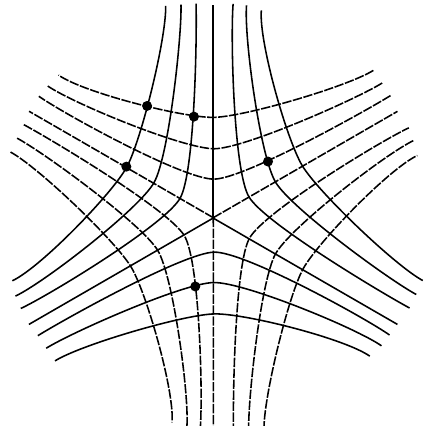
	\caption{The flow is coming toward us. The unstable foliation is drawn with dotted lines and the stable foliation is drawn with solid lines. The Bowen bracket of $x$ and $y$ exists, but the Bowen bracket of $x'$ and $y'$ does not exist.}\label{fig:bowen_bracket}
\end{figure}

\begin{lemma}\label{lemma:bowen_bracket}
Let $\epsilon_0 > 0$ be sufficiently small. For every $x,y \in M$, there is at most one point $z \in W_{\epsilon_0}^s(x)$ such that there is $t \in [ - \epsilon_0,\epsilon_0]$ such that $\varphi_{t}(z) \in W_{\epsilon_0}^u(y)$. When this point exists, we call it $[x,y]$ and:
\begin{enumerate}[label=(\roman*)]
\item if $F$ is a closed set in $M$ that does not intersect the singular orbits of $\varphi$, then there is $\delta' > 0$ such that for every $x,y \in F$ with $d(x,y) \leq \delta'$ the point $[x,y]$ exists;
\item \label{item:bracket_near_singular} let $\delta' > 0$, there is $\delta'' > 0$ such that for every $j \in \set{1,\dots,N}, k \in \mathbb{Z}/2 n_j \mathbb{Z}$ and $ \tau \in [- r_{\max},r_{\max}]$, if $x,y \in \Upsilon_{j,k,\tau}([- \delta'',\delta''] \times [0,\delta''] \times [ - \delta'',\delta''])$, then the point $[x,y]$ exists and belong to $\Upsilon_{j,k,\tau}([- \delta',\delta'] \times [0,\delta'] \times [ - \delta',\delta'])$.
\end{enumerate}
Moreover, the map $(x,y) \mapsto [x,y]$ is continuous on its domain of definition.
\end{lemma}

\begin{proof}
The continuity follows from Propositions \ref{proposition:continuity_unstable_manifolds_regular}, \ref{proposition:continuity_unstable_manifolds_singular_parallel} and \ref{proposition:continuity_unstable_manifolds_singular_perpendicular} and a compactness argument once uniqueness is established.

Away from the singular orbits, the result (existence and uniqueness) follows from the continuity of the stable and unstable manifolds in the $C^1$ topology (\cref{proposition:continuity_unstable_manifolds_regular}) and a transversality argument as in the case of Anosov flows (see for instance \cite{fisher_hasselblatt_book}).

In order to prove that in the situation from \ref{item:bracket_near_singular} the point $[x,y]$ exists, one only needs to use the standard transversality argument (from the Anosov flow case) in the half-space parametrization $\Upsilon_{j,k,\tau}$. This argument proves the existence of $[x,y]$ in this case, but also that there is at most one point in the range of $\Upsilon_{j,k,\tau}$ satisfying the properties required for $[x,y]$.

It remains to prove that if $x$ and $y$ are two points in $M$ close to a singular orbit then there is at most one point satisfying the definition of $[x,y]$. Assume first that $y$ does not belong to the local unstable manifold of the singular orbit. Then, with $\delta$ as in \cref{proposition:unstable_manifolds}, if $y$ is close enough to a singular orbit, there is $(j,k,\tau) \in \mathcal{A}$ such that $y \in \Upsilon_{j,k,\tau}([ - \delta/2,\delta/2] \times [0,\delta/2] \times [ - \delta/2,\delta/2])$, and thus $W_{\epsilon_0}^u(y)$ is contained in $\Upsilon_{j,k,\tau}([ - \delta,\delta] \times [0,\delta] \times [ - \delta,\delta])$. Hence, a point satisfying the properties of $[x,y]$ must be in the range of $\Upsilon_{j,k,\tau}$, and thus there can only be one.

We can write a similar arguments under the assumption that $x$ does not belong to the local stable manifold of the singular orbit. It remains to deal with the case in which $x$ belongs to the local stable manifold of the singular orbit and $y$ to the local unstable manifold. In that case, $[x,y]$ must belong to the intersection of the local stable and unstable manifolds of the singular orbit, and must consequently be on the singular orbit. But we can use \cref{proposition:unstable_manifolds} to see that $W_{\epsilon_0}^s(x)$ has at most one point of intersection with the singular orbit.
\end{proof}

Using the Bowen bracket $[\cdot,\cdot]$, we prove that the stable and unstable manifolds of a periodic orbit for $\varphi$ are dense in $M$, a result that will be crucial to construct Markov partitions for the flow $\varphi$.

\begin{corollary}\label{corollary:dense}
Assume that $\varphi$ is transitive and that $\gamma$ is a periodic orbit for $\varphi$. Then the sets $W^{\textup{s}}(\gamma)$ and $W^{\textup{u}}(\gamma)$ from \cref{definition:full_manifolds_periodic} are dense in $M$.
\end{corollary}

\begin{proof}
Let $x \in M$. Choose a point $x'$ on $\gamma$. Let $\epsilon > 0$. According to \cite[Proposition 1.6.9]{fisher_hasselblatt_book}, $\varphi$ has a dense positive semi-orbit. Consequently, there is a point $y \in M$ and $T > 0$ such that $d(y,x') \leq \epsilon$ and $d(\varphi_T(y), x) \leq \epsilon$. Provided $\epsilon$ is small enough, it follows from \cref{lemma:bowen_bracket} that the point $z \coloneqq [y,x']$ is well-defined. Indeed, if $\gamma$ is not a singular orbit, we just take $\epsilon$ small enough to ensure that $y$ and $x'$ are away from all singular orbits, and, if $\gamma$ is a singular orbit, we take $\epsilon$ small enough to ensure that $y$ is in the image of a small neighbourhood of zero by a half-space parametrization.

It follows from the definition of $z$ that it belongs to $W^{\textup{u}}(\gamma)$, and thus so does $\varphi_T(z)$. Moreover, $z$ belongs to $W_{\epsilon_0}^s(y)$. Actually, it follows from the continuity of $[\cdot,\cdot]$ that $z \in W_{q(\epsilon)}^s(y)$ where the quantity $q(\epsilon)$ goes to $0$ with $\epsilon$. Thus, $d(\varphi_T(z), x) \leq d(\varphi_T(z),\varphi_T(y)) + d(\varphi_T(y),x) \leq \epsilon + q(\epsilon)$. Hence, $W^{\textup{u}}(\gamma)$ is dense in $M$. A similar argument deals with the case of $W^{\textup{s}}(\gamma)$.
\end{proof}

\subsection{Definition of rectangles}\label{subsection:rectangles}

We will define two notions of rectangle that will be used in the definition of a Markov partition: one for rectangle close to the singular orbits and the other for rectangles away from the singular orbits. In the following, we let $\epsilon_0 > 0$ be fixed small enough so that \cref{lemma:bowen_bracket} holds and smaller than the third of the length of the smallest periodic orbit of $\varphi$.

\begin{definition}\label{definition:regular_rectangle}
We say that a subset $R$ of $M$ is a \emph{regular rectangle} if:
\begin{enumerate}[label=(\roman*)]
\item there is a smooth disk $D \subseteq M \setminus \bigcup_{j = 1}^N \gamma_j$, transverse to the flow, such that $R \subseteq D$;
\item for every $x,y \in R$, the point $[x,y]$ is defined and there is $t \in (-\epsilon_0,\epsilon_0)$ such that $\varphi_t([x,y]) \in R$;
\item there is a smooth function $t_R$ from a neighbourhood of $R$ to $(- \epsilon_0,\epsilon_0)$, segments $I$ and $J$ in $D$, tangent to $E_u \oplus E_0$ and $E_s \oplus E_0$ respectively, such that
\begin{equation*}
R = \set{\varphi_{t_R([x,y])}([x,y]): x \in I, y \in J}.
\end{equation*}
\end{enumerate}
With the notations above, if $R$ is a regular rectangle, we let $R^*$ denote the interior of $R$ as a subset of $D$. If $x_0,x_1$ and $y_0,y_1$ denote respectively the extremities of $I$ and $J$, we also define the unstable and stable boundaries of $R$ as
\begin{equation*}
\partial^u R = \set{\varphi_{t_R([x,y])}([x,y]): x \in I, y \in \set{y_0,y_1}}
\end{equation*}
and
\begin{equation*}
\partial^s R = \set{\varphi_{t_R([x,y])}([x,y]): x \in \set{x_0,x_1}, y \in J}.
\end{equation*}
\end{definition}

\begin{definition}\label{definition:singular_rectangles}
We say that a subset $R$ of $M$ is a \emph{singular rectangle} if there are $j \in \set{1,\dots,N}, k \in \mathbb{Z}/ 2 n_j \mathbb{Z}$, $\tau \in [ - r_{\max},r_{\max}]$, two segments $\widetilde{I}$ and $\widetilde{J}$ in $\mathcal{V}$ and $\widetilde{R} \subseteq \mathcal{V}$ such that
\begin{enumerate}[label=(\roman*)]
\item $R = \Upsilon_{j,k,\tau}( \widetilde{R} \times \set{0})$;
\item for every $x,y \in R$, the point $[x,y]$ is defined, belongs to the range of $\Upsilon_{j,k,\tau}$ and there is $t \in (-\epsilon_0,\epsilon_0)$ such that $\varphi_t([x,y]) \in R$;
\item $\widetilde{I} \times \set{0}$ and $\widetilde{J} \times \set{0}$ are tangent respectively to the pullbacks of $E_u \oplus E_0$ and $E_s \oplus E_0$ by $\Upsilon_{j,k,\tau}$ (continuously continued when $\Upsilon_{j,k,\tau}$ is singular);
\item there is a smooth function $t_R$ from a neighbourhood of $\widetilde{R} \times \set{0}$ in $\mathbb{R}^3$ to $(-\epsilon_0,\epsilon_0)$ such that
\begin{equation*}
R = \set{\varphi_{t_R(\Upsilon_{j,k,\tau}^{-1}[x,y])}([x,y]) : x \in I, y \in J}
\end{equation*}
where $I \coloneqq \Upsilon_{j,k,\tau}(\widetilde{I} \times \set{0})$ and $J \coloneqq \Upsilon_{j,k,\tau}(\widetilde{J} \times \set{0})$.
\end{enumerate}

With the notations above, if $R$ is a singular rectangle, then we let $R^*$ be the image by $\Upsilon_{j,k,\tau}$ of the interior of $\widetilde{R} \times \set{0}$ as a subset of $\mathbb{R}^2 \times \set{0}$. If $x_0,x_1$ and $y_0,y_1$ denote respectively the extremities of $I$ and $J$, we also define the unstable and stable boundaries of $R$ as
\begin{equation*}
\partial^u R = \set{\varphi_{t_R(\Upsilon_{j,k,\tau}^{-1}[x,y])}([x,y]): x \in I, y \in \set{y_0,y_1}}
\end{equation*}
and
\begin{equation*}
\partial^s R = \set{\varphi_{t_R(\Upsilon_{j,k,\tau}^{-1}[x,y])}([x,y]): x \in \set{x_0,x_1}, y \in J}.
\end{equation*}
\end{definition}

\begin{remark}
If $R$ is a singular rectangle and $0$ does not belong to the intervals $\widetilde{I}$ and $\widetilde{J}$ from \cref{definition:singular_rectangles}, then $R$ is also a regular rectangle.
\end{remark}

\begin{definition}
We say that a subset $R$ of $M$ is a rectangle if it is a regular rectangle or a singular rectangle. In that case, for $x,y \in R$, we define
\begin{equation*}
[x,y]_{R} = \varphi_t([x,y]),
\end{equation*}
where $t \in (\epsilon_0,\epsilon_0)$ is such that $\varphi_t([x,y]) \in R$. For $x \in R$, we define then
\begin{equation*}
W_R^s(x) = \set{[x,y]_R: y \in R} \textup{ and } W_R^u(x) = \set{[y,x]_R: y \in R}.
\end{equation*}
\end{definition}

\begin{remark}
Our definition of rectangles is more restrictive that definitions that can be found in the literature on general Anosov flows. We impose indeed that our rectangles are connected, which implies that they actually look like rectangles\footnote{This is of course a loose statement. In the case of a singular rectangle with a side that intersects a singular orbit not at a corner, the similarity with an actual rectangle in $M$ is subject to interpretation. However, it truly looks like a rectangle in the local model.}. This is essential for the topological applications in \cref{section:rtorsion}, and is possible thanks to the low dimension. The possibility to work with connected rectangles in dimension $3$ is a well-known feature of Anosov flows on $3$-manifolds, see \cite{ratner_markov,chernov_markov}. Notice also that our definition implies that if $R$ is a rectangle then $R = \overline{R}^*$, that is $R$ is proper.
\end{remark}

\begin{remark}\label{remark:rectangle_stable_unstable}
Notice that if $R$ is a small enough rectangle and $x \in R$ then
\begin{equation*}
W_R^s(x) = R \cap \bigcup_{t \in ( - \epsilon_0,\epsilon_0)} \varphi_t(W_{\epsilon_0}^s(x)) \textup{ and } W_R^u(x) = R \cap \bigcup_{t \in ( - \epsilon_0,\epsilon_0)} \varphi_t(W_{\epsilon_0}^u(x)).
\end{equation*}
In particular, $\partial^u R$ and $\partial^s R$ are the intersections with $R$ of pieces of unstable and stable manifolds respectively.
\end{remark}

\subsection{Construction of a Markov partition}\label{subsection:Markov_partition}

Now that we have defined the class of rectangles that we want to use, we are ready to construct a Markov partition. Let us start with a definition.

\begin{definition}\label{definition:Markov_partition}
Let $\alpha > 0$. We say that a finite family $(R_\theta)_{\theta \in \Theta}$ of rectangles is a Markov partition of size $\alpha$ if:
\begin{enumerate}[label=(\roman*)]
\item for every $\theta \in \Theta$, the set $R_\theta$ is a rectangle of diameter less than $\alpha$;
\item let $\Omega = \bigcup_{\theta \in \Theta} R_\theta$, we have\footnote{We use the notation $\varphi_ {[0,\alpha]}(\Omega) = \set{ \varphi_t(x) : t \in [0,\alpha], x \in \Omega}$.} $M = \varphi_{[0,\alpha]}(\Omega)$;
\item if $\theta_1,\theta_2$ are distinct elements of $\Theta$ then $\varphi_{[0,\alpha]} (R_{\theta_1}) \cap R_{\theta_2}$ or $R_{\theta_1} \cap \varphi_{[0,\alpha]} (R_{\theta_2})$ is empty;
\item let $T : \Omega \to \Omega$ denote the first return map of the flow, if $\theta_1,\theta_2 \in \Theta$ and $x \in R_{\theta_1}^*$ are such that $Tx \in R_{\theta_2}^*$ then $W_{R_{\theta_1}}^s(x) \subseteq \overline{T^{-1}(W_{R_{\theta_2}}^s(Tx))}$ and $W_{R_{\theta_2}}^u(Tx) \subseteq \overline{T(W_{R_{\theta_1}}^u(x))}$;\label{item:Markov_property}
\item for every $\theta_1,\theta_2 \in \Theta$ the set $\set{ x \in R_{\theta_1}^* : T x \in R_{\theta_2}^*}$ is connected. \label{item:stay_connected}
\end{enumerate}
\end{definition}

\begin{remark}
We follow here the definition that can be found in \cite{chernov_markov}. Notice however that our definition is more restrictive than the standard one since we are only working with connected rectangles. We are also requiring an extra property \ref{item:stay_connected}. This is because we may want to split the rectangles in the Markov partition into smaller rectangles, and this assumption ensures that the new rectangles are also connected.

The interested reader may refer to \cite{chernov_markov} for a discussion of Markov partitions for Anosov flows, with an emphasis on the possibility to get connected rectangle in the $3$-dimensional case. We will construct Markov partitions following the approach from \cite{ratner_markov} in the $3$-dimensional Anosov case. See \cite{bowen_markov} for a construction of Markov partitions for Anosov flows in any dimension. Since this kind of construction is somehow standard we may not go over all details in the proofs.
\end{remark}

The main result of this section is the following. Notice that we have some control on the stable and unstable boundaries in the Markov partition that we are going to construct. It will be essential in \cref{section:rtorsion}.

\begin{proposition}\label{proposition:existence_Markov}
Let $\Gamma^u$ and $\Gamma^s$ be two non-empty finite sets of primitive periodic orbits for $\varphi$ such that $\set{\gamma_1,\dots,\gamma_N} \subseteq \Gamma^u \cup \Gamma^s$. For every $\alpha > 0$, there is a Markov partition $(R_\theta)_{\theta \in \Theta}$ of size $\alpha$ such that
\begin{equation}\label{eq:assumption_boundary}
\bigcup_{\theta \in \Theta} \partial^u R_\theta \subseteq \bigcup_{\gamma \in \Gamma^u} W^{\textup{u}}(\gamma) \textup{ and } \bigcup_{\theta \in \Theta} \partial^s R_\theta \subseteq \bigcup_{\gamma \in \Gamma^s} W^{\textup{s}}(\gamma).
\end{equation}
Moreover, for every $\gamma \in \Gamma^u$ there is $\theta \in \Theta$ such that $\partial^u R_\theta \subseteq W^{\textup{u}}(\gamma)$ and for every $\gamma \in \Gamma^s$ there is $\theta \in \Theta$ such that $\partial^s R_\theta \subseteq W^{\textup{s}}(\gamma)$.
\end{proposition}

The first thing that we need is a way to produce rectangles. The following lemma explains how to construct a regular rectangle.

\begin{lemma}\label{lemma:small_rectangles}
Let $\alpha > 0$. Let $D \subseteq M \setminus \bigcup_{j = 1}^{N} \gamma_j$ be a smooth disk transverse to the flow. Let $x_0 \in D$. Let $\gamma^u, \gamma^s$ be primitive periodic orbits for $\varphi$. Then, there is a regular rectangle $R$ of diameter less than $\alpha$ such that $x_0 \in R^*,R \subseteq D, \partial^u R \subseteq W^{\textup{u}}(\gamma^u)$ and $\partial^s R \subseteq W^{\textup{s}}(\gamma^s)$. 
\end{lemma}

\begin{proof}
We can replace $D$ by a disk of diameter less than $\alpha$, to ensure that at the end $R$ will also have diameter less than $\alpha$. Since $D$ is transverse to the flow, there is an open neighbourhood $U$ of $x_0$ in $M$ and a $C^\infty$ function $t_U : U \to \mathbb{R}$ such that $t_U(x_0) = 0$ and $\varphi_{t_U(x)}(x) \in D$ for every $x \in U$. Up to making $U$ smaller, we may assume that the range of $t_U$ is within $(-\epsilon_0,\epsilon_0)$. Pick then $\epsilon$ very small and define
\begin{equation*}
I_0 = \set{ \varphi_{t_U(x)}(x) :x \in W_{\epsilon}^u(x_0)} \textup{ and } J_0 = \set{ \varphi_{t_U(x)}(x) :x \in W_{\epsilon}^s(x_0)}.
\end{equation*}
Define then the map 
\begin{equation*}
\begin{array}{ccccc}
\Psi &: &I_0 \times J_0 \times (- \epsilon,\epsilon) & \to & M \\
& & (x,y,t) & \mapsto & \varphi_{t + t_U([x,y])}([x,y]).
\end{array}
\end{equation*}
If $\epsilon$ is small enough, then $\Psi$ defines a homeomorphism from $I_0 \times J_0 \times (- \epsilon,\epsilon)$ to a neighbourhood of $x_0$ in $M$. Indeed, the map $(x,y) \mapsto \varphi_{t_U([x,y])}([x,y])$ is a homeomorphism from $I_0 \times J_0$ to a neighbourhood of $x_0$ in $D$ (because $z \mapsto (\varphi_{t_U([z,x_0])}([z,x_0]), \varphi_{t_U([x_0,z])}([x_0,z]))$ is its inverse), and $D$ is transverse to the flow. 

In particular, notice that $\Psi(I_0 \times J_0 \times \set{0})$ is a neighbourhood of $x_0$ in $D$, that the stable and unstable manifolds are parallel to the axis in the coordinates defined by $\Psi$, and that $\Psi$ is a flow box.

Let $I_1, I_2$ and $J_1,J_2$ denote respectively the connected components of $I_0 \setminus \set{x_0}$ and $J_0 \setminus \set{x_0}$. According to \cref{corollary:dense}, the stable manifold $W^{\textup{s}}(\gamma^s)$ is dense in $M$. Hence, it must intersect the open sets $\Psi(I_1 \times J_0 \times (- \epsilon,\epsilon))$ and $\Psi(I_2 \times J_0 \times (- \epsilon,\epsilon))$. Since $W^{\textup{s}}(\gamma^s)$ is invariant under the action of the flow, we find actually that there are $(x_1,y_1) \in I_1 \times J_0$ and $(x_2,y_2) \in I_2 \times J_0$ such that
$\Psi(x_1,y_1,0)$ and $\Psi(x_2,y_2,0)$ belong to $W^{\textup{s}}(\gamma^s)$. Similarly, there are $(x_3,y_3) \in I_0 \times J_1$ and $(x_4,y_4) \in I_0 \times J_2$ such that $\Psi(x_3,y_3,0)$ and $\Psi(x_4,y_4, 0)$ belong to $W^{\textup{u}}(\gamma^u)$.

Let then $I'$ be the subinterval of $I_0$ that goes from $x_1$ to $x_2$ and $J'$ be the subinterval of $J_0$ that goes from $y_3$ to $y_4$. We can define 
\begin{equation*}
R = \Psi(I' \times J' \times \set{0}).
\end{equation*}
It follows from the definition of $\Psi$ that $R$ is a rectangle with the required properties.
\end{proof}

\begin{figure}[h]
	\centering
\def\svgwidth{0.7\linewidth}
	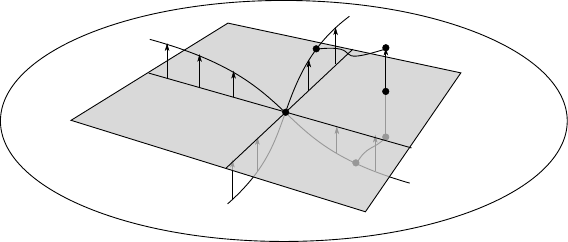
	\caption{Local product coordinates used in the proof of \cref{lemma:small_rectangles}}\label{fig:transverse_disk}
\end{figure}

The proof of \cref{lemma:small_rectangles} adapts easily to produce singular rectangles.

\begin{lemma}\label{lemma:small_rectangles_singular}
Let $j \in \set{1,\dots,N}, k \in \mathbb{Z} / 2 n_j \mathbb{Z}$ and $\tau \in [ - r_{\max}, r_{\max}]$. Let $\gamma^u$ and $\gamma^s$ be primitive periodic orbits for $\varphi$. If $k$ is even assume that $\gamma^u = \gamma_j$ and if $k$ is odd assume that $\gamma^s = \gamma_j$. Let $\alpha > 0$. Let $x_0 \in (\mathbb{R} \times \set{0}) \cap \mathcal{V}$. Then, there is a singular rectangle $R = \Upsilon_{j,k,\tau}(\widetilde{R} \times \set{0})$ (using the notation from \cref{definition:singular_rectangles}) of diameter less than $\alpha$ such that $x_0$ belongs to the interior of $\widetilde{R}$ as a subset of $\mathbb{R} \times [0,+\infty)$ and such that 
\begin{equation}\label{eq:condition_boundary_Markov}
\partial^u R \subseteq W^{\textup{u}}(\gamma^u) \textup{ and } \partial^s R \subseteq W^{\textup{s}}(\gamma^s).
\end{equation}
\end{lemma}

\begin{proof}
The proof follows the same lines as the proof of \cref{lemma:small_rectangles}, except that we are working in half-space coordinates. The main change is that $\Upsilon_{j,k,\tau}(x_0,0)$ will be an extremity of $I_0$ or $J_0$, but this is not an issue for the argument because it happens when $\Upsilon_{j,k,\tau}(x_0,0)$ belongs respectively to the stable or unstable manifold of the singular orbit $\gamma_j$ (that we can use to define one side of our rectangle according to our assumption on $\gamma^u$ and $\gamma^s$).
\end{proof}

In order to avoid self-intersections of rectangles (that are forbidden by \cref{definition:Markov_partition}), we will break rectangles into smaller rectangles using the following result.

\begin{lemma}\label{lemma:into_pieces}
Let $\alpha > 0$. Let $\gamma^u, \gamma^s$ be primitive periodic orbits for $\varphi$. Let $R$ be a rectangle such that $\partial^u R \subseteq W^{\textup{u}}(\gamma^u)$ and $\partial^s R \subseteq W^{\textup{s}}(\gamma^s)$. Then, $R$ is a finite union of rectangles of diameter less than $\alpha > 0$, whose unstable and stable boundaries are contained respectively in $W^{\textup{u}}(\gamma^u)$ and $W^{\textup{s}}(\gamma^s)$.
\end{lemma}

\begin{proof}
Let us deal with the case of a regular rectangle. The small difficulty here is that we require for our rectangles to be bounded by specific stable and unstable manifolds. Let $R$ be a regular rectangle contained within a disk $D$. We will first show that the intersection of $R$ with another rectangle is still a rectangle, and then conclude by covering $R$ by a finite number of rectangles using \cref{lemma:small_rectangles}.

Recall the intervals $I$ and $J$ from \cref{definition:regular_rectangle}. Since $D$ is transverse to the flow and $I$ and $J$ are compact, one can find slightly larger intervals $ \mathcal{I}$ and $\mathcal{J}$ that contain respectively $I$ and $J$ and are contained in $W^{\textup{u}}(\gamma^u) \cap D$ and $W^{\textup{s}}(\gamma^s) \cap D$. Now, using again that $D$ is transverse to the flow, we find that there is a smooth function $t_D$ from a neighbourhood of $\set{[x,y]:x,y \in R}$ to $\mathbb{R}$ such that $t_D([x,y]) = t_R([x,y])$ when $x,y \in R$ (where $t_R$ is from \cref{definition:regular_rectangle}) and $\varphi_{t_D(x)}(x) \in D$ for every $x$ in the domain of $t_D$. Define then the map $\Psi :(x,y,t) \to \varphi_{t+t_D([x,y])}([x,y])$, and notice that $\Psi$ induces a homeomorphism from a neighbourhood of $I \times J \times \set{0}$ in $\mathcal{I} \times \mathcal{J} \times \mathbb{R}$ to a neighbourhood of $R$ in $M$. Now, let $R'$ be a regular rectangle of small diameter contained in the intersection of $D$ and the image of $\Psi$, and with $\partial^u R' \subseteq W^{\textup{u}}(\gamma^u)$ and $\partial^s R' \subseteq W^{\textup{s}}(\gamma^s)$. We write as in the definition $R' = \set{\varphi_{t_{R'}([x,y])}([x,y]): x \in I', y \in J'}$. Since $I'$ and $J'$ are pieces of weak stable and unstable manifolds respectively intersected with $D$, we find that $\Psi^{-1}(I')$ and $\Psi^{-1}(J')$ are of the form $\widetilde{I} \times \set{a} \times \set{0}$ and $\set{b} \times \widetilde{J} \times \set{0}$, with $\widetilde{I} \subseteq \mathcal{I}, \widetilde{J} \subseteq \mathcal{J}, a \in \mathcal{J}, b \in \mathcal{I}$. Hence, we find that $R' \cap R = \Psi(I \cap \widetilde{I} \times J \cap \widetilde{J} \times \set{0})$ is also a rectangle with its stable and unstable boundaries within $W^{\textup{s}}(\gamma^s)$ and $W^{\textup{u}}(\gamma^u)$ respectively.

By \cref{lemma:small_rectangles} and compactness of $R$, we may cover $R$ by finitely many rectangles of diameter less than $\alpha$ contained in the intersection of $D$ with the image of $\Psi$. Intersecting these rectangles with $R$, the lemma follows by the argument above.

To deal with the case of singular rectangles, we just need to use \cref{lemma:small_rectangles_singular} instead of \cref{lemma:small_rectangles} in order to cover the side of the rectangle which is on the boundary of the range of the half-space parametrization (if there is one).
\end{proof}

In order to prove \cref{proposition:existence_Markov}, we start by constructing a family of rectangles partially satisfying \cref{definition:Markov_partition}.

\begin{lemma}\label{lemma:Markov_first_step}
Let $\Gamma^u$ and $\Gamma^s$ be two finite sets of primitive periodic orbits for $\varphi$ such that $\set{\gamma_1,\dots,\gamma_N} \subseteq \Gamma^u \cup \Gamma^s$. Let $\alpha > 0$ be small enough. There is a family $(R_\theta)_{\theta \in \Theta}$ satisfying \cref{eq:assumption_boundary} and the first three points in \cref{definition:Markov_partition}. Moreover, for each $\gamma^u \in \Gamma^u$ and $\theta \in \Theta$, if $x$ is in in the intersection of $\gamma^u$ and $R_\theta$ then $x \in \partial^u R_\theta$. The same is true with $u$ replaced by $s$.
\end{lemma}

\begin{proof}
It is easy to produce a family of rectangles satisfying the first two points from \cref{definition:Markov_partition} and \cref{eq:assumption_boundary}: you use \cref{lemma:small_rectangles_singular} to produce a family of rectangles covering a neighbourhood of the singular orbits and then \cref{lemma:small_rectangles} to cover the rest of the manifold. 

One can then use \cref{lemma:into_pieces} to remove all intersections between rectangles. You pick a rectangle, break it into very small pieces using \cref{lemma:into_pieces}, then by shifting these pieces along the flow, you can ensure that none of them intersect any other rectangles. This process produces a new family of rectangles with at least one less rectangle that intersect other rectangles. Iterating this process, we get a family of rectangles that do not intersect each other, and if $\alpha$ is small enough, it ensures that the third point in \cref{definition:Markov_partition} holds.

To get the last condition, if a rectangle intersects some $\gamma^u \in \Gamma^u$ and that the intersection is not on its unstable boundary, then we cut the rectangle in two along the connected component of the intersection of $W^{\textup{u}}(\gamma^u)$ and the rectangle.
\end{proof}

The key point in order to construct a Markov partition is the following:

\begin{lemma}\label{lemma:key_Markov}
Let $(R_\theta)_{\theta \in \Theta}$ be as in \cref{lemma:Markov_first_step}. Write $\Omega = \bigcup_{\theta \in \Theta} R_\theta$ and let $T : \Omega \to \Omega$ be the first return map of $\varphi$. Let $W$ be a compact subset of $\Omega$ such that $W \subseteq W^{\textup{u}}(\gamma)$ for some $\gamma \in \Gamma^u$. Then, there is $n_0 \in \mathbb{N}$ such that $T^{-n}(W) \subseteq \bigcup_{\theta \in \Theta} \partial^u R_\theta$ for every $n \geq n_0$.

The same statement with $u$ replaced by $s$ and $T$ by $T^{-1}$ also holds.
\end{lemma}

\begin{proof}
Let us make a proof by contradiction. If the result were false, then one could construct a sequence $(x_j)_{j \geq 0}$ of points in $W$ and a sequence $(n_j)_{j \geq 0}$ of integers going to $+ \infty$ such that $T^{-n_j}(x_j) \notin \bigcup_{\theta \in \Theta} \partial^u R_\theta$. Up to extracting a subsequence, we may assume that there is $\theta \in \Theta$ such that $T^{-n_j}(x_j) \in R_\theta$ for every $j \geq 0$. Up to extracting again, we may assume that $(T^{-n_j}(x_j))_{j \geq 0}$ converges to a point $x \in R_{\theta}$.

Now, since the $R_\theta$'s are disjoint, the first return time of $\varphi$ to $\Omega$ is bounded below by some constant $t_*$. Hence, for each $j \geq 0$, we have $T^{-n_j}(x_j) = \varphi_{t_j}(x_j)$ with $t_j \leq - n_j t_*$. It follows that $x$ is a point on $\gamma$, and thus that $x \in \partial^u R_{\theta}$. Using that $W$ is compact and the description of local unstable manifolds from \cref{proposition:unstable_manifolds}, we find that for $j$ large enough, $T^{-n_j} (x_j)$ belongs to an arbitrarily small neighbourhood of $x$ in the projection on $R_\theta$ along the flow of its local unstable manifold. It follows then that $T^{-n_j}(x_j)$ belongs to $\partial^u R_\theta$ (see \cref{remark:rectangle_stable_unstable}), a contradiction.
\end{proof}

We can now upgrade the family of rectangles from \cref{lemma:Markov_first_step} in order to construct a Markov partition.

\begin{proof}[Proof of \cref{proposition:existence_Markov}]
Let us start with a family of rectangles $(R_\theta)_{\theta \in \Theta}$ as in \cref{lemma:Markov_first_step}. We are going to divide the $R_\theta$'s into subrectangles using an inductive procedure. Thus, we will get a sequence of family of rectangles $\mathfrak{R}_0,\dots, \mathfrak{R}_n,\dots$, and we will see that for $n$ large enough $\mathfrak{R}_n$ is a Markov partition, except that the rectangles in $\mathfrak{R}_n$ may intersect on their boundaries, but this can be fixed by shifting them a little along the flow, since the elements of $\mathfrak{R}_n$ are obtained by cutting the $R_\theta$'s into pieces.

Let us define $\mathfrak{R}_0$ first. For each $\theta_1,\theta_2 \in \Theta$, let us consider the sets
\begin{equation*}
\overline{\set{x \in R_{\theta_1}^* : Tx \in R_{\theta_2}^*}} \textup{ and } \overline{\set{x \in R_{\theta_1}^* : T^{-1}x \in R_{\theta_2}^*}}
\end{equation*}
They are bounded by finitely many intervals within $\bigcup_{\gamma \in \Gamma^s} W^{\textup{s}}(\gamma) \cap R_{\theta_1}$ and $\bigcup_{\gamma \in \Gamma^u} W^{\textup{u}}(\gamma) \cap R_{\theta_1}$. Letting $\theta_2$ run over $\Theta$, we get that way a finite collection of segments within $R_{\theta_1}$. Extending these segments up to the sides of $R_{\theta_1}$, we get a grid on $R_{\theta_1}$, hence dividing $R_{\theta_1}$ into smaller rectangles. We let $\mathfrak{R}_0$ be the collection of rectangles obtained this way (using the same procedure for each $\theta_1 \in \Theta$). What we gain from this first step is that if $R \in \mathfrak{R}_0$ then there is $\theta \in \Theta$ such that $T(R^*) \subseteq R_\theta$.

Using this last property, we explain how to go from $\mathfrak{R}_0$ to $\mathfrak{R}_1$. With $R \in \mathfrak{R}_0$, we take $\theta \in \Theta$ such that $T(R^*) \subseteq R_\theta$. Hence, $T_{|R^*}$ induces (by extending it by continuity to $R$) a bijection between $R$ and a subset of $R_\theta$. The subdivision of $R_\theta$ into elements of $\mathfrak{R}_0$ induces consequently a subdivision of $R$ into subrectangles. Doing the same for $T^{-1}$ and ``intersecting'' the two subdivisions, we get a subdivision of $R$ into rectangles. Doing the same for each element of $\mathfrak{R}_0$, we let $\mathfrak{R}_1$ denote the new collection of rectangles. 

Iterating this procedure, we get a sequence of family of rectangles $\mathfrak{R}_0,\mathfrak{R}_1,\mathfrak{R}_2$, \dots Let us prove that for $n$ large enough, $\mathfrak{R}_n$ satisfies \ref{item:Markov_property} in \cref{definition:Markov_partition}. Notice that on the interior of every rectangle $R \in \mathfrak{R}_n$, one can define inductively the $n$th return map of $\varphi$ to $\bigcup_{\theta \in \Theta} R_\theta$, and this maps extends continuously to a map $T_n$ from $R$ to $R_\theta$ for some $\theta \in \Theta$. Let then $x,y \in  R$ belong to the same stable manifold. We want to prove that $T_n(x)$ and $T_n(y)$ belong to the same element of $\mathfrak{R}_0$. If this is not the case, there is some piece $W$ of unstable manifold bounding an element of $\mathfrak{R}_0$ that intersects $W^s_{R_{\theta}}(T_n(x))$ between $T_n(x)$ and $T_n(y)$. Thus, $T_n^{-1}(W)$ is a piece of unstable manifold that intersects $W_R^s(x)$ between $x$ and $y$. But, \cref{lemma:key_Markov} implies that, provided $n$ is large enough, $T_n^{-1}(W)$ is contained within the unstable boundary of one of the $R_{\theta'}$'s, and thus it cannot cut $W_R^s(x)$ between $x$ and $y$, a contradiction. Here, how large $n$ needs to be is uniform in $x$ and $y$ since we only need to apply \cref{lemma:key_Markov} to a finite number of pieces of unstable manifolds (each side of the unstable boundaries of the elements of $\mathfrak{R}_0$). Hence, $T_n x$ and $T_n y$ belong to the same element of $\mathfrak{R}_0$. Consequently, if we look at the first return time $T_1$ from the interior of $R$ to one of the original rectangles extended by continuity to $R$, then $T_1(x)$ and $T_1(y)$ belong to the same element of $\mathfrak{R}_n$. The assumption on stable manifolds in \ref{item:Markov_property} from\cref{definition:Markov_partition} follows. We prove similarly the properties for unstable manifolds in backward time.

Finally, since the elements of $\mathfrak{R}^n$ are obtained by cutting the elements of $(R_\theta)_{\theta \in \Theta}$ into pieces, we find that for every $\gamma \in \Gamma^u$ there is $R \in \mathfrak{R}_n$ such that $\partial^u R \subseteq W^{\textup{u}}(\gamma)$ and for every $\gamma \in \Gamma^s$ there is $R \in \mathfrak{R}_n$ such that $\partial^s R \subseteq W^{\textup{s}}(\gamma)$.
\end{proof}

\section{Zeta functions associated to smooth pseudo-Anosov flows}\label{section:zeta_continuation}

In this section, we prove \cref{theorem:continuation_zeta} about the meromorphic continuation of twisted zeta functions associated to smooth pseudo-Anosov flows. We will also give a formula for the value at zero of some zeta functions (\cref{lemma:zeta_Markov_at_zero}) that will be used in the proof of \cref{theorem:fried_conjecture}. For the whole section, let us fix a transitive smooth pseudo-Anosov flow $\varphi$ on a manifold $M$.

Our approach is based on a symbolic coding for smooth pseudo-Anosov flows based on Markov partitions. We explain how this coding works in \cref{subsection:symbolic_dynamics}. Most technical results associated to the coding are very similar to the case of Anosov flows (see for instance \cite{bowen_markov}).

We explain then in \cref{subsection:dynamical_determinant} how the analysis from \cite{baladi_tsujii_determinant} can be used in order to get holomorphic extensions for certain dynamical determinants associated to our symbolic dynamics. We use these dynamical determinants in \cref{subsection:zeta_functions} to construct a meromorphic continuation for a first zeta function. However, due to errors we make when counting periodic orbits using the symbolic coding, this is not exactly the zeta function from \cref{theorem:continuation_zeta}. We solve this issue in \cref{subsection:correcting} by making the counting errors explicit.

In \cref{subsection:value_at_zero} we give a first formula for the value of the certain twisted zeta functions at zero that will be used in the proof of Theorems \ref{theorem:fried_conjecture} and \ref{theorem:dirichletclassnumber}.

Here, even though we work with symbolic coding, we are not limited by the regularity of the coding when constructing meromorphic continuations for zeta functions, because we do not study directly the symbolic flow associated with $\varphi$ but instead work with a family of open hyperbolic maps associated to our Markov partition that are smooth. This is the strategy that was used in the analytic setting by Rugh \cite{rugh96}.

\subsection{Symbolic dynamics}\label{subsection:symbolic_dynamics}

Let us explain how one can use Markov partitions to construct a symbolic model for the flow that can be used to count periodic orbits and study the associated zeta functions. Many technical details are similar to the Anosov case, and we may consequently go rapidly over them (see e.g. \cite{bowen_markov,parry_pollicott_asterisque}).

Let $\Gamma^u$ and $\Gamma^s$ be two non-empty finite sets of primitive periodic orbits for $\varphi$ such that $\Gamma^u \cup \Gamma^s$ contains the singular orbits of $\varphi$. In order to compute the value at zero of zeta functions, we will require further properties on $\Gamma^u$ and $\Gamma^s$, but it is crucial that we do not need to impose them yet. Choose a Markov partition $(\widetilde{R}_\theta)_{\theta \in \Theta}$ with small diameter, as constructed in \cref{proposition:existence_Markov}, and let $\widetilde{T}$ be the first return map of the flow $(\varphi_t)_{t \in \mathbb{R}}$ to $\bigcup_{\theta \in \Theta} \widetilde{R}_\theta$, and $\widetilde{\mathfrak{r}}$ the associated first return time. Let $A = (A_{\theta,\vartheta})_{\theta,\vartheta \in \Theta}$ be the adjacency matrix defined by
\begin{equation*}
A_{\theta,\vartheta} = \begin{cases} 1 & \textup{ if } \widetilde{T}(\widetilde{R}_\theta^*) \cap \widetilde{R}_\vartheta^* \neq \emptyset; \\
0 & \textup{ otherwise}. \end{cases}
\end{equation*}
We refer to Definitions \ref{definition:regular_rectangle}, \ref{definition:singular_rectangles} and \ref{definition:Markov_partition} for the definitions of rectangles and Markov partitions. Let $(\Sigma,\sigma)$ be the shift of finite type associated to $A$:
\begin{equation*}
\Sigma = \set{ (\theta_m)_{m \in \mathbb{Z}} \in \Theta^{\mathbb{Z}} : A_{\theta_m, \theta_{m+1}} = 1 \textup{ for every } m \in \mathbb{Z}}
\end{equation*}
and $\sigma : \Sigma \to \Sigma$ is defined by $\sigma((\theta_m)_{m \in \mathbb{Z}}) = (\theta_{m+1})_{m \in \mathbb{Z}}$. We will also need for $n \geq 1$ the set
\begin{equation*}
\Sigma_n = \set{(\theta_0,\dots, \theta_{n-1}) \in \Theta^n : A_{\theta_{k},\theta_{k+1}} = 1 \textup{ for } k = 0,\dots,n-2}.
\end{equation*}

For $\theta \in \Theta$, if $\widetilde{R}_\theta$ is a regular rectangle, we choose a smooth disk $\widetilde{D}_\theta$ such that $\widetilde{R}_\theta \subseteq \widetilde{D}_\theta$ as in \cref{definition:regular_rectangle}. We let then $D_\theta$ be an open disk in $\mathbb{R}^2$ such that there is a $C^\infty$ diffeomorphism $\kappa_\theta :D_\theta \to \widetilde{D}_\theta$. We set\footnote{Beware that our notations differ from the previous section. The $\widetilde{R}_\theta$'s are now the rectangles in the manifold $M$ and the $R_\theta$'s are subsets of $\mathbb{R}^2$. The reason for this change is that most of the analysis will involve the subsets of $\mathbb{R}^2$ rather than the rectangles in $M$ from now on.} $R_\theta = \kappa_\theta^{-1}(\widetilde{R}_\theta)$. If $\widetilde{R}_\theta$ is not a regular rectangle, then there is $j \in \set{1,\dots,N}$,$k \in \mathbb{Z}/ 2 n_j \mathbb{Z}$ and $\tau \in [-r_{\max},r_{\max}]$ such that $\widetilde{R}_\theta = \Upsilon_{j,k,\tau}(R_\theta \times \set{0})$ for some subset $R_\theta$ of $\mathcal{V}$ as in \cref{definition:singular_rectangles}. In that case, we define $\kappa_\theta$ by $\kappa_\theta(x,y) = \Upsilon_{j,k,\tau}(x,y,0)$ for $(x,y) \in R_\theta$.

If $x \in R_\theta$ for some $\theta \in \Theta$, we define the stable and unstable manifolds 
\begin{equation*}
W_{R_\theta}^s(x) = \kappa_\theta^{-1}(W_{\widetilde{R}_\theta}^s(\kappa_\theta x)) \textup{ and } W_{R_\theta}^u(x) = \kappa_\theta^{-1}(W_{\widetilde{R}_\theta}^u(\kappa_\theta x)).
\end{equation*}
We denote their tangent spaces at $x$ by $E^s_\theta(x)$ and $E^u_\theta(x)$. Notice that these directions depend continuously on $x$. If $x,y \in R_\theta$, we will also let $[x,y]_{R_\theta}$ denote the intersection of $W_{R_\theta}^s(x)$ and $W_{R_\theta}^u(y)$ (which exists since $\widetilde{R}_\theta$ is a rectangle).

If $\theta,\vartheta \in \Theta$ are such that $A_{\theta,\vartheta} = 1$, define the rectangle
\begin{equation*}
\widetilde{R}_{\theta,\vartheta} = \overline{\set{x \in \widetilde{R}_\theta^*: Tx \in \widetilde{R}_\vartheta^*}} \textup{ and } R_{\theta,\vartheta} = \kappa_\theta^{-1}(\widetilde{R}_{\theta,\vartheta})
\end{equation*}
Recall that $\widetilde{R}_{\theta,\vartheta}$, and thus $R_{\theta,\vartheta}$, are connected. Using the compactness of $\widetilde{R}_{\theta,\vartheta}$ and the implicit function theorem (for regular rectangles) or the local model for the flow (for singular rectangles that are not regular), we see that $\tilde{\mathfrak{r}}_{|\widetilde{R}_{\theta,\vartheta}^*} \circ \kappa_\theta$ extend to a smooth map $\mathfrak{r}_{\theta,\vartheta}$ on $R_{\theta,\vartheta}$. We can then define $T_{\theta,\vartheta}$ from $R_{\theta,\vartheta}$ to $R_\vartheta$ by 
\begin{equation*}
T_{\theta,\vartheta} (x) = \kappa_\vartheta^{-1}( \varphi_{\mathfrak{r}_{\theta,\vartheta}(x)}(\kappa_\theta(x))).
\end{equation*}
Notice that $\kappa_\vartheta \circ T_{\theta,\vartheta} \circ \kappa_\theta^{-1}$ is a continuous extension of $\widetilde{T}_{|\widetilde{R}^*_{\theta,\vartheta}}$ and that $T_{\theta,\vartheta}$ is smooth on $R_{\theta,\vartheta}$, even when $\widetilde{R}_{\theta,\vartheta}$ is not a regular rectangle due to the local model for the flow.

From the invariance of the stable and unstable manifolds, we find that if $\theta,\vartheta \in \Theta$ are such that $A_{\theta,\vartheta} = 1$ and that $x \in R_{\theta,\vartheta}$ then 
\begin{equation*}
D T_{\theta,\vartheta}(x)(E_\theta^u(x)) = E_\vartheta^u(T_{\theta,\vartheta}x) \textup{ and } D T_{\theta,\vartheta}(x)(E_\theta^s(x)) = E_\vartheta^s(T_{\theta,\vartheta}x).
\end{equation*}

Now, if $w = (w_0,\dots,w_n) \in \Sigma_{n+1}$, define by induction the set
\begin{equation*}
R_w = \set{ x \in R_{w_0,w_1} : T_{w_0,w_1} x \in R_{w_1,\dots, w_n}}.
\end{equation*}
\begin{lemma}\label{lemma:one_sided_rectangle}
Let $n \in \mathbb{N}$ and $w \in \Sigma_{n+1}$. Then, $R_w$ is a non-empty connected rectangle.
\end{lemma}

\begin{proof}
The proof is by induction on $n$. We already know that it is true for $n = 0$ and $ n =1$. For the induction step, just write
\begin{equation*}
R_w = R_{w_0,w_1} \cap T_{w_0,w_1}^{-1}(R_{w_1,\dots,w_n}).
\end{equation*}
Hence, $R_w$ is the intersection of two connected subrectangles of $R_{w_0}$ and is thus a connected rectangle. To find that it is non-empty, just notice that if $x \in R_{w_0,w_1}$ and $y \in R_{w_1,\dots,w_n}$ then $T_{w_0,w_1}^{-1}[y, T_{w_0,w_1} x] \in R_w$ (here, we use the Markov property to see that $[y,T_{w_0,w_1} x] \in R_{w_1,\dots,w_n} \cap T_{w_0,w_1}(R_{w_0,w_1})$).
\end{proof}
Notice that on $R_w$, we can define the map $T_w = T_{w_{n-1}, w_n} \circ \dots \circ T_{w_0,w_1}$. The maps defined this way are hyperbolic in the following sense : there are constants $C > 0$ and $\lambda > 1$ such that for every $n \in \mathbb{N}, w \in \Sigma_{n+1}$, $x \in R_w, v \in E^s_{w_0}(x)$ and $v' \in E^u_{w_n}(T_w x)$ we have
\begin{equation}\label{eq:symbolic_hyperbolicity}
|D T_w(x) \cdot v| \leq C \lambda^{-n} |v| \textup{ and } |(D T_w (x))^{-1} \cdot v'| \leq C \lambda^{-n} |v'|.
\end{equation}
This fact can be proved by an argument similar to the one in \cref{lemma:derivative_graph_transform} (when we proved the hyperbolicity for the maps $F_n$'s).

Up to subdividing further the rectangles to get a new Markov partition with smaller rectangles, we may assume that the $T_{\theta,\vartheta}$'s are $C^1$ close to their affine approximation at any point of $R_{\theta,\vartheta}$. We deduce several consequences from that fact:
\begin{itemize}
\item we can extend the $T_{\theta,\vartheta}$'s to diffeomorphism from $\mathbb{R}^2$ to itself (using a formula similar to \eqref{eq:extension_diffeo}) that is affine outside of a compact set. We will still call them $T_{\theta,\vartheta}$, and their composition is also denoted by $T_w$.
\item we can assume that for each $\theta \in \Theta$ there is a polarization $\mathcal{P}_\theta$ such that for each $\theta,\vartheta \in \Theta$ such that $A_{\theta,\vartheta} = 1$, the map $T_{\theta,\vartheta}$ is regular cone-hyperbolic with respect to $\mathcal{P}_\vartheta$ and $\mathcal{P}_\theta$ (see the definition in \cite{baladi_tsujii_determinant} page 18).
\end{itemize}

The following lemma is essential in order to relate the symbolic dynamical system $(\Sigma,\sigma)$ with the flow $\varphi$.

\begin{lemma}\label{lemma:iterated_rectangles}
There is a constant $C > 0$ such that for every $n \in \mathbb{N}$, if $w,w' \in \Sigma_{n+1}$ are such that $w'_{n} = w_0$ then the set $R_w \cap T_{w'}(R_{w'})$ is non-empty and has diameter less than $C \lambda^{-n}$.
\end{lemma}

\begin{proof}
Notice that $R_w \cap T_{w'}(R_{w'}) = T_{w'}( R_{\bar{w}})$ where $$\bar{w} = w'_0 w'_1 \dots w'_{n-1} w_0 w_1 \dots w_{n+1}.$$ Hence, it follows from \cref{lemma:one_sided_rectangle} that $R_w \cap T_{w'}(R_{w'})$ is a non-empty connected rectangle.

Now, let $x,y \in R_w \cap T_{w'}(R_{w'})$, and write
\begin{equation*}
d(x,y) \leq d(x,[x,y]_{R_{w_0}}) + d(y,[x,y]_{R_{w_0}}).
\end{equation*}
Let us bound for instance $d(x,[x,y]_{R_{w_0}})$ (the other case is similar using unstable manifolds instead of stable manifolds). The point $x$ and $[x,y]_{R_{w_0}}$ belong to the same stable manifold. Consider the piece of stable manifold $W$ that joins $x$ and $[x,y]_{R_{w_0}}$. It follows from the Markov property that $W \subseteq T_{w'}(R_{w'})$. Hence, $T_{w'}^{-1}(W)$ is a piece of stable manifold in $R_{w'_0}$ that joins $T_{w'}^{-1}(x)$ and $T_{w'}^{-1}([x,y]_{R_{w_0}})$. Since the geometries of the stable manifolds are bounded, there is a uniform bound on the length of $T_{w'}^{-1}(W)$, and it follows then from the hyperbolicity of $T_{w'}$ that the length of $W$ is bounded by $C \lambda^{-n}$ for some $C > 0$. Since the length of $W$ is larger than $d(x,[x,y]_{R_{w_0}})$, the result follows.
\end{proof}

It follows from \cref{lemma:iterated_rectangles} that for every $w \in \Sigma$, there is a unique point $x \in R_{w_0}$ such that $x$ belongs to $R_{w_0,\dots,w_n}$ and $T_{w_{-n},\dots,w_0}(R_{w_{-n}})$ for every $n \in \mathbb{N}$. We denote this point by $\pi_{\Sigma}(w)$. The diameter estimate in \cref{lemma:iterated_rectangles} implies that the map $\pi_{\Sigma}$ is H\"older-continuous when $\Sigma$ is endowed with the distance
\begin{equation*}
d((x_m)_{m \in \mathbb{Z}}, (y_m)_{m \in \mathbb{Z}}) = 2^{- m_0((x_m)_{m \in \mathbb{Z}}, (y_m)_{m \in \mathbb{Z}})}
\end{equation*}
where 
\begin{equation*}
m_0((x_m)_{m \in \mathbb{Z}}, (y_m)_{m \in \mathbb{Z}}) = \inf\set{m \in \mathbb{N} : x_m \neq y_m \textup{ or } x_{-m} \neq y_{-m}}.
\end{equation*}
It is also standard that the map $\pi_{\Sigma}$ is surjective from $\Sigma$ to $\sqcup_{\theta \in \Theta} R_{\theta}$ and is finite-to-one. Using the map $\pi_\Sigma$, we can associate to a point $w \in \Sigma$ a point $\kappa_{w_0}(\pi_\Sigma(w))$ in $M$. It follows then from the definition of $\pi_\Sigma$ that points that are on the same orbit for $\sigma$ will give rise to points on the same orbit for $\varphi$. We are particularly interested in periodic orbits: 

\begin{lemma}
Let $n \geq 1$ and $w \in \Sigma$ be such that $\sigma^n w = w$. Then $\pi_{\Sigma}(w)$ is the unique fixed point of $T_{w_0,\dots,w_n}$ in $R_{w_0}$.
\end{lemma}

\begin{proof}
The fact that $\pi_{\Sigma}(w)$ is a fixed point of $T_{w_0,\dots,w_n}$ follows from its definition. The fact that it is unique follows from the cone-hyperbolicity of $T_{w_0,\dots,w_n}$ and the fact that it is linear outside of a compact set. Indeed, for $m \geq 1$ large enough, the map $T_{w_0,\dots,w_n}^m - I : \mathbb{R}^2 \to \mathbb{R}^2$ is invertible.
\end{proof}

\begin{definition}\label{definition:gammao}
To a periodic orbit $\mathcal{O}$ for $\sigma$, we associate a periodic orbit $\gamma_{\mathcal{O}}$ for $\varphi$: take a point $w \in \mathcal{O}$ and let $\gamma_{\mathcal{O}}$ be the orbit of $\varphi$ starting\footnote{Remember that we identify periodic orbits that only differ by a shift in their parametrizations, so that $\gamma_{\mathcal{O}}$ does not depend on the choice of $w$.} at $\kappa_{w_0}(\pi_\Sigma(w))$ and of length $T_{\mathcal{O}} \coloneqq \sum_{j = 0}^{m-1} \mathfrak{r}_{w_j,w_{j+1}} (\pi_{\Sigma}(\sigma^j w))$ where $m$ is the minimal period of $w$.
\end{definition}

It is important in \cref{definition:gammao} that $m$ is the minimal period of $w$, otherwise we would get a multiple of the orbits that we actually want to consider. Notice that the orbit $\gamma_{\mathcal{O}}$ may not be primitive, due to the lack of injectivity of $\pi_{\Sigma}$. It is not easy in general to understand which periodic orbits of $\varphi$ are of the form $\gamma_{\mathcal{O}}$. It is essential for us though, since the orbits of the form $\gamma_{\mathcal{O}}$ are those that will be counted by the symbolic dynamics associated to the Markov partition. There is a general argument to fix the error made when counting periodic orbits this way \cite{bowen_markov}. However, the very specific Markov partition that we were able to build due to the low dimension will allow us to evaluate more directly how to fix the counting of periodic orbits. We start by noticing that most periodic orbits are counted correctly. The remaining orbits will be dealt with in the proof of \cref{proposition:explicit_zeta_Markov3}.

\begin{lemma}\label{lemma:no_duplicate_regular}
Let $\gamma$ be a primitive periodic orbit for $\varphi$. Assume that $\gamma$ does not belong to $\Gamma^u \cup \Gamma^s$. Then:
\begin{itemize}
\item there is a unique periodic orbit $\mathcal{O}$ for $\sigma$ such that $\gamma = \gamma_{\mathcal{O}}$;
\item if $k \geq 2$, then for every periodic orbit $\mathcal{O}$ for $\sigma$ we have $\gamma^k \neq \gamma_{\mathcal{O}}$.
\end{itemize}
\end{lemma}

\begin{proof}
Since $\pi_{\Sigma}$ is surjective, there is a periodic orbit $\mathcal{O}$ for $\sigma$ and an integer $k \geq 1$ such that $\gamma_{\mathcal{O}} = \gamma^k$. Now, if this equality happens for some $k \geq 2$, then it means that there is a point on $\gamma$ with two antecedents by the map $w \mapsto \kappa_{w_0}(\pi_{\Sigma}(w))$. The same happens if $\gamma$ has more than one antecedents by $\mathcal{O} \mapsto \gamma_{\mathcal{O}}$.

Thus, we only need to prove that the points on $\gamma$ have at most one antecedent by the map $w \mapsto \kappa_{w_0}(\pi_{\Sigma}(w))$. However, it follows from the definition of $\pi_\Sigma$ that the orbit of a point with several antecedents intersects the side of some rectangle. But $\gamma$ cannot intersect the side of a rectangle, since it would then belong either to the weak unstable manifold of an element of $\Gamma^u$ or to the weak stable manifold of an element of $\Gamma^s$, and thus be an element of $\Gamma^u$ or $\Gamma^s$.
\end{proof}

Let us also notice that it follows from the transitivity of the flow $(\varphi_t)_{t \in \mathbb{R}}$ that $\sigma$ acting on $\Sigma$ is transitive. We will now use the existence of Markov partitions and symbolic coding to prove the existence of specific periodic orbits for pseudo-Anosov flows, which we will then be able to use to produce better Markov partitions (in view of the proof of \cref{theorem:fried_conjecture}).

\begin{lemma}[McMullen's realization lemma]\label{lemma:mcmullen_realization} 
	Let $\widehat{M} \coloneqq M \setminus \bigcup_{\gamma \in \Gamma^u \cup \Gamma^s} \gamma$.  Let $G$ be a compact group and $\rho:\pi_1(\widehat{M}) \to G$ be a group morphism. Choose a basepoint for $\pi_1(\widehat{M})$ on a rectangle $R$ of our Markov partition for $\varphi$. Let $g\in \pi_1(\widehat{M})$ and $W$ be a neighbourhood of $\rho(g)$ in $G$. There exist infinitely many primitive periodic orbits $\gamma$ for $ \varphi$, not in $\Gamma^u$ or $\Gamma^s$, starting and ending on $R$ such that $\rho(\gamma) \in W$. Here, we convert $\gamma$ to an element of $\pi_1(\widehat{M})$ by joining it to the basepoint with a path in $R$.
\end{lemma}
\begin{proof}
	We follow an argument from \cite{mcmullen}. Let $\Gamma$ be the directed graph associated with the Markov partition, i.e. the directed graph whose adjacency matrix is $A$. Identifying $\Gamma$ with a $1$-dimensional CW-complex, one may embed $\Gamma$ in $\widehat{M}$ in the following way. For each $\theta \in \Theta$, choose a point $x_\theta \in \widetilde{R}_\theta$ and identify with the corresponding vertex in $\Gamma$. If $(\theta,\vartheta)$ is an edge in $\Gamma$, choose a small piece of orbit $a_{\theta,\vartheta}$ for $\varphi$ joining a point in $\widetilde{R}_\theta^*$ and a point in $\widetilde{R}_\vartheta^*$, corresponding to a first return to $\bigcup_{\theta' \in \Theta} \widetilde{R}_{\theta'}$. Join then the extremities of $a_{\theta,\vartheta}$ with $x_{\theta}$ and $x_{\vartheta}$ using injective paths in $\widetilde{R}_\theta$ and $\widetilde{R}_\vartheta$ respectively, and identify the resulting path with the edge from $\theta$ to $\vartheta$ in $\Gamma$.
	
Whenever $A_{\theta, \vartheta}\neq 0$, let $U_{\theta, \vartheta}$ be the flow box between $\widetilde{R}_{\theta}$ to $\widetilde{R}_{\vartheta}$:
\begin{equation*}
U_{\theta,\vartheta} = \set{ \varphi_t(x) : x \in \widetilde{R}_{\theta,\vartheta}, 0 \leq t \leq \mathfrak{r}_{\theta,\vartheta}(\kappa_\theta^{-1}(x))}.
\end{equation*}
Let $U=\bigcup_{\theta, \vartheta} U_{\theta, \vartheta}$. Let $\partial^{ws}U$ be the union of the boundaries of the $U_{\theta, \vartheta}$ which are contained in the weak stable foliation. Similarly define $\partial^{wu}U$. Note that $\Gamma$ is homotopy equivalent to $M\setminus (\partial^{wu}U \cup \partial^{ws}U)$. 
	
	Let $g \in \pi_1(\widehat{M})$ and $c$ be a closed loop in $\widehat{M}$ representing $g$. Let us show that we can assume that $c$ is in the image of $\Gamma$. In fact, this follows from the stronger statement that $\partial^{ws}U \cup \partial^{wu}U$ retracts to a finite union of one-dimensional submanifolds of $M$. Here is a more pedestrian explanation. Homotope $c$ to intersect $\partial^{ws} U \cup \partial^{wu} U$ transversely in finitely many points. Now each of the intersection points with $\partial^{ws}U$ can be eliminated by a further homotopy. Simply arrange that $c$ intersects $\partial^{ws} U$ at a point which is not in one of the orbits in $\Gamma^s$. The backward orbit of this point eventually leaves $\partial^{ws} U$. So $c$ can be rerouted in a small tubular neighbourhood of this orbit to avoid $\partial^{ws} U$. Similarly, $c$ can be rerouted to avoid $\partial^{us} U$. Now $c$ lies in $M \setminus (\partial^{ws} U \cup \partial^{wu})$, so $c$ can be further homotoped to lie in the image of $\Gamma$.

	Consequently, $c$ is the image of a cycle $\gamma_0$ in the graph $\Gamma$. However, $\gamma_0$ may not necessarily be a directed cycle in $\Gamma$, since it may have some number $N$ of backwards edges. But we claim that $\gamma_0$ can be replaced with a directed cycle $\gamma_1$ in $\Gamma$ so that $\rho(\gamma_1) \in W$ (where we use the map $\pi_1(\Gamma) \to \pi_1(\widehat{M})$ to define $\rho$ on $\pi_1(\Gamma)$) for any open neighbourhood $W$ of $\rho(\gamma_0) = \rho(g)$. Suppose $\gamma_0$ traverses an edge of $(\theta, \vartheta)$ in the reverse direction. Since our flow is transitive, $\Gamma$ is strongly connected. Therefore, we can find a directed cycle $P$ in $\Gamma$ which starts with the edge $(\theta, \vartheta)$. Since $G$ is compact, we may select a large $k$ so that $\rho(P^k)$ is arbitrarily close to the neutral element of $G$. Now splice $P^k$ into $\gamma_0$ immediately after the backwards edge $(\vartheta, \theta)$. We may cancel the first edge of $P^k$ with the backwards edge. This operation decreases the number of backwards edges in $\gamma_0$ and, provided $\rho(P^k)$ is close enough to the neutral element, $\rho(\gamma_0)$ remains in $W$. Repeating this process, we may eliminate all the backwards edges of $\gamma$ and get a directed cycle $\gamma_1$ in $\Gamma$ with $\rho(\gamma_1) \in W$. This cycle corresponds to a periodic orbit of $\varphi$ through the symbolic coding of the flow. 
	
	However, this orbit might not be primitive or could be an element of $\Gamma^u$ or $\Gamma^s$. To avoid this issue, we want $\gamma_1$ to be primitive and not one of the finitely many primitive cycles in $\Gamma$ that code an orbit of $\Gamma^u$ or $\Gamma^s$. To do so, one can find another cycle $\gamma_2$ in $\Gamma$, with the same base point than $\gamma_1$ and such that the cycle $\gamma_1 \gamma_2^k$ is primitive for all $k$ large enough (this is possible because $\Gamma$ is not made of a single cycle, since otherwise $M$ would be made of a single periodic orbit of $\varphi$). The same argument as before implies that there are infinitely many $k$'s for which $\rho(\gamma_1 \gamma_2^k) \in W$. Excluding at most finitely many $k$, we ensure that the orbit corresponding to $\gamma_1 \gamma_2^k$ does not belong to $\Gamma^u$ or $\Gamma^s$. The orbit $\gamma$ of $\varphi$ corresponding to the cycle $\gamma_1 \gamma_2^k$ for one of those $k$ is homotopy equivalent to $\gamma_1 \gamma_2^k$ in $\widehat{M}$ and thus $\rho(\gamma) \in W$. Here, we identified $\gamma$ with a loop based at $x_0$ as specified in the statement of the lemma and we used the embedding of $\Gamma$ in $\widehat{M}$ constructed above to identify $\gamma_1 \gamma_2^k$ with a loop in $\widehat{M}$.
\end{proof}

\begin{remark}
Notice that \cref{lemma:mcmullen_realization} still holds if we replace all occurrences of $\pi_1(\widehat{M})$ by $\pi_1(M)$. Indeed, if $\rho : \pi_1(M) \to G$ is a group morphism, where $G$ is a compact group, then one may apply \cref{lemma:mcmullen_realization} to $\hat{\rho} = \rho \circ \iota$, where $\iota : \pi_1(\widehat{M}) \to \pi_1(M)$ is induced by the inclusion of $\widehat{M}$ in $M$, after noticing that $\hat{\rho}$ and $\rho$ have the same image (since $\iota$ is surjective).
\end{remark}

\subsection{Dynamical determinants following Baladi--Tsujii}\label{subsection:dynamical_determinant}

Let us explain now how to use \cite{baladi_tsujii_determinant} to construct dynamical determinants associated to the symbolic coding of the flow $\varphi$. In the spirit of \cite{rugh96}, we will then express the zeta functions associated to $\varphi$ in terms of these dynamical determinants in \cref{subsection:zeta_functions} and \cref{subsection:correcting}, which will allow us to prove the existence of a meromorphic continuation for these zeta functions and prove \cref{theorem:continuation_zeta}.

Let $m \geq 1$ and $k \in \set{0,1,2}$. For each $\theta,\vartheta \in \Theta$ such that $A_{\theta,\vartheta} = 1$, choose a compactly supported $C^\infty$ function $B_{\theta,\vartheta}$ from $\mathbb{R}^2$ to the space of $m \times m$ matrices with complex coefficients. Define then the operator $\mathcal{L}_k$ from $\oplus_{\theta \in  \Theta} \Omega^k(\mathbb{R}^2, \mathbb{C}^m)$ to itself\footnote{Where $\Omega^k(\mathbb{R}^2, \mathbb{C}^m)$ denotes the space of smooth differential $k$-form on $\mathbb{R}^2$ with values in $\mathbb{C}^m$.} defined for $u = (u_\theta)_{\theta \in \Theta}$ by
\begin{equation*}
(\mathcal{L}_k u)_\theta (x) = \sum_{\substack{\vartheta \in \Theta \\ A_{\theta,\vartheta} = 1}} B_{\theta,\vartheta}(x) (T_{\theta,\vartheta}^* u_\vartheta)(x).
\end{equation*}
One can then apply the analysis from \cite{baladi_tsujii_determinant} to find that the flat determinant defined for $|z| \ll 1$ by
\begin{equation}\label{eq:definition_determinant}
d_k(z) \coloneqq \exp\left( - \sum_{n = 1}^{+ \infty} \frac{1}{n}\tr^{\flat}(\mathcal{L}_k^n) z^n \right)
\end{equation} 
has a holomorphic continuation to $\mathbb{C}$. In this definition, we use the flat trace as an abbreviation for
\begin{equation}\label{eq:definition_flat_trace}
\tr^{\flat}(\mathcal{L}_k^n) \coloneqq \sum_{\substack{w \in \Sigma \\ \sigma^n w = w}} \frac{\tr(B_{w_0,\dots,w_n}(\pi_\Sigma(w))) \tr(\Lambda^k D T_w(\pi_{\Sigma}(w)))}{\left|\det(I - D T_w (\pi_\sigma(w)))\right|} \textup{ for } n \geq 1,
\end{equation}
where
\begin{equation*}
B_{w_0,\dots,w_n}(x) = B_{w_0,w_1}(x) B_{w_1,w_2}(T_{w_0,w_1}x) \dots B_{w_{n-1},w_n}(T_{w_0,\dots,w_{n-1}}x). 
\end{equation*}

There are some details that needs to be adjusted from \cite{baladi_tsujii_determinant}. The analysis from \cite{baladi_tsujii_determinant} does not detail how to work with sections of a vector bundle, but this is a simple addition, as mentioned in \cite[\S 2]{baladi_tsujii_determinant}. The other difference is that Baladi and Tsujii work with a family of transfer operators that come from a single hyperbolic map (see \cite[\S 5]{baladi_tsujii_determinant}). To adapt their work to our context, one only needs to replace their family of transfer operators by ours. However, their specific choice of transfer operators yields a simpler formula for the flat trace, see \cite[Proposition 6.3]{baladi_tsujii_determinant}. To find that the method of Baladi and Tsujii applied in our context produce a holomorphic continuation for the dynamical determinant \eqref{eq:definition_determinant} with flat trace \eqref{eq:definition_flat_trace}, one only needs to stop the proof of \cite[Proposition 6.3]{baladi_tsujii_determinant} before the final simplification and to notice that if $w \in \Sigma$ and $n \in \mathbb{N}^*$ are such that $\sigma^n w = w$ then $\pi_\Sigma(w)$ is the unique fixed point of $T_{w_0,\dots,w_{n-1}}$.

\subsection{Zeta function associated to the Markov partition}\label{subsection:zeta_functions}

We want to apply the approach from the previous section in order to analyse twisted zeta functions for the smooth pseudo-Anosov flow $\varphi$. We construct first a meromorphic continuation for a zeta function that counts the periodic orbits of $\varphi$ as in the symbolic coding, i.e. with some errors that we will correct in \cref{subsection:correcting}. Introduce the manifold
\begin{equation}\label{eq:widehat_manifold}
\widehat{M} \coloneqq M \setminus \bigcup_{\gamma \in \Gamma^u \cup \Gamma^s} \gamma
\end{equation}
and let $\rho$ be a representation of $\pi_1(\widehat{M})$. We will work here in a slightly more general setting than needed for Theorems \ref{theorem:continuation_zeta} and \ref{theorem:fried_conjecture} (for which we just need to work with a representation of $\pi_1(M)$). This is because we need some intermediate results to  be stated in this generality in prevision for the proof of \cref{theorem:dirichletclassnumber}.

\begin{figure}[h]
	\centering
\def\svgwidth{0.7\linewidth}
	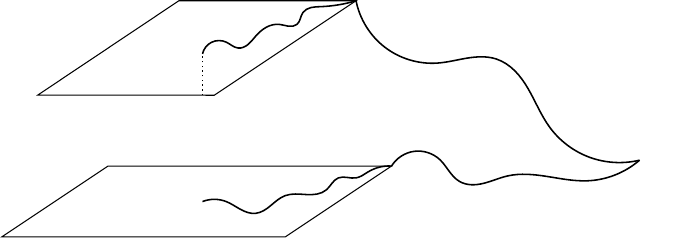
	\caption{The path $c_{\theta,\vartheta}$}\label{fig:basepoint}
\end{figure}
The first thing to do is to make a choice for the $B_{\theta,\vartheta}$'s that appear in the previous section. Let $x_0$ be the base point used to define $\pi_1(\widehat{M})$. For each $\theta \in \Theta$, choose a path $c_\theta$ in $\widehat{M} $ joining $x_0$ and a point in $\widetilde{R}_\theta$. Then, if $\theta,\vartheta \in \Theta$ are such that $A_{\theta,\vartheta} = 1$, we build a loop $c_{\theta,\vartheta}$ in the following way. We choose a point $x_{\theta,\vartheta}$ in $\widetilde{R}_{\theta,\vartheta}^*$ and let $c_{\theta,\vartheta}$ be the concatenation of : $c_\theta$, a path in $\widetilde{R}_\theta$ from the extremity of $c_\theta$ to $x_{\theta,\vartheta}$, the orbit segment from $x_{\theta,\vartheta}$ to $T x_{\theta,\vartheta}$ corresponding to the first return map, a path in $\widetilde{R}_\vartheta$ from $T x_{\theta,\vartheta}$ to the extremity of $c_\vartheta$ and finally the path $c_\vartheta$ backward. See \cref{fig:basepoint}. We let then $\rho_{\theta,\vartheta} = \rho(c_{\theta,\vartheta})$.  With $\chi$ a compactly supported $C^\infty$ function identically equal to $1$ on a neighbourhood of the $\widetilde{R}_\theta$'s, we define
\begin{equation*}
\widetilde{B}_{\theta,\vartheta}^s(x) = \chi(x) e^{ - s \mathfrak{r}_{\theta,\vartheta}(x)} \rho_{\theta,\vartheta}
\end{equation*}
for $x \in \mathbb{R}^2,s \in \mathbb{R}$ and $\theta,\vartheta \in \Theta$ such that $A_{\theta,\vartheta} = 1$.

Fixing $s \in \mathbb{C}$, we apply the analysis from \cref{subsection:dynamical_determinant} with $B_{\theta,\vartheta} = \widetilde{B}_{\theta,\vartheta}^s$. This way, we get a dynamical determinant $z \mapsto d_k(s,z)$, which is an entire function, given for $|z|$ small by the expressions \eqref{eq:definition_determinant} and \eqref{eq:definition_flat_trace} with $B_{\theta,\vartheta}$ replaced by $\widetilde{B}_{\theta,\vartheta}^s$. We want now to consider $d_{k}(s,z)$ as a function of both $s$ and $z$. We know that when $s$ is fixed, $z \mapsto d_k(s,z)$ is holomorphic in $\mathbb{C}$. Using that the number of periodic orbits of period $n$ of $\sigma$ grows at most exponentially fast with $n$, we can deduce directly from the definitions \eqref{eq:definition_determinant} and \eqref{eq:definition_flat_trace} that $d_k$ is holomorphic as a function of $s$ and $z$ on a neighbourhood of $\mathbb{C} \times \set{0}$ in $\mathbb{C}^2$. Thus, the following result implies that $d_k$ is actually holomorphic on $\mathbb{C}^2$.

\begin{lemma}
Let $f : \mathbb{C}^2 \to \mathbb{C}$ be a function such that
\begin{itemize}
\item for every $s \in \mathbb{C}$, the function $z \mapsto f(s,z)$ is holomorphic on $\mathbb{C}$;
\item there is an open neighbourhood $U$ of $\mathbb{C}\times \set{0}$ such that $f$ is holomorphic in $U$.
\end{itemize}
Then $f$ is holomorphic on $\mathbb{C}^2$.
\end{lemma}

\begin{proof}
For $n \geq 0$ and $s \in \mathbb{C}$, let $(a_n(s))_{n \geq 0}$ denote the coefficients in the power series expansion at zero for $z \mapsto f(s,z)$, i.e.
\begin{equation*}
f(s,z) = \sum_{n \geq 0} a_n(s) z^n \textup{ for } z \in \mathbb{C}.
\end{equation*}
Since $z \mapsto f(s,z)$ is entire, we have
\begin{equation}\label{eq:coefficients_entire}
\limsup_{n \to + \infty} \frac{1}{n} \log | a_n(s)| = - \infty. 
\end{equation}

Moreover, if $s_0 \in \mathbb{C}$, then there is $\epsilon > 0$ such that $\mathbb{D}(s_0,\epsilon) \times \mathbb{D}(0,2\epsilon) \subseteq U$. Thus, for $s \in \mathbb{D}(s_0,\epsilon)$ and $n \geq 0$ we have
\begin{equation*}
a_n(s) = \frac{1}{2 i \pi} \int_{\partial \mathbb{D}(0,\epsilon)} \frac{f(s,w)}{w^{n+1}}\mathrm{d}w.
\end{equation*}
It follows then that, for $n \geq 0$, the function $s \mapsto a_n(s)$ is entire. In particular, the function $s \mapsto \frac{1}{n} \log |a_n(s)|$ is subharmonic.

We deduce then from \cref{eq:coefficients_entire} and Hartogs Lemma that for every $R > 0$, there is $C > 0$ such that $|a_n(s)| \leq C R^{-n}$ for every $n \geq 0$ and $s \in \mathbb{D}(s_0,\epsilon/2)$. Thus, $f$ is holomorphic on $\mathbb{D}(s_0,\epsilon/2) \times \mathbb{D}(0,R)$. Since $R$is arbitrary, we get that $f$ is holomorphic on $\mathbb{D}(s_0,\epsilon/2) \times \mathbb{C}$. Since $s_0$ is arbitrary, the result follows.
\end{proof}

Since the $d_k$'s are holomorphic on $\mathbb{C}^2$, we know in particular the functions $s \mapsto d_k(s,1)$ are holomorphic for $k =0,1,2$. Let us then define the zeta function
\begin{equation*}
\zeta_{\rho}^{\textup{Markov}}(s) \coloneqq \frac{d_1(s,1)}{d_0(s,1)d_2(s,1)}.
\end{equation*}
Here, the denominator is a non-constant holomorphic function: for $\re s$ large enough it is non-zero (as follows from the fact that the number of periodic orbit of period $n$ for $\sigma$ is at most exponential in $n$). The function $\zeta_{\rho}^{\textup{Markov}}(s)$ is meromorphic on $\mathbb{C}$, but it is not exactly the zeta function $\zeta_\rho(s)$, due to the error when counting periodic orbits by means of a Markov partition.

\subsection{Correcting the zeta function}\label{subsection:correcting}

We want now to relate $\zeta_{\varphi,\rho}(s)$ and $\zeta_{\rho}^{\textup{Markov}}(s)$. To this end, we will assume from now on that the sets of periodic orbits $\Gamma^u$ and $\Gamma^s$ used to construct our Markov partition are disjoint. This is always possible: we know from the symbolic coding that $\varphi$ has infinitely many primitive periodic orbits.

We will start by identifying the contribution of each orbit of $\sigma$ to $\zeta_{\rho}^{\textup{Markov}}(s)$. To do so, let $P_{\Sigma}$ denote the set of periodic orbits for $\sigma$ in $\Sigma$. For $\mathcal{O} \in P_{\Sigma}$, remember the periodic orbit $\gamma_{\mathcal{O}}$ from \cref{definition:gammao} and introduce the holomorphic function:
\begin{equation*}
F_{\mathcal{O}}(s,z) = \det( I - z^m e^{- s T_{\mathcal{O}}} \Delta_{\gamma_{\mathcal{O}}} \rho(\gamma_{\mathcal{O}})).
\end{equation*}
Notice that there might be an issue with the definition of $\Delta_{\gamma_{\mathcal{O}}}$ and $\rho(\gamma_{\mathcal{O}})$. The holonomy $\Delta_{\gamma_{\mathcal{O}}}$ is not defined when $\gamma_{\mathcal{O}}$ is a singular orbit, but in that case, we can make a small homotopy to replace $\gamma_{\mathcal{O}}$ by a closed curve in $\widetilde{M}$ which is a concatenation of integral curves for $\varphi$ between the rectangles of the Markov partitions corresponding to an orbit in $\mathcal{O}$ and of paths within the rectangles (see \cref{fig:pushoff} for a similar construction). Since the interior of the rectangles do not intersect the singular orbits, the holonomy of this new orbit for $\Delta$ does not depend on any choice made, and can be used to define $\Delta_{\gamma_{\mathcal{O}}}$. Alternatively, $\Delta_{\gamma_{\mathcal{O}}}$ may be computed in the following way. For each $\theta \in \Theta$, choose an orientation for the unstable bundle over $R_\theta$ (which is possible since $R_\theta$ is simply connected). Now, if $\theta,\vartheta \in \Theta$ are such that $A_{\theta,\vartheta} = 1$, let $\Delta_{\theta,\vartheta} = 1$ if $T_{\theta,\vartheta}$ preserves the chosen orientation of the unstable direction and $\Delta_{\theta,\vartheta} = -1$ otherwise. Then we have
\begin{equation*}
\Delta_{\gamma_{\mathcal{O}}} = \prod_{j = 0}^{m-1} \Delta_{w_j,w_{j+1}}.
\end{equation*}
Similarly, $\rho(\gamma_{\mathcal{O}})$ is not defined when $\gamma_{\mathcal{O}}$ is a multiple of an element of $\Gamma^u \cup \Gamma^s$ (recall that $\rho$ is a representation of $\pi_1(\widehat{M})$ where $\widehat{M}$ is defined in \cref{eq:widehat_manifold}). We can proceed as above to homotope $\gamma_{\mathcal{O}}$ with a closed curve in $\widehat{M}$ whose free homotopy class in $\widehat{M}$ only depends on $\mathcal{O}$. We find then that $\rho(\gamma_{\mathcal{O}})$ is given (up to conjugacy) by
\begin{equation*}
\rho(\gamma_{\mathcal{O}}) = \rho_{w_0,w_1} \dots \rho_{w_{m-1},w_m}.
\end{equation*}
This formula is valid for any $\mathcal{O} \in P_{\Sigma}$, indeed, when $\gamma_{\mathcal{O}}$ lies in $\widehat{M}$, with the same notation as above,
\begin{equation*}
\rho_{w_0,w_1} \dots \rho_{w_{m-1},w_m} = \rho(c_{w_0,w_1} \cdot \dots \cdot c_{w_{m-1},w_m})
\end{equation*}
and the path $c_{w_0,w_1} \cdot \dots \cdot c_{w_{m-1},w_m}$ is freely homotopic to $\gamma_{\mathcal{O}}$ in $\widehat{M}$ since the rectangles of the Markov partition are simply connected.

The zeta function $\zeta_{\rho}^{\textup{Markov}}(s)$ is naturally described in terms of the $F_{\mathcal{O}}$'s:

\begin{proposition}\label{proposition:explicit_zeta_Markov2}
Let $s \in \mathbb{C}$. Then, there is $\delta >0$ such that for $z\in \mathbb{C}$ with $|z|<\delta$ we have
\begin{equation}\label{equation:markov_correction}
	 \frac{d_1(s,z)}{d_0(s,z)d_2(s,z)}  = \prod_{\mathcal{O} \in P_{\Sigma}} F_{\mathcal{O}}(s,z).
\end{equation}
\end{proposition}
\begin{proof}
	Start by noticing that for $z \in \mathbb{C}$ small enough we have
	\begin{equation*}
	\sum_{k = 0}^2\sum_{n \geq 1} \sum_{\substack{w \in \Sigma \\ \sigma^n w = w}} \frac{|z|^n}{n} \frac{\left|\tr(\widetilde{B}^s_{w_0,\dots,w_n}(\pi_\Sigma(w))) \tr(\Lambda^k D T_w(\pi_{\Sigma}(w)))\right|}{\left|\det(I - D T_w (\pi_\Sigma(w)))\right|} < + \infty.
	\end{equation*}
	This convergence follows from the fact that the number of periodic orbits of length $n$ of $\sigma$ grows exponentially fast with $n$. Here, $\widetilde{B}^s_{w_0,\dots,w_n}$ is defined as $B_{w_0,\dots,w_n}$ in \cref{subsection:dynamical_determinant} using the $\widetilde{B}_{\theta,\vartheta}^s$ from \cref{subsection:zeta_functions} as the $B_{\theta,\vartheta}$'s. We can use Fubini's theorem to rearrange the sums and find for $|z|$ small enough
	\begin{equation*}
	\begin{split}
	& \log \left(\frac{d_1(s,z)}{d_0(s,z)d_2(s,z)}\right) \\
	& \quad = \sum_{k = 0}^2(-1)^{k+1}\sum_{n \geq 1} \sum_{\substack{w \in \Sigma \\ \sigma^n w = w}} \frac{z^n}{n} \frac{\tr(\widetilde{B}^s_{w_0,\dots,w_n}(\pi_\Sigma(w))) \tr(\Lambda^k D T_{w_0,\dots,w_n}(\pi_{\Sigma}(w)))}{\left|\det(I - D T_{w_0,\dots,w_n} (\pi_\Sigma(w)))\right|} \\
		& \quad = - \sum_{\mathcal{O} \in P_{\Sigma}} \underbrace{\sum_{w \in \mathcal{O}} \sum_{ \substack{ n \geq 1 \\ \sigma^n w = w}} \frac{z^n \tr(\widetilde{B}^s_{w_0,\dots,w_n}(\pi_\Sigma(w))) }{n} \frac{\det(I - D T_{w_0,\dots,w_n} (\pi_\Sigma(w)))}{\left|\det(I - D T_{w_0,\dots,w_n} (\pi_\Sigma(w)))\right|}}_{C_{\mathcal{O}}(s,z) \coloneqq}.
	\end{split}
	\end{equation*}
	
	Now, fix $\mathcal{O} \in P_\Sigma$ and let $m$ be the minimal period of the elements of $\mathcal{O}$. The term $C_{\mathcal{O}}(s,z)$ from the last line becomes
	\begin{equation*}
	\begin{split}
	& C_{\mathcal{O}}(s,z) \\ & \qquad = \frac{1}{m}\sum_{w \in \mathcal{O}} \sum_{p \geq 1} \frac{z^{p m} \tr((\widetilde{B}_{w_0,\dots,w_{m}}^s(\pi_\Sigma(w)))^p)}{p} \frac{\det(I - D T_{w_0,\dots,w_{m}}^p (\pi_\Sigma(w)))}{\left|\det(I - D T_{w_0,\dots,w_{m}}^p (\pi_\Sigma(w)))\right|}.
	\end{split}
	\end{equation*}
	To simplify this expression, first notice that all the terms in the sum are equal (by cyclicity of the trace). Moreover, for $w \in \mathcal{O}$, we have
	\begin{equation*}
	\frac{\det(I - D T_{w_0,\dots,w_{m}}^p (\pi_\Sigma(w)))}{\left|\det(I - D T_{w_0,\dots,w_{m}}^p (\pi_\Sigma(w)))\right|} = \Delta_{\gamma_{\mathcal{O}}}^p.
	\end{equation*}
	Notice also that 
	\begin{equation*}
	\widetilde{B}_{w_0,\dots,w_{m}}^s(\pi_\Sigma(w)) = e^{-s T_{\mathcal{O}}}\rho_{w_0,w_1} \rho_{w_1,w_2}  \dots \rho_{w_{n-1},w_{m}}
	\end{equation*}
	Recall that $\rho(\gamma_{\mathcal{O}})$ is $\rho_{w_1,w_2}  \dots \rho_{w_{n-1},w_{m}}$ up to conjugacy. Thus, for $p \geq 1$ we have
	\begin{equation*}
	\tr((\widetilde{B}_{w_0,\dots,w_{m}}^s(\pi_\Sigma(w)))^p) = e^{-s p T_{\mathcal{O}}} \tr(\rho(\gamma_{\mathcal{O}})^p).
	\end{equation*}
	Putting everything together, we find that
	\begin{equation*}
	\begin{split}
	C_{\mathcal{O}}(s,z) & = \sum_{p \geq 1} \frac{z^{p m}}{p}e^{- s p T_{\mathcal{O}}} \Delta_{\gamma_{\mathcal{O}}}^p \tr(\rho(\gamma_{\mathcal{O}})^p) \\
		& = - \tr(\log( I - \Delta_{\gamma_{\mathcal{O}}} e^{-s T_{\mathcal{O}}} z^{m} \rho(\gamma_{\mathcal{O}}))) \\
		& = - \log(\det(I - \Delta_{\gamma_{\mathcal{O}}} e^{-s T_{\mathcal{O}}} z^{m} \rho(\gamma_{\mathcal{O}}))) = - \log F_{\mathcal{O}}(s,z)
	\end{split}
	\end{equation*}
	Notice that the convergence of the series above implies that this logarithm is well-defined.
	\end{proof}
	
	The rest of this section will be used for the proofs of Theorems \ref{theorem:continuation_zeta} and \ref{theorem:fried_conjecture}. Consequently, we will assume that $\rho$ is a representation of $\pi_1(M)$. All the constructions and results above are still valid, as we can apply them to the representation $\hat{\rho} = \rho \circ \iota$ where $\iota :\pi_1(\widehat{M}) \to \pi_1(M)$ is the morphism induced by the injection $\widehat{M} \hookrightarrow M$.
	
	In order to identify the contribution of periodic orbits in $\Gamma^u \cup \Gamma^s$ to $\zeta^{\textup{Markov}}_\rho(s)$, it will be useful to describe the contribution of a given periodic orbit in terms of stable leaves (instead of the unstable leaves). First, recall the definition in \cref{eq:auxiliary_function}:
	\begin{equation*}
	\xi_{\varphi,\rho,\gamma}(s) = \det(I-e^{-s T_\gamma }\rho(\gamma)),
	\end{equation*}
	where $T_\gamma$ denotes the length of $\gamma$. Also recall the definition of the local zeta function at $\gamma$ from \cref{eq:definition_local_zeta_function}:
	\begin{equation*}
	\zeta_{\varphi,\rho,\gamma}(s) = \frac{\prod_{i=1}^m\xi_{\varphi,\rho,\gamma^{r_i^{u}}}(s)}{\xi_{\varphi,\rho,\gamma}(s)}.
	\end{equation*}
	where $\lambda_1^{u},\dots,\lambda_n^{u}$ are the half weak unstable leaves incident to $\gamma$ and $\lambda_i^{u}$ wraps $r_i^{u}$ times around $\gamma$. Finally, recall that $\varepsilon$ is the orientation cocycle of $TM$.
		
	\begin{lemma}\label{lemma:local_zeta_stable}
	Assume that $\rho$ is a representation of $\pi_1(M)$. Let $\gamma$ be a primitive periodic orbit for $\varphi$. Suppose $\lambda_1^{s},\dots,\lambda_n^{s}$ are the half weak stable leaves incident to $\gamma$, where $\lambda_i^{s}$ wraps $r_i^{s}$ times around $\gamma$. Then
	\begin{equation}\label{eq:definition_local_zeta_function_stable}
		\zeta_{\varphi,\rho,\gamma}(s) = \frac{\prod_{i=1}^n\xi_{\varphi,\varepsilon\rho,\gamma^{r_i^{s}}}(s)}{\xi_{\varphi,\varepsilon\rho,\gamma}(s)}.
	\end{equation}
	The only difference with \cref{eq:definition_local_zeta_function} is that $\rho$ is replaced with $\varepsilon \rho$, which is $\rho$ twisted by the orientation cocycle of $TM$.
	\end{lemma}
	
	\begin{proof}
	As in the introduction, let us denote the half weak unstable leaves incident to $\gamma$ by $\lambda_1^u,\dots,\lambda_m^u$, where $\lambda_i^u$ wraps $r_i^u$ times around $\gamma$. When $\varepsilon_\gamma=1$, there is a 1-1 correspondence between weak stable half-leaves and weak unstable half-leaves incident to $\gamma$: we associate each half-weak stable leaf to the next half-weak unstable leave when turning positively around $\gamma$ (after the choice of an orientation over $\gamma$). So the expressions for $\zeta_{\varphi,\rho,\gamma}$ in terms of weak stable and weak unstable leaves are the same. 
	
	Now suppose $\varepsilon_\gamma=-1$. Then the first return map of $\gamma$ acts as a reflection $\sigma$ on the set of prongs (see \cref{remark:prong_permutation} when $\gamma$ is singular, the case of a regular orbit is similar). Since the prongs alternate between stable and unstable and $\sigma$ maps stable prongs to stable prongs, the reflection $\sigma$ must fix two prongs. There are two cases. If the number of unstable prongs is odd, then one of the fixed prongs is unstable and the other is stable, and $\sigma$ acts as an involution on the remaining prongs. Then $m=n$ and, up to reordering, $(r_1^u,\dots,r_m^u) = (1,2,2,\dots,2)$ and $(r_1^s,\dots,r_n^s)=(1,2,2,\dots,2)$. All the terms with $r=2$ are the same in \cref{eq:definition_local_zeta_function} and \cref{eq:definition_local_zeta_function_stable}, and the terms with $r=1$ cancel with the term in the denominator. So \cref{eq:definition_local_zeta_function} and \cref{eq:definition_local_zeta_function_stable} agree in this case. The second case is that the number of unstable prongs is even. Then both of the fixed prongs are of the same type, which we assume without loss of generality to be unstable. As before, $\sigma$ acts as an involution on the remaining prongs. Thus, $m=n+1$ and, up to reordering, $(r_1^u,\dots,r_m^u)=(1,1,2,\dots,2)$ and $(r_1^s,\dots,r_n^s)=(2,2,\dots,2)$. In this case, we have
	\begin{align*}
		\frac{\prod_{i=1}^m\xi_{\varphi,\rho,\gamma^{r_i^{u}}}(s)}{\xi_{\varphi,\rho,\gamma}(s)}\frac{\xi_{\varphi,\varepsilon\rho,\gamma}(s)}{\prod_{i=1}^n\xi_{\varphi,\varepsilon\rho,\gamma^{r_i^{s}}}(s)} &= \frac{\xi_{\varphi, \rho, \gamma}(s)^2 \xi_{\varphi, \varepsilon \rho, \gamma}(s)}{\xi_{\varphi, \rho, \gamma}(s)\xi_{\varphi,\varepsilon\rho, \gamma^2}(s)}\\
		&= \frac{\det(I-e^{-s T_\gamma }\rho(\gamma))\det(I+e^{-s T_\gamma }\rho(\gamma))}{\det(I-(e^{-s T_\gamma }\rho(\gamma))^2)}\\
		&=1
	\end{align*}
	So \cref{eq:definition_local_zeta_function} and \cref{eq:definition_local_zeta_function_stable} agree in this case as well.
	\end{proof}
	
	We are now ready to relate $\zeta_{\rho}^{\textup{Markov}}(s)$ and $\zeta_{\varphi,\rho}(s)$.

\begin{proposition}\label{proposition:explicit_zeta_Markov3}
	Assume that $\rho$ is a representation of $\pi_1(M)$. There is $C_0 > 0$ such that for every $s \in \mathbb{C}$ with $\re s \geq C_0$ we have
\begin{equation}\label{equation:markov_correction2}
	\zeta_{\rho}^{\textup{Markov}}(s) = \zeta_{\varphi,\rho}(s)\prod_{\gamma \in \Gamma^s} \xi_{\varphi,\rho,\gamma}(s) \prod_{\gamma \in \Gamma^u} \xi_{\varphi,\varepsilon \rho, \gamma}(s).
\end{equation}
\end{proposition}

\begin{proof}
	For $\re s$ large enough, both sides converge absolutely. According to \cref{proposition:explicit_zeta_Markov2}, we need to prove that, for $\re s$ large we have
	\begin{equation}\label{equation:markov_correction3}
	\prod_{\mathcal{O} \in P_{\Sigma}} F_{\mathcal{O}}(s,1) = \prod_{\gamma \textup{ primitive orbit }} \zeta_{\varphi,\rho,\gamma}(s)\prod_{\gamma \in \Gamma^s} \xi_{\varphi,\rho,\gamma}(s) \prod_{\gamma \in \Gamma^u} \xi_{\varphi,\varepsilon \rho, \gamma}(s).
\end{equation}
	It follows from \cref{lemma:no_duplicate_regular} that any closed orbit not in $\Gamma^u \cup \Gamma^s$ has exactly one symbolic coding, and is therefore counted exactly once on the left side of \cref{equation:markov_correction3}. By \cref{remark:nonsingular_zeta}, this agrees with the contribution to the right side. 
	
	Now consider an orbit $\gamma \in \Gamma^s$. Using the notation established in \cref{eq:definition_local_zeta_function}, let $\lambda_1^u,\dots,\lambda_m^u$ be the weak unstable half-leaves incident to some power of $\gamma$, and let $r_i^u$ be the number of times $\partial \lambda_i^u$ wraps around $\gamma$. The contribution of $\gamma$ to the right side of \cref{equation:markov_correction3} is \begin{equation*}
		\zeta_{\varphi,\rho,\gamma}(s) \xi_{\varphi,\rho,\gamma}(s) = \prod_{i=1}^m \xi_{\varphi,\rho,\gamma^{r_i^u}}(s).
	\end{equation*}	
	We turn to weighing the contribution of $\gamma$ to the left side of \cref{equation:markov_correction3}. By construction, $\gamma$ can only intersect a rectangle on its stable boundary. Since $\Gamma^u$ is disjoint from $\Gamma^s$, the orbit $\gamma$ does not intersect the corner of any rectangle. For each rectangle which $\gamma$ touches, there is a corresponding word $w\in \Sigma$ which codes some multiple of $\gamma$. We wish to count the equivalence classes of such words under the shift. As in the notation established in \cref{eq:definition_local_zeta_function},  For each $i\in 1,\dots,m$, push $\gamma^{r_i^u}$ slightly into the interior of $\lambda_i^u$ and record the sequence of rectangles pierced by this curve. This is an orbit of $\sigma$. Moreover, every rectangle intersecting $\gamma$ is pierced by exactly one of these pushoffs. Therefore, these $m$ orbits are all the orbits of $\sigma$ which code some multiple of $\gamma$. Note that the weak unstable bundle is orientable over $\gamma^{r_i^u}$. So the contribution of $\gamma$ to the left side of \cref{equation:markov_correction3} is
	\begin{equation*}
		\prod_{i=1}^m \xi_{\varphi,\rho,\gamma^{r_i^u}}(s).
	\end{equation*}
	This agrees  with the contribution to the right side of \cref{equation:markov_correction3}.

	Now consider an orbit $\gamma \in \Gamma^u$, and define $\lambda_1^s,\dots,\lambda_n^s$ and $r_1^s,\dots,r_n^s$ in the analogous way. Arguing as before, the weak stable bundle is orientable over $\gamma^{r_i^s}$. The orientation class of the weak stable bundle is $\varepsilon\Delta$, so $\varepsilon_{\gamma^{r_i^s}} \Delta_{\gamma^{r_i^s}}=1$. So the contribution of $\gamma$ to the left side of \cref{equation:markov_correction3} is
	\begin{equation*}
		\prod_{i=1}^m \xi_{\varphi,\varepsilon\rho,\gamma^{r_i^s}}(s).
	\end{equation*}
	This agrees with the contribution to the right side of \cref{equation:markov_correction3}.
\end{proof}

We can finally prove the existence of a meromorphic extension to $\mathbb{C}$ for $\zeta_{\varphi,\rho}(s)$.

\begin{proof}[Proof of \cref{theorem:continuation_zeta}]
It follows from \eqref{eq:definition_local_zeta_function}, \eqref{eq:definition_local_zeta_function_stable} , \Cref{proposition:explicit_zeta_Markov2}, and \Cref{proposition:explicit_zeta_Markov3} that for $\re s \gg 1$ we have
\begin{equation}\label{eq:Markov_to_standard}
\zeta_{\varphi,\rho}(s) = \zeta_{\rho}^{\textup{Markov}}(s) \times \left(\prod_{\gamma \in \Gamma^s} \xi_{\varphi,\rho,\gamma}(s) \prod_{\gamma \in \Gamma^u} \xi_{\varphi,\varepsilon\rho,\gamma}(s) \right)^{-1}.
\end{equation}
The function $\zeta_{\rho}^{\textup{Markov}}(s)$ has been defined in \cref{subsection:zeta_functions} as a meromorphic function on $\mathbb{C}$. The other factor is also meromorphic on $\mathbb{C}$, as can be seen from the definition \eqref{eq:auxiliary_function}.
\end{proof}

\subsection{Value at zero of zeta functions}\label{subsection:value_at_zero}

In order to prove \cref{theorem:fried_conjecture}, we need to compute $\zeta_{\varphi,\rho}(0)$ for $\rho$ a representation satisfying the hypotheses of the theorem. Thanks to \cref{eq:Markov_to_standard}, this can be done by computing $\zeta_\rho^{\textup{Markov}}(0)$ first. This is the point of the following lemma. We will also need this lemma for the proof of \cref{theorem:dirichletclassnumber}, so in this section we work again with $\rho$ a representation of $\pi_1(\widehat{M})$.

\begin{lemma}\label{lemma:zeta_Markov_at_zero}
If $d_0(0,1) \neq 0$ and $d_2(0,1) \neq 0$ then
\begin{equation*}
\zeta_\rho^{\textup{Markov}}(0) = \det(I - H)
\end{equation*}
where the matrix $H : \mathbb{C}^\Theta \otimes \mathbb{C}^m \to \mathbb{C}^\Theta \otimes \mathbb{C}^m$ is defined for $v = (v_\theta)_{\theta \in \Theta}$ by
\begin{equation*}
(Hv)_\theta= \sum_{\substack{\vartheta \in \Theta : A_{\theta,\vartheta} = 1}} \Delta_{\theta,\vartheta} \rho_{\theta,\vartheta} u_\vartheta,
\end{equation*} 
where\footnote{As mentioned in the beginning of \cref{subsection:correcting}, one may choose an orientation of the unstable direction on each rectangle, since rectangles are simply connected.}
\begin{equation*}
\Delta_{\theta,\vartheta} = \begin{cases} 1 & \textup{ if } T_{\theta,\vartheta} \textup{ preserves the orientation of the unstable direction}\\ - 1 & \textup{ otherwise.} \end{cases}
\end{equation*}
\end{lemma}

\begin{proof}
As in the proof of \Cref{proposition:explicit_zeta_Markov2}, we notice that for $|z|$ small enough, we have
\begin{equation*}
\begin{split}
& \log\left(\frac{d_1(0,z)}{d_0(0,z)d_2(0,z)}\right) \\ & \qquad \quad = - \sum_{n \geq 1} \sum_{\substack{w \in \Sigma \\ \sigma^n w = w}} \frac{z^n}{n} \tr(\widetilde{B}^0_{w_0,\dots,w_n}(\pi_{\Sigma}(w))) \Delta_{w_0,w_1} \dots \Delta_{w_{n-1},w_n}.
\end{split}
\end{equation*}
But, since we have now $s = 0$, we find that
\begin{equation*}
\tr(\widetilde{B}^0_{w_0,\dots,w_n}(\pi_{\Sigma}(w)))  = \rho_{w_0,w_1} \dots \rho_{w_{n-1},w_n}
\end{equation*}
and thus
\begin{equation*}
\log\left(\frac{d_1(0,z)}{d_0(0,z)d_2(0,z)}\right) = - \sum_{n \geq 1} \frac{z^n}{n} \tr(H^n) = \log(\det(I - zH)).
\end{equation*}
Consequently, for $|z|$ small enough, we have
\begin{equation*}
\frac{d_1(0,z)}{d_0(0,z)d_2(0,z)} = \det(I - z H).
\end{equation*}
Thanks to our assumption that $d_0(0,1)$ and $d_2(0,1)$ are non-zero, we find that
\begin{equation*}
\zeta_{\rho}^{\textup{Markov}}(0) = \frac{d_1(0,1)}{d_0(0,1)d_2(0,1)} = \det(I - H).
\end{equation*}
\end{proof}

\begin{remark}
Let us explain why the conditions $d_0(0,1) \neq 0$ and $d_2(0,1) \neq 0$ are needed in the proof of \cref{lemma:zeta_Markov_at_zero}. Consider the holomorphic functions of two variables $f(s,z) = 1 - z$ and $g(s,z) = 1 - e^{-s} z$. Fixing $z = 1$, one may define the meromorphic function $\zeta_{\textup{example}}(s) = f(s,1)/g(s,1)$. Indeed, $g(s,1)$ is not identically equal to $0$. A direct inspection gives that $\zeta_{\textup{example}}$ is identically equal to $0$, in particular the singularity at $s = 0$ is removable. However, if we fix $s = 0$, then for $z$ small, we find that $f(0,z)/g(0,z) = 1$, and we see that we cannot take $z = 1$ in this equality to compute $\zeta_{\textup{example}}(0)= 0$.
\end{remark}

\cref{lemma:zeta_Markov_at_zero} is of no use without a way to check the conditions $d_0(0,1) \neq 0$ and $d_2(0,1) \neq 0$. The first condition is dealt with in the following lemma.

\begin{lemma}\label{lemma:denominator_non_zero}
Assume that $\rho$ is a unitary representation of $\pi_1(\widehat{M})$. If $d_0(0,1)$ is equal to zero, then $H^0(\widehat{M};\rho)\neq 0$. If $\rho$ extends to a representation\footnote{I.e. if $\rho(\gamma) = I$ whenever $\gamma$ is homotopically trivial in $M$.} of $\pi_1(M)$, then $H^0(M;\rho)\neq 0$.
\end{lemma}

For some background on twisted homology, see \cref{subsection:twisted_homology}. Let us prepare the proof of Lemma \ref{lemma:denominator_non_zero} with:

\begin{lemma}\label{lemma:holder_direction}
The map $w = (w_n)_{n \in \mathbb{Z}} \mapsto E_{w_0}^u(\pi_\Sigma(w))$ from $\Sigma$ to $\mathbb{P}^1(\mathbb{R})$ (the space of lines in $\mathbb{R}^2$) is H\"older-continuous\footnote{We endow $\mathbb{P}^1(\mathbb{R})$ with any smooth metric.}.
\end{lemma}

\begin{proof}
For each $\theta \in \Theta$ and $x \in R_\theta$, let us choose a unit vector $v_\theta(x) \in \mathbb{R}^2$ that spans $E^u_\theta(x)$. It follows from Lemma \ref{lemma:continuity_unstable} and the fact that $R_\theta$ is simply connected that we can choose $v_\theta$ as a continuous function of $x$. Let then $\tilde{v}_\theta$ be a function from $R_\theta$ to $\mathbb{R}^2$ obtained by mollification of $v_\theta$. Hence, $\tilde{v}_\theta$ is a smooth approximation of $v_\theta$. If the approximation is good enough, then for each $x \in R_\theta$, the projection of $\tilde{v}_\theta(x)$ on $E^u_\theta(x)$ along $E^s_\theta(x)$ is non-zero (and by compactness of $R_\theta$, its size is even bounded below).

If $w \in \Sigma$ and $n \in \mathbb{N}$, let $a_n(w) v_{w_0}(\pi_\Sigma(w))$ denote the projection of 
\begin{equation}\label{eq:approximate_unstable}
\frac{D T_{w_{-n},\dots,w_0}(\pi_\Sigma(\sigma^{-n}(w))) \cdot \tilde{v}_{w_{-n}}(\pi_\Sigma(\sigma^{-n}(w))))}{|D T_{w_{-n},\dots,w_0}(\pi_\Sigma(\sigma^{-n}(w))) \cdot \tilde{v}_{w_{-n}}(\pi_\Sigma(\sigma^{-n}(w))))|}
\end{equation}
on $E^u_{w_0}(\pi_\Sigma(w))$ along $E^s_{w_0}(\pi_\Sigma(w))$. It follows from \eqref{eq:symbolic_hyperbolicity} that the function $|a_n|$ converges to $1$ exponentially fast in the supremum norm when $n$ goes to $+ \infty$.

Let us now consider $w,w' \in \Sigma$ close. Let $m_0$ be the infimum of the $m$ in $\mathbb{N}$ such that $w_m \neq w_{-m}'$ or $w_{-m} \neq w_{-m}'$. Let $n$ be an integer between $1$ and $m_0-1$. Notice that the map
\begin{equation*}
x \mapsto \frac{D T_{w_{-n},\dots,w_0}(T_{w_{-n},\dots,w_0}^{-1}(x)) \cdot \tilde{v}_{w_{-n}}(T_{w_{-n},\dots,w_0}^{-1}(x))}{|D T_{w_{-n},\dots,w_0}(T_{w_{-n},\dots,w_0}^{-1}(x)) \cdot \tilde{v}_{w_{-n}}(T_{w_{-n},\dots,w_0}^{-1}(x))|}
\end{equation*}
is smooth on $T_{w_{-n},\dots,w_0}(R_{w_{-n},\dots,w_0})$, and its derivatives blow up at most exponentially fast with $n$. Hence, there is $C > 1$ that does not depend on $w,w'$ and $n$ such that the distance between \eqref{eq:approximate_unstable} and the same quantity with $w$ replaced by $w'$ is at most $C^{1 + n} d(\pi_\Sigma(w),\pi_\Sigma(w'))$. Since $\pi_\Sigma$ is H\"older-continuous, we find that for some $\alpha > 0$, this distance is bounded by $C^{1 + n}2^{-\alpha m_0}$ (maybe changing the value of $C$).

Since we know that there is $\delta \in (0,1)$ such that \eqref{eq:approximate_unstable} approximate the vector $\pm v_{w_0}(\pi_\Sigma(w))$ with a precision $\delta^{-1 + n}$, and the same is true when we replace $w$ by $w'$ everywhere, we find that the distance between $E_{w_0}^u(\pi_\Sigma(w))$ and $E_{w_0}^u(\pi_\Sigma(w'))$ is controlled by
\begin{equation*}
C^{1 + n} 2^{- \alpha m_0} + 2 \delta^{-1 + n}.
\end{equation*}
The result then follows by choosing $n$ to be the integral part of $\frac{\alpha \log 2}{2 \log C} m_0$.
\end{proof}

\begin{proof}[Proof of Lemma \ref{lemma:denominator_non_zero}]
For $\theta,\vartheta$ such that $A_{\theta,\vartheta} = 1$, and $x \in R_{\theta,\vartheta}$ define 
\begin{equation*}
J^u_{\theta,\vartheta} (x) = | \det D T_{\theta,\vartheta} (x)_{| E^u_\theta(x)}| \textup{ and } J^s_{\theta,\vartheta} (x) = | \det D T_{\theta,\vartheta} (x)_{| E^s_\theta(x)}|. 
\end{equation*}
It follows then from the hyperbolicity of the first return map $T$ that
\begin{equation*}
\tr^\flat( \mathcal{L}_0^n) = \sum_{\substack{w \in \Sigma \\ \sigma^n w = w}} \frac{\tr(\rho_{w_0,\dots,w_n})}{J^u_{w_0,\dots,w_n}(\pi_{\Sigma}(w))} ( 1 + \mathcal{O}(\eta^n))
\end{equation*}
where $J^u_{w_0,\dots,w_n} = J^u_{w_0,w_1} J^u_{w_1,w_2} \circ T_{w_1,w_2} \dots J^u_{w_{n-1}, w_n} \circ T_{w_1,\dots,w_{n-1}}$ and $\eta \in (0,1)$. It follows that the holomorphic function
\begin{equation*}
\tilde{d}_0 (z) = \exp\left( - \sum_{n = 1}^{+ \infty} \frac{1}{n} z^n  \sum_{\substack{w \in \Sigma \\ \sigma^n w = w}} \frac{\tr(\rho_{w_0,\dots,w_n})}{J^u_{w_0,\dots,w_n}(\pi_{\Sigma}(w))} \right)
\end{equation*}
extends as a holomorphic function on $\set{z \in \mathbb{C} : |z| < \eta^{-1}}$ with the same zeroes as $d_0(0,z)$. In particular, we have $\tilde{d}_0(1) = 0$.

Define then the function $\widetilde{J}^u$ on $\Sigma$ by $\widetilde{J}^u (w) = J^u_{w_0,w_1} \circ \pi_\Sigma(w)$ and notice that the function $\widetilde{J}^u$ is H\"older-continuous (according to Lemma \ref{lemma:holder_direction}) and positive. Consequently, there are two real-valued H\"older functions $g_u,f_u$ on $\Sigma$ such that $\widetilde{J}^u = e^{-g_u + f_u \circ \sigma - f_u}$, where $g_u$ only depends on the coordinates of non-negative index of its argument \cite[Proposition 1.2]{parry_pollicott_asterisque}.

Let then 
\begin{equation*}
\Sigma^+ = \set{ (\theta_m)_{m \in \mathbb{N}} \in \Theta^{\mathbb{N}} : A_{\theta_m, \theta_{m+1}} = 1 \textup{ for every } m \in \mathbb{N}},
\end{equation*}
and notice that $g_u$ may be considered as a function on $\Sigma^+$.  Letting $\sigma$ also denote the shift on $\Sigma_+$, we can rewrite the function $\tilde{d}_0(z)$ as
\begin{equation*}
\tilde{d}_0 (z) = \exp\left( - \sum_{n = 1}^{+ \infty} \frac{1}{n} z^n  \sum_{\substack{w \in \Sigma_+ \\ \sigma^n w = w}} e^{\sum_{k = 0}^{n-1} g_u(\sigma^k w)}\tr({}^t \rho_{w_{n-1},w_n} \dots {}^t \rho_{w_0,w_1}) \right)
\end{equation*}

Let $ \beta \in(0,1)$ be close to $1$ and define the spaces $\mathcal{F}^\beta(\Sigma^+)$ and $\mathcal{F}^\beta(\Sigma^+,\mathbb{C}^m)$ of Lipschitz-continuous functions on $\Sigma_+$ for the distance
\begin{equation*}
d((\theta_m)_{m \in \mathbb{N}}, (\tilde{\theta}_m)_{m \in  \mathbb{N}}) = \beta^{\inf \set{m \in \mathbb{N}: \theta_m \neq \tilde{\theta}_m}}
\end{equation*}
with values in $\mathbb{C}$ and $\mathbb{C}^m$ respectively. Let then define the operators $L$ and $\widetilde{L}$ acting respectively on $\mathcal{F}^\beta(\Sigma^+)$ and $\mathcal{F}^\beta(\Sigma^+, \mathbb{C}^m)$ by
\begin{equation*}
L u (x) = \sum_{\sigma y = x} e^{g_u(y)} u(y) 
\end{equation*}
and
\begin{equation*}
\widetilde{L} u(x) = \sum_{\sigma y = x} e^{g_u(y)} \, {}^t \rho_{y_0,y_1} u(y).
\end{equation*}
One may wonder here why we need to work with the transpose of $\rho$ all of a sudden. This is because after moving to $\Sigma_+$, we are now studying an expanding (instead of hyperbolic) dynamical system, and it is more convenient for expanding dynamical systems to study transfer operators (as $L$ and $\widetilde{L}$) than Koopman operators. Transfer operators are the adjoint of Koopman operators, and it is consequently natural when working with vector valued functions that a transpose appears when we go from one to the other.

Since $g_u$ differs from $\log \widetilde{J}^u$ by a cocycle, it follows from a standard volume argument in hyperbolic dynamics that the spectral radius of $L$ is $1$ (this is proved in \cref{lemma:spectral_radius_volume} below). Moreover, it follows from the Ruelle--Perron--Frobenius Theorem (see for instance \cite[Theorem 2.2]{parry_pollicott_asterisque}) that $1$ is a simple eigenvalue of $L$ associated to a positive eigenvector $h$, and that there is a corresponding right eigenvector  that is a probability measure $\nu$ on $\Sigma_+$ with full support. In addition, the fact that $\tilde{d}_0(1) = 0$ implies that $1$ is also an eigenvalue for the operator $\widetilde{L}$ (see \cite{pollicott_zeta} or \cite{haydn_zeta} for the scalar case, the tensor product does not introduce specific difficulties, it is dealt with for instance in \cite{ruelle_thermodynamic_expanding} in a slightly different context). Let $u$ be an eigenvector for $\widetilde{L}$ associated to the eigenvalue $1$.

Notice then that for $x \in \Sigma^+$ we have (by the triangle inequality)
\begin{equation}\label{eq:triangle_inequality_transfer_operator}
\n{u(x)} = \n{\widetilde{L} u(x)} \leq (L \n{u}) (x).
\end{equation}
Moreover, recall the left eigenvector $\nu$ for $L$ associated to the eigenvalue $1$, which is a probability measure with full support, and notice that
\begin{equation*}
\int_{\Sigma^+} (L \n{u} - \n{u} )\mathrm{d}\nu = 0.
\end{equation*}
Thus $L \n{u} = \n{u}$. Hence, $\n{u}$ is proportional to $h$. Without loss of generality, we can assume that $h = \n{u}$. Consequently, there is a function $\alpha : \Sigma^+ \to \mathbb{C}^m$, with values in the sphere, such that $u = h \alpha$.

From the equality in the triangle inequality in \cref{eq:triangle_inequality_transfer_operator}, we find that for every $w = w_0 w_1\dots \in \Sigma^+$, we have
\begin{equation}\label{eq:invariance_alpha}
\alpha(\sigma w) = {}^t \rho_{w_0,w_1} \alpha(w).
\end{equation}
Iterating this relation, we find that for any $w \in \Sigma_+$ and $n \in \mathbb{N}$, we have
\begin{equation}\label{eq:relation_homology}
\alpha(\sigma^n w) = {}^t (\rho_{w_0,w_1} \dots \rho_{w_{n-1}, w_n}) \alpha(w).
\end{equation}

Now we claim that $\alpha(w)$ depends only on $w_0$. It suffices to prove this for periodic words since periodic words are dense in $\Sigma^+$. Suppose  $y$ and $z$ are two periodic words in $\Sigma^+$, with $y_0=z_0$. Consider the word $y_{\textup{per}}^n z$ in $\Sigma^+$ where $y_{\textup{per}}$ is a prefix of length $q$ for $y$ and $q$ a period for $y$. Note that 
\begin{align*}
	\alpha(y) &= \lim_{n\to \infty} \alpha(y_{\textup{per}}^n z)\\\
	&=\lim_{n\to \infty} {}^t(\rho_{y_0,y_1} \dots \rho_{y_{m-1}, y_m})^n \alpha(z)
\end{align*}
Since $\rho_{y_0,y_1} \dots \rho_{y_{m-1}, y_m}$ is unitary, $(\rho_{y_0,y_1} \dots \rho_{y_{m-1}, y_m})^n$ is often very close to the identity. So we must have $\alpha(y)=\alpha(z)$.

We can consequently define a map $a : \Theta \to \mathbb{C}^m$ by setting $a(\theta) = \overline{\alpha(w)}$ for $\theta \in \Theta$ and $w \in \Sigma_+$ with $w_0 = \theta$. It follows then from \eqref{eq:invariance_alpha} that for every $\theta,\vartheta \in \Theta$ such that $A_{\theta,\vartheta} = 1$ we have $a(\theta) = \rho_{\theta,\vartheta} a(\vartheta)$ (here we use that $\rho$ is unitary).

As in the proof of \cref{lemma:mcmullen_realization}, we embed the graph $\Gamma$ associated to our Markov partition in $\widehat{M}$. The map $\pi_1(\Gamma) \to \pi_1(\widehat{M})$ allows to deduce from $\rho$ a representation $\rho_{\Gamma} : \pi_1(\Gamma) \to U(m)$. Now, if $\theta \in \Theta$ is the vertex of $\Gamma$ used as a basepoint for $\pi_1(\Gamma)$, then $a(\theta)$ is fixed by the image of $\rho_{\Gamma}$. As noted in the proof of \cref{lemma:mcmullen_realization}, $\pi_1(\Gamma) \to \pi_1(\widehat{M})$ is a surjection. Therefore, $a(\theta)$ is a global fixed vector for $\rho:\pi_1(\widehat{M})\to U(m)$. This gives us a nontrivial element of $H^0(\widehat{M}; \rho)$ (see the discussion in \cref{subsection:twisted_homology}). In turn, we have a surjection $\pi_1(\widehat{M})\to \pi_1(M)$. So if $\rho$ comes from a representation on $M$, then we have a global fixed vector for $\rho:\pi_1(M) \to U(m)$, and $H^0(M;\rho)\neq 0$.
\end{proof}

Let us prove now the spectral radius estimates that was delayed in the proof of \cref{lemma:denominator_non_zero}. We rely on a standard strategy, commonly used to prove that the pressure of the opposite unstable jacobian of a closed hyperbolic system is zero.

\begin{lemma}\label{lemma:spectral_radius_volume}
The spectral radius of the operator $L$ appearing in the proof of \cref{lemma:denominator_non_zero} is $1$.
\end{lemma}

\begin{proof}
It follows from the Ruelle--Perron--Frobenius Theorem \cite[Theorem 2.2.]{parry_pollicott_asterisque} that, if $r > 0$ denote the spectral radius of the operator $L$ above, then $r$ is a simple eigenvalue for $L$ associated to a positive eigenfunction $h$. Moreover, the associated left eigenvector is a measure. 

Since $h$ is continuous and positive, and $\Sigma_+$ is compact, there is $C > 0$ such that $C^{-1} < h < C$. It follows from the positivity of the transfer operator $L$ that
\begin{equation}\label{eq:reduction_to_one}
C^{-1} L^n 1(w) \leq r^n h(w) \leq C L^n 1 (w) \textup{ for every } w \in \Sigma_+.
\end{equation}

Let $w \in \Sigma_+$ and $n \geq 1$. For each $\theta \in \Theta$, choose a point $x_\theta \in R_\theta$. For $(v_0,\dots,v_n) \in \Sigma_{n+1}$ with $v_n = w_0$, consider the piece of unstable manifold $W_{v_0,\dots,v_n} \coloneqq W_{R_{v_0}}^u(x_{v_0}) \cap R_{v_0,\dots,v_n}$. The change of variable formula and the Markov property yield
\begin{equation}\label{eq:change_of_variable}
|W_{v_0,\dots,v_n}| = \int_{W_{R_{v_n}}^u(T_{v_0,\dots,v_n} x_{v_0})} \frac{1}{J^u_{v_0,\dots,v_n} \circ T_{v_0,\dots,v_n}^{-1}} \mathrm{d}m,
\end{equation}
where $m$ denote the arclength metric. Write then $x_{v_0} = \pi_{\Sigma}(\tilde{w})$ with $\tilde{w}_0 = v_0$, and notice that $W_{v_0,\dots,v_n}$ is the image by $\pi_\Sigma$ of the sets of words of the form $z = \dots \tilde{w}_{-m} \dots \tilde{w}_{-1} v_0 v_1 \dots v_n \tilde{v}$ where $\tilde{v}$ is any word such that the result is in $\Sigma$. For such a $z$, we have, recalling $\log \widetilde{J}^u = - g_u + f_u \circ \sigma - f_u$, that
\begin{equation*}
S_n g_u(z) \coloneqq \sum_{k = 0}^{n-1} g_u(\sigma^k z) = f_u(\sigma^n(z)) - f_u(z) - \log J^u_{v_0,\dots,v_n}(\pi_{\Sigma}(z)).
\end{equation*}
Since $f_u$ is bounded, there is a constant $C_0 > 0$ (that does not depend on $z$, nor $n$, nor $w$), such that
\begin{equation*}
\frac{C_0^{-1}}{J^u_{v_0,\dots,v_n}(\pi_{\Sigma}(z))} \leq e^{S_n g_u(z)} \leq \frac{C_0}{J^u_{v_0,\dots,v_n}(\pi_{\Sigma}(z))}.
\end{equation*}
By a standard bounded distortion argument, we find that there is also $C_1 > 0$ such that
\begin{equation*}
C_1^{-1} \leq e^{S_n g_u(v_0\dots v_{n-1} w) - S_n g_u(z)} \leq C_1.
\end{equation*}
The last two estimates together give that for every $x \in W_{v_0,\dots,v_n}$, we have
\begin{equation*}
\frac{C_0^{-1} C_1^{-1}}{J^u_{v_0,\dots,v_n}(x)} \leq e^{S_n g_u(v_0 \dots v_{n-1} w)} \leq \frac{C_0 C_1}{J^u_{v_0,\dots,v_n}(x)} .
\end{equation*}
Hence, we deduce from \cref{eq:change_of_variable} that 
\begin{equation*}
\begin{split}
& \frac{C_0^{-1} C_1^{-1}}{|W_{R_{w_0}}^u(T_{v_0,\dots,v_n}(x_{v_0}))|} |W_{v_0,\dots,v_n}| \leq e^{S_n g_u(v_0 \dots v_{n-1} w)} \\ & \qquad \qquad \qquad \qquad \qquad \qquad \qquad \qquad \qquad \leq \frac{C_0 C_1}{|W_{R_{w_0}}^u(T_{v_0,\dots,v_n}(x_{v_0}))|} |W_{v_0,\dots,v_n}|.
\end{split}
\end{equation*}
Since the unstable manifolds are uniformly smooth, we know that the length $|W_{R_{w_0}}^u(T_{v_0,\dots,v_n}(x_{v_0}))|$ is uniformly bounded from above and from below. 

It follows then that there is another constant $C_2 > 0$ such that
\begin{equation}\label{eq:length_comparison}
C_2^{-1} \sum_{ \substack{y \in \Sigma \\ \sigma^n y = w} } |W_{y_0,\dots,y_n}| \leq L^n 1 (w) \leq C_2 \sum_{ \substack{y \in \Sigma \\ \sigma^n y = w} } |W_{y_0,\dots,y_n}|.
\end{equation}

Now, notice that
\begin{equation*}
\bigcup_{\substack{y \in \Sigma \\ \sigma^n y = w} }W_{y_0,\dots,y_n} \subseteq \bigcup_{\theta \in \Theta} W^u_{R_\theta} (x_\theta)
\end{equation*}
and thus
\begin{equation*}
\sum_{ \substack{y \in \Sigma \\ \sigma^n y = w} } |W_{y_0,\dots,y_n}| \leq \sum_{\theta \in \Theta} |W^u_{R_\theta}(x_\theta)|
\end{equation*}
is bounded above independently on $w$ and $n$. With \cref{eq:reduction_to_one}, this implies that $r \leq 1$ (choosing any value for $w$).

For each $\theta \in \Theta$, choose $w_\theta \in \Sigma_+$ that starts with $\theta$. It follows from \cref{eq:length_comparison} that
\begin{equation*}
\sum_{\theta \in \Theta} |W^s_{R_\theta}(x_\theta)| = \sum_{(v_0,\dots,v_n) \in \Sigma_{n+1}} |W_{v_0,\dots,v_n}| \leq C_2 \sum_{\theta \in  \Theta} L^n 1(w_\theta).
\end{equation*}
And since the left-hand side of this inequality does not depend on $n$, it follows from \cref{eq:reduction_to_one} that $r \geq 1$, and thus $r = 1$.
\end{proof}

Finally, let us end this section with a result dealing with the second condition, $d_2(0,1) \neq 0$, from \cref{lemma:zeta_Markov_at_zero}. We recall that $\varepsilon$ is our notation for the holonomy of the orientation bundle of $TM$.

\begin{lemma}\label{lemma:denominator_non_zero_two}
	Assume that $\rho$ is a unitary representation of $\pi_1(\widehat{M})$. If $d_2(0,1)$ is equal to zero, then $H^0(\widehat{M};\varepsilon \rho)\neq 0$. If $\rho$ extends to a representation of $M$, then $H^0(M;\varepsilon\rho) \simeq H_3(M;\rho)\neq 0$.
\end{lemma}

\begin{proof}
Let us explain how one can adapt the proof \cref{lemma:denominator_non_zero} to this case. We also make a proof by contrapositive and assume that $d_2(0,1) = 0$. As in the proof of \cref{lemma:denominator_non_zero}, we see that this assumption implies that the holomorphic extension of
\begin{equation*}
\tilde{d}_2(z) = \exp\left( - \sum_{n = 1}^{+ \infty} \frac{1}{n} z^n  \sum_{\substack{w \in \Sigma \\ \sigma^n w = w}} \frac{\tr(\rho_{w_0,\dots,w_n}) \det D T_{w}(\pi_{\Sigma}(w))}{J^u_{w_0,\dots,w_n}(\pi_{\Sigma}(w))} \right)
\end{equation*}
has a zero at $z = 1$. Choose an orientation in the neighbourhood of each rectangle. Define $\orient_{\theta,\vartheta}=\pm 1$ depending on whether the orientations at $R_\theta$ and $R_\vartheta$ agree or disagree on their common flow box. Let us define for $(w_0,\dots,w_{n}) \in \Sigma_{n+1}$ the sign $\orient_{w_0,\dots,w_{n}} = \prod_{k = 0}^{n-1} \orient_{w_k,w_{k+1}}$. With this notation, we find that
\begin{equation*}
\tilde{d}_2(z) = \exp\left( - \sum_{n = 1}^{+ \infty} \frac{1}{n} z^n  \sum_{\substack{w \in \Sigma \\ \sigma^n w = w}} \tr(\rho_{w_0,\dots,w_n}) \orient_{w_0,\dots,w_n} J^s_{w_0,\dots,w_n}(\pi_{\Sigma}(w)) \right).
\end{equation*}
Here, we see that $\tilde{d}_2(z)$ looks like $\tilde{d}_0(z)$ except that we are working with the backward flow instead of the flow itself and that the representation $\rho$ is twisted by the signs $\orient_{\theta,\vartheta}$.

Hence, working as in the proof of \cref{lemma:denominator_non_zero}, we find that there is a map $\alpha: \Theta \to \C^m$ such that for any edge $(w_0, w_1)$ of $\Gamma$, we have
\begin{equation*}
	\alpha(w_0) = \overline{\rho}_{w_0,w_1}\orient_{w_0,w_1}\alpha(w_1).
\end{equation*}
Therefore, $\rho$ contains a 1-dimensional subrepresentation isomorphic to the orientation representation $\varepsilon$, and $H^0(\widehat{M};\varepsilon \rho)\neq 0$. Here, we are using again the fact that $\pi_1(\Gamma) \to \pi_1(\widehat{M})$ is surjective, as in the proof of \cref{lemma:mcmullen_realization}. If $\rho$ comes from a representation on $M$, then as in \cref{lemma:denominator_non_zero} we obtain a nontrivial element of $H^0(M; \varepsilon\rho)$. By twisted Poincar\'e duality (\cref{equation:twisted_duality}), this gives a nontrivial element of $H_3(M;\rho)$.
\end{proof}

\section{Zeta functions and Reidemeister torsion}\label{section:rtorsion}

This final section is dedicated to the proof of Theorems \ref{theorem:fried_conjecture} and \ref{theorem:dirichletclassnumber}. We start by recalling some facts about twisted homology in \cref{subsection:twisted_homology} and the Reidemeister torsion in \cref{subsection:reidemeister_torsion}. Then, we state and prove \cref{theorem:fried_conjecture_general}, a generalization of \cref{theorem:fried_conjecture}, in \cref{subsection:fried_conjecture}. Finally, we prove \cref{theorem:dirichletclassnumber} in \cref{subsection:branched_covers}.

\subsection{Twisted homology}\label{subsection:twisted_homology}
Suppose we have a finite cell complex $M$ and a representation $\rho: \pi_1(M) \to GL(m,\mathbb{C})$. Lift the cell decomposition of $M$ to its universal cover $M_{\textup{univ}}$ and let $C_*(M_{\textup{univ}})$ denote the associated chain complex. Define then the complex
\begin{equation*}
C_*(M; \rho)=C_*( M_{\textup{univ}}) \otimes_{\Z[\pi_1(M)]} \C^m,
\end{equation*}
where $\gamma\in \Z[\pi_1(M)]$ acts on the right on $c\in C_*(M_{\textup{univ}})$ by $c\cdot  \gamma = \gamma^{-1} \cdot c$ and $\Z[\pi_1(M)]$ acts on the left on $\C^m$ by $\rho$. The differential is defined to be zero on $\C^m$, so by the Leibniz rule, the boundary operator on $C_*(M; \rho)$ is the tensor product of the boundary operator on $C_*(M_{\textup{univ}})$ with the identity on $\mathbb{C}^m$. The homology $H_{*}(M;\rho)$ of the chain complex $C_*(M;\rho)$ is called the $\rho$-twisted homology of $M$. We may similarly define the twisted cohomology $H^*(M; \rho)$ as the homology of the dual chain complex $C^*(M;\rho)$. When $M$ is a closed $n$-manifold with orientation class $\varepsilon:\pi_1(M)\to \{\pm 1\}$, twisted Poincar\'e duality holds:
\begin{equation}\label{equation:twisted_duality}
	H_i(M;\rho) \cong H^{n-i}(M; \varepsilon \rho)
\end{equation}
When $\rho$ decomposes as $\rho_1 \oplus \rho_2$, the twisted cohomology groups split as well. In particular, if $\rho$ has a nonzero globally fixed vector $v \in \C^m$, then $H^i(M; \rho)$ contains a copy of the ordinary cohomology group $H^i(M; \C)$.

One may choose a basis for $C(M;\rho)$ called a \emph{cellular basis} as follows. For each cell $c$ in the cell decomposition for $M$, choose a lift $\widetilde c$ in $\widetilde M_{\textup{univ}}$. Then as a vector space, $$C(M;\rho) = \bigoplus_c \widetilde c \otimes \C^m.$$ The boundary map is given using this basis by $$\partial (c \otimes v) = \sum_{\substack{d \text{ cell in } M\\ g\in \pi_1(M)}}\langle \partial \widetilde c, g\widetilde {d}\rangle d\otimes \rho(g)v,$$ where the scalar product $\langle \cdot,\cdot \rangle$ makes the cells of $M_{\textup{univ}}$ an orthonormal basis of $C(M_{\textup{univ}})$. 

Now suppose that $\rho$ is unimodular, ie $|\det(\rho(g))|=1$ for all $g\in \pi_1(M)$. Then the choice of basis for $C(M;\rho)$ combined with the standard volume of $\C^m$ (which is preserved by $\rho$) endows each chain group with a volume form which we call the \emph{cellular volume form}. 

\begin{lemma}\label{lemma:cellular_volume_form}
	The cellular volume form on $C(M; \rho)$ is unique up to multiplication by unit norm scalars.
\end{lemma}
\begin{proof}
	Different choices of orderings of cells of $M$ may only change the sign of the cellular volume by $\pm 1$. Choices of lift $\widetilde c$ differing by $g\in \pi_1(M)$ may change the volume form by a $\det(\rho(g))$, which is a unit norm scalar by assumption.
\end{proof}

\begin{remark}
	If $\det \rho$ takes values in a subgroup $G$ of the unit circle in $\C$, then the cellular volume form is well-defined up to multiplication by elements of $G$ and $\pm 1$. So the Reidemeister torsion defined in the next section can be refined to an invariant in $\C/\langle G \cup \set{\pm 1} \rangle$.
\end{remark}

\subsection{Reidemeister torsion}\label{subsection:reidemeister_torsion}
Suppose that $C$ is a finite-dimensional chain complex over $\R$ or $\C$ equipped with a volume form $\alpha_i$ on each chain group $C_i$ and a volume form $\omega_i$ on each homology group $H_i(C)$. Make a choice of a volume form on $\ker(d_i)$ for each $i$. Combined with $\alpha_i$, this gives a volume form on $\image(d_i)$. In turn, this gives a volume form $\omega'_i$ on $H_i(C) = \ker(d_i)/\image(d_{i+1})$. Although $\omega_i'$ depends on the choice of volume form on $\ker(d_i)$, the product
\begin{equation}\label{equation:rt_formula}
	\tau(C,\alpha,\omega) =  \prod_i \left( \frac{\omega_i}{\omega_i'} \right)^{(-1)^i} \in \R
\end{equation}
does not depend on this choice. We call $\tau$ the Reidemeister torsion of $C$. 

\begin{remark}\label{remark:rtorsion_determinant_def}
	The Reidemeister torsion is often defined instead as
	\begin{equation*}\tau(C,\alpha)=\bigotimes_i (\omega_i')^{(-1)^i} \in \bigotimes_i (\det H_i(C))^{(-1)^i} 
	\end{equation*}
	Our definition is equivalent to pairing this volume form with a reference volume form $\otimes_i \omega_i^{(-1)^{i+1}}$ to get a real number.
\end{remark}

Now as in \cref{subsection:twisted_homology}, suppose we have a topological space $M$ with a finite cell decomposition and a unimodular representation $\rho:\pi_1(M) \to GL(m,\C)$. Let $\alpha_i$ be a cellular volume form on $C(M; \rho)$. For any choice of volume forms $\omega_i$ on $H_i(M;\rho)$, we may consider the Reidemeister torsion of the complex $C(M;\rho)$ $$\tau_\rho(M,\omega) := |\tau(C(M;\rho),\alpha,\omega)|\in \R_{>0},$$ which is called the $\rho$-twisted Reidemeister torsion of $M$. Since the cellular volume form is unique up to multiplication by unit norm scalars (\cref{lemma:cellular_volume_form}), $\tau_\rho(M,\omega)$ is well-defined. We often suppress $\omega$ (eg when $C$ is acyclic and we are using the canonical choice $\omega_i=1$) and abbreviate $$\tau_\rho(M) = \tau_\rho(M,\omega).$$

Reidemeister torsion satisfies the following properties:

\begin{enumerate}
	\item $\tau_{\rho}(M)$ does not depend on the choice of cell decomposition.
	\item $\tau_{\rho}(M)$ is invariant under simple homotopy equivalences (ie sequences of elementary collapses or their inverses)
	\item (Short exact sequence) Suppose we have a short exact sequence $$0\to A \to B \to C \to 0$$ of chain complexes where $A_*$, $C_*$, $H_*(A)$, and $H_*(C)$ are equipped with volume forms. The volume forms on $A$ and $C$ induce a volume form on $B$. Call the volume forms on $H_*(A)$ and $H_*(C)$ $\omega^A_*$ and $\omega^C_*$ respectively. Let $\delta_i: H_i(C)\to H_{i-1}(A)$ be the connecting homomorphism in the long exact sequence on homology. Choose a volume form on $\ker(\delta)$. Then the long exact sequence induces a volume form $\omega^B_i$ on $H_i(B)$. Then $$\tau(B)=\tau(A)\tau(C).$$ In particular, $\tau(B)$ doesn't depend on the choice of volume form on $\ker(\delta)$.
	\item (Spectral sequence) Suppose we have a filtered complex each of whose chain groups is equipped with a volume form. Consider the associated spectral sequence $(E^r, \partial^r)$. After choosing a volume form on $\ker(\partial^r)$ for each $r$, we get an induced volume form on the chain groups in each page. Then, regardless of the choices made, the Reidemeister torsion of $C$ is equal to the Reidemeister torsion of $(E^\infty, \partial^\infty)$. This means that we can compute the Reidemeister torsion by keeping track of volume forms through the pages of the spectral sequence. See \cite[Section 1]{freed} for more discussion.
\end{enumerate}
See \cite{nicolaescu} for more background on Reidemeister torsion.

\begin{example}
	Let $M=S^1$ and let $A$ be the monodromy of $\rho$. If $A$ has no eigenvalue equal to 1, then $H_*(M;\rho)=0$. Choose $\omega_i=1$ in each dimension. Then $\tau_{\rho}(M,\omega) = \det(I-A)$.

	If instead $A$ has a 1-eigenspace of dimension $k$, then $H_0(M;\rho)\cong\R^k$ and $H_1(M;\rho)\cong \R^k$. We can write $A=I_k \oplus A'$, where $I_k$ is the identity on $\R^k$ and $A'$ has no 1 eigenvalue. For any choice of volume forms $\omega_0$, $\omega_1$ on $\R^k$, we have $\tau_{\rho}(M,\omega)=\frac{\omega_0}{\omega_1} \det(I-A')$.

\end{example}
\begin{example}\label{example:manifold_rtorsion}
	Let $M$ be a rational homology sphere and fix $\omega$ so that $\omega_3$ evaluates to 1 on the fundamental class, $\omega_0$ evaluates to 1 on the point class, and $\omega_1=\omega_2=1$. Let $\rho$ be the trivial representation. Then $\tau_{\rho}(M,\omega)=|H_1(M,\Z)|^{-1}$. To see this, choose a cell decomposition $C_*$ of $M$ with a single vertex and a single 3-cell. (Such a cell decomposition always exists; start with a triangulation, delete 2-dimensional faces until there is only one 3-cell, and then contract a spanning tree for the 1-skeleton to a single vertex.) Then we have $\ker(d_0) = C_0$, $\ker(d_1)=C_1$, $\ker(d_2)=0$, and $\ker(d_3)=C_3$. Choose for $\ker(d_2)$ the trivial volume form and choose for $\ker(d_0)$, $\ker(d_1)$, and $\ker(d_2)$ the volume form agreeing with the one from the cellular basis on $C_0, C_1, C_3$. Then the induced volume forms on $H_*(M,\mathbb R)$ are  $\omega_0'=\omega_0$, $\omega_2'=1$, and $\omega_3'=\omega_3$. Finally, $\omega_1'$ is the covolume of the lattice $\image(d_2)$ in $\C_1$, which is exactly $|H_1(M,\Z)|$. So \cref{equation:rt_formula} evaluates to $|H_1(M,\mathbb Z)|^{-1}$.
\end{example}

\subsection{Fried's conjecture}\label{subsection:fried_conjecture}
The goal of this section is to prove the following more general version of \cref{theorem:fried_conjecture}.

\begin{theorem}\label{theorem:fried_conjecture_general}
Let $\varphi = (\varphi_t)_{t \in \mathbb{R}}$ be a transitive smooth pseudo-Anosov flow on a $3$-dimensional closed manifold. Let $\rho$ be an acyclic unitary representation of the fundamental group of $M$. Assume that:
\begin{enumerate}[label=(\roman*)]
\item for each singular orbit $\gamma$ of $\varphi$, either $1$ or $\varepsilon_\gamma$ is not an eigenvalue for $\rho(\gamma)$;\label{item:monodromy_singular}
\item there is $g \in \pi_1(M)$ such that $1$ is not an eigenvalue for $\rho(g)$;\label{item:gammasnotempty}
\item there is $g \in \pi_1(M)$ such that $\varepsilon_g$ is not an eigenvalue for $\rho(g)$.\label{item:gammaunotempty}
\end{enumerate}
Then
\begin{equation*}
|\zeta_{\varphi,\rho}(0)|^{-1} = \tau_\rho(M).
\end{equation*}
\end{theorem}

Consequently, let $\varphi$ be a transitive smooth pseudo-Anosov flow on a closed $3$-manifold $M$ and $\rho$ be a unitary acyclic representation on $M$ satisfying the assumptions from \cref{theorem:fried_conjecture_general}.

We will be working with a Markov partition as in \cref{section:Markov_partition}. We will impose some additional conditions on the sets of primitive periodic orbits $\Gamma^u$ and $\Gamma^s$ used to construct this Markov partition. We want $\Gamma^u$ and $\Gamma^s$ to be disjoint and that:
\begin{itemize}
\item for every $\gamma \in \Gamma^s$ the number $1$ is not an eigenvalue of $\rho(\gamma)$;
\item for every $\gamma \in \Gamma^u$ the number $\varepsilon_\gamma$ is not an eigenvalue of $\rho(\gamma)$.
\end{itemize}
Recall that for the construction of a Markov partition we also need that $\Gamma^u$ and $\Gamma^s$ are not empty and that $\Gamma^u \cup \Gamma^s$ contains all the singular orbits of $\varphi$. The assumption \ref{item:monodromy_singular} in \cref{theorem:fried_conjecture_general} ensures that each singular orbit for $\varphi$ may be put in $\Gamma^u$ or $\Gamma^s$. Moreover, assumption \ref{item:gammasnotempty} and \cref{lemma:mcmullen_realization} imply that $\varphi$ has infinitely many periodic orbits $\gamma$ such that $1$ is not an eigenvalue of $\rho(\gamma)$. Hence, if we did not put already a singular orbit in $\Gamma^s$, we can always find another periodic orbit to put in it. Similarly, using \ref{item:gammaunotempty} and \cref{lemma:mcmullen_realization}, we ensure that $\Gamma^u$ is not empty.

Let us now fix a Markov partition for $\varphi$ constructed using the specific $\Gamma^u$ and $\Gamma^s$ we just chose (see \cref{proposition:existence_Markov}). We will be using notations from \cref{section:zeta_continuation}. We are now able to express Reidemeister torsion in terms of our Markov partition. We also prove another formula that will be needed for the proof of \cref{theorem:dirichletclassnumber}.

	\begin{proposition}\label{prop:rtorsionformula}
		We have $$\tau_{\rho}(M)=\left|\det(I-H)^{-1} \prod_{\gamma\in \Gamma^s} \det(I-\rho(\gamma))\prod_{\gamma\in \Gamma^u}\det(I-\varepsilon_{\gamma}\rho(\gamma))\right|.$$
		where $H$ is the matrix from \cref{lemma:zeta_Markov_at_zero}. Similarly, if $\rho$ is an acyclic representation of $\pi_1(M\setminus \bigcup \Gamma^u)$, then
		$$\tau_\rho(M\setminus \bigcup \Gamma^u) = \left| \det(I-H)^{-1} \prod_{\gamma\in \Gamma^s} \det(I-\rho(\gamma)) \right|.$$
	\end{proposition}

	\begin{proof}
	For $\gamma \in \Gamma^s$, let $X_\gamma^s$ be the union of those vertical faces of flow boxes which are in the weak stable leaves of $\gamma$. Similarly define $X_\gamma^u$ as the union of faces in the weak unstable leaves of some $\gamma\in \Gamma^u$. By the Markov property, $X_\gamma^s$ is closed under forward flow of $\varphi$. Since $X_\gamma^s$ is compact, flowing it forward by $\varphi$ for a sufficiently long time brings $X_\gamma^s$ into a neighbourhood of $\gamma$. Therefore, $X_\gamma^s$ deformation retracts onto $\gamma$. See \cref{fig:stableface}. Similarly, $X^u_\gamma$ deformation retracts onto $\gamma$ for $\gamma \in \Gamma^u$.

	\begin{figure}[!b]
		\centering
	\def\svgwidth{1.75in}
		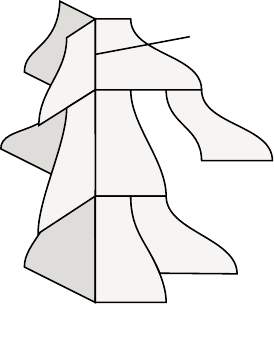
		\caption{A picture of $X_\gamma^s$, the union of the flow box faces on the stable leaves containing $\gamma$}\label{fig:stableface}
	\end{figure}

	Now cut open $M$ along $\bigcup_{\gamma \in \Gamma^u} X_\gamma^u$ to obtain a 3-manifold (with boundary) $M_{cut}$ composed of flow boxes glued along their horizontal faces to the rectangles $\{\widetilde{R_\theta}\}_{\theta \in \Theta}$, and glued along their vertical unstable faces to one another. Topologically, $M_{cut}$ is obtained by deleting a tubular neighbourhood of the elements of $\Gamma^u$, but we want to remember the cell decomposition by flow boxes. The leftmost panel of \cref{fig:complex} shows how $M$ would look after cutting along both $\bigcup_{\gamma \in \Gamma^u} X_\gamma^u$ and $\bigcup_{\gamma\in \Gamma^s}$; to obtain $M_{cut}$ from this picture one would need to reglue flow boxes along their vertical stable faces.  Collapse the $M_{cut}$ in the stable direction; see the middle panel of \cref{fig:complex} which shows how each flow box collapses to a rectangular 2-cell with two sides in the flow direction and two in the unstable direction. This gives a deformation retraction to a branched surface $B\subset M$ which carries the weak unstable foliation, see the right panel of \cref{fig:complex}. Following Sanchez-Morgado, this 2-complex has a cell decomposition of the form

	\begin{equation}
		0 \to \Z^\Theta \xrightarrow{\partial_1} \Z^\Theta \oplus \Z^{2\Theta} \xrightarrow{\partial_0} \Z^{2\Theta} \to 0\\
	\end{equation}

	To see this, we write the complex in terms of the generators shown in \cref{fig:complex}. Each rectangle $\widetilde{R}_\theta$ of our Markov partition collapses to an edge $e_\theta^{hor}$ with endpoints named $v_\theta^0$ and $v_\theta^1$. It is unimportant which endpoint is named $v_\theta^0$ and which is $v_\theta^1$. The union of the flowboxes leaving $\widetilde{R}_\theta$ collapses to a 2-cell which we call $f_\theta$. Finally, $e_\theta^{vert,0}$ and $e_\theta^{vert,1}$ are edges that connect from $v_\theta^0$ and $v_\theta^1$ respectively to the next vertex on the collapse of the weak stable face of a flow box leaving $\widetilde{R}_\theta$.

	\begin{equation} \label{eqn:2stepcomplex}
\begin{tikzcd}
	\bigoplus_{\theta \in \Theta} f_\theta \arrow[r, "{\partial_{1,0}}"] \arrow[rd, "{\partial_{1,1}}"] & \bigoplus_{\theta \in \Theta} e_{\theta}^{hor} \arrow[rd, "{\partial_{0,1}}"] \arrow[d, "\oplus", phantom] &                                                          \\
																										& {\bigoplus_{\theta \in \Theta, i\in \{0,1\}} e^{vert,i}_\theta} \arrow[r, "{\partial_{0,0}}"]              & {\bigoplus_{\theta \in \Theta, i\in \{0,1\}} v_\theta^i}
	\end{tikzcd}
\end{equation}

	\begin{figure}[!b]
		\centering
	\def\svgwidth{0.9\linewidth}
		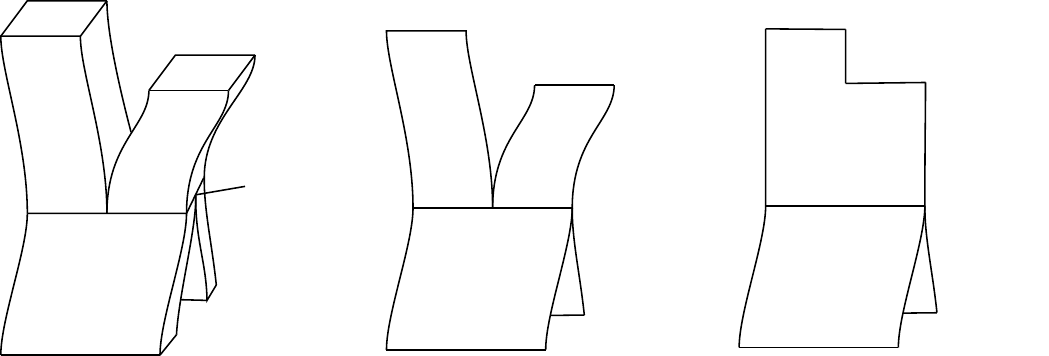
		\caption{Collapsing along the stable direction to get a 2-complex. In the picture on the left, we have drawn a rectangle $\widetilde{R}_\theta$ and separated the flow boxes entering or leaving $\widetilde{R}_\theta$.}\label{fig:complex}
	\end{figure}

	The chain complex is drawn in \cref{eqn:2stepcomplex} to emphasize a natural 2-step filtration. In this complex and in \cref{eqn:3stepcomplex} below, we denote by $\partial_{i,j}$ the part of $\partial_i$ which drops filtration level by $j$.
	
	To get a cell decomposition for $M$, we need to fill in a solid torus (or a twisted disk bundle over $S^1$ in the non-orientable case) for each $\gamma \in \Gamma^u$. This can be accomplished with a 2-cell $p_\gamma$ and a 3-cell $v_\gamma$ for each $\gamma \in \Gamma^u$. The resulting complex $C$ has a 3-step filtration:
\begin{equation}\label{eqn:3stepcomplex}
\begin{tikzcd}
	\bigoplus_{\gamma \in \Gamma^u}v_\gamma \arrow[r, "{\partial_{2,0}}"] \arrow[rd] & \bigoplus_{\gamma\in \Gamma^u} p_\gamma \arrow[rd] \arrow[rdd] \arrow[d, "\oplus", phantom] &                                                                                               &                                                          \\
																					 & \bigoplus_{\theta \in \Theta} f_\theta \arrow[r, "{\partial_{1,0}}"] \arrow[rd]             & \bigoplus_{\theta \in \Theta} e_{\theta}^{hor} \arrow[rd] \arrow[d, "\oplus", phantom]        &                                                          \\
																					 &                                                                                             & {\bigoplus_{\theta \in \Theta, i\in \{0,1\}} e^{vert,i}_\theta} \arrow[r, "{\partial_{0,0}}"] & {\bigoplus_{\theta \in \Theta, i\in \{0,1\}} v_\theta^i}
	\end{tikzcd}
	\end{equation}

	Now, using the cell decomposition associated to the chain complex \cref{eqn:3stepcomplex}, we construct the twisted chain complex $C(M;\rho)$ of $M$ as in the beginning of the current section. The resulting complex is of the form
	\begin{equation}\label{eqn:twisted_complex}
		0 \to \C^{\Gamma^u} \otimes \W \xrightarrow{\partial_2} (\C^{\Gamma^u} \oplus \C^\Theta) \otimes \W \xrightarrow{\partial_1} (\C^\Theta \oplus \C^{2\Theta}) \otimes \W \xrightarrow{\partial_0} \C^{2\Theta} \otimes \W \to 0
	\end{equation}

The twisted chain complex \cref{eqn:twisted_complex} admits a 3-step filtration similar to \cref{eqn:3stepcomplex}. Let us call $F_1,F_2,F_3$ the levels of this filtration. At the top filtration level, using a suitable choice of $\widetilde{v}_\gamma$ and $\widetilde{p}_\gamma$ in the description above, we find that the complex $F_3/F_2$ is isomorphic to $$\bigoplus_{\gamma \in \Gamma^u}\left(\W \xrightarrow{I-\varepsilon_{\gamma}\rho(\gamma)} \W\right).$$ The term $\varepsilon_{\gamma}$ appears because the cell $v_\gamma$ attaches to $p_\gamma$ by the map $I - \varepsilon_{\gamma}$. Recall that we imposed that for every $\gamma \in \Gamma^u$ the number $\varepsilon_\gamma$ is not an eigenvalue of $\rho(\gamma)$. Hence, the complex $F_3/F_2$ is acyclic, with torsion $\prod_{\gamma \in \Gamma^u} | \det(I - \varepsilon_\gamma \rho(\gamma)) |$.

	Continuing to the middle level, it will be convenient to make a special choice of lifts which is compatible with the choices made in \cref{subsection:zeta_functions}. Recall that we chose a basepoint $x_0$ and for each $\theta$ a path $c_\theta$ from $x_0$ to $\widetilde{R}_\theta$. This gives us a canonical homotopy class of path from $x_0$ to $e_{\theta}^{hor}$, $v_\theta^{0}$, $v_\theta^{1}$, and $f_\theta$ as well. Choose once and for all a lift of $x_0$ to $M_{\textup{univ}}$, and then choose the lifts $\widetilde{e}_{\theta}^{hor}$, $\widetilde{v}_\theta^{0}$, $\widetilde{v}_\theta^{1}$, and $\widetilde{f}_\theta$ by concatenating with the lift of $c_\theta$. With this choice, $$\partial_{1,0} f_\theta = e_\theta^{hor} - \sum_{\vartheta\in \Theta: A_{\theta, \vartheta}=1} \Delta_{\theta,\vartheta} \rho_{\theta, \vartheta}.$$ Therefore, the complex $F_2/F_1$ can be written as $$(\W)^\Theta \xrightarrow{I-H} (\W)^\Theta$$ where $H$ is the $\Delta \rho$ twisted, directed adjacency matrix of the Markov partition (as in \cref{lemma:zeta_Markov_at_zero}). Hence, $F_2/F_1$ is acyclic if and only if $1$ is not an eigenvalue of $H$, in which case its torsion is $|\det(I - H)|^{-1}$.
	
	 Finally, the bottom level computes the twisted homology of the 1-complex composed of all the edges $e^{vert,0}_\theta$ and $e^{vert,1}_\theta$. This space is a deformation retract of $\bigcup_{\gamma \in \Gamma^s} X_\gamma^s$. Therefore, the underlying chain complex $F_1$ is simple homotopy equivalent to $$\bigoplus_{\gamma \in \Gamma^s} \left(\W \xrightarrow{I-\rho(\gamma)}\W\right).$$ Since for every $\gamma \in \Gamma^s$, the number $1$ is not an eigenvalue of $\rho(\gamma)$, we find that $F_1$ is acyclic with torsion $\prod_{\gamma \in \Gamma^s} |\det(I - \rho(\gamma))|$.
	 
	 By assumption, the complex $F_3$ is acyclic. Hence, using the long exact sequence associated to the short exact sequence $0 \to F_2 \to F_3 \to F_3 /F_2 \to 0$, we find that $F_2$ is acyclic and that $\tau(F_3) = \tau(F_2) \tau(F_3/F_2)$. Using then the long exact sequence associated to the short exact sequence $0 \to F_1 \to F_2 \to F_2 / F_1 \to 0$, we get that $F_2/F_1$ is acyclic, and that $\tau(F_3) = \tau(F_3/F_2) \tau(F_2/F_1) \tau(F_1)$. But $\tau(F_3)$ is just $\tau_\rho(M)$, and thus the first statement follows by the computation of the torsions of $F_3/F_2, F_2/F_1$ and $F_1$ above.

	 To prove the second statement, we simply repeat the argument, but with the top filtration level $F_3$ removed. The result is the $\rho$-twisted chain complex for $B$, which is homotopy equivalent to $M\setminus \bigcup \Gamma^u$.

	\end{proof}
	
\begin{proof}[Proof of \cref{theorem:fried_conjecture}]
Recall the relation \eqref{eq:Markov_to_standard} between the zeta function $\zeta_\rho^{\textup{Markov}}$ associated to the Markov partition and the zeta function $\zeta_{\varphi,\rho}$. Our assumptions on $\rho$ implies that the hypotheses of \cref{lemma:zeta_Markov_at_zero} are satisfied (thanks to Lemmas \ref{lemma:denominator_non_zero} and \ref{lemma:denominator_non_zero_two}), so that we have $\zeta_\rho^{\textup{Markov}}(0) = \det(I - H)$. We also find that the other factor in \cref{eq:Markov_to_standard} does not have a pole at zero, and consequently,
\begin{equation*}
\left|\zeta_{\varphi,\rho}(0)\right|^{-1} = \left|\det(I-H)^{-1} \prod_{\gamma\in \Gamma^s} \det(I-\rho(\gamma))\prod_{\gamma\in \Gamma^u}\det(I-\varepsilon_{\gamma}\rho(\gamma))\right| = \tau_{\rho}(M),
\end{equation*}
where the last equality comes from \cref{prop:rtorsionformula}.
\end{proof}

\subsection{Double branched covers}\label{subsection:branched_covers}
Finally, let us prove \cref{theorem:dirichletclassnumber}. Let $M$ be an integer homology sphere and let $\varphi$ be a transitive $C^\infty$ Anosov flow on $M$. Let $\gamma_1,\dots,\gamma_n$ be a nonempty collection of primitive closed orbits of $\varphi$. Given a curve $\gamma \subset M\setminus \bigcup_i \gamma_i$, define $\rho(\gamma)=(-1)^{\lk(\gamma, \bigcup_i \gamma_i)}$, where $\lk$ denotes the linking number. By \cref{lemma:mcmullen_realization}, there exists a closed orbit $\gamma$ with $\rho(\gamma)=-1$. Choose a Markov partition with $\Gamma^u = \set{\gamma_1,\dots,\gamma_n}$ and $\Gamma^s$ consisting of some closed orbits each satisfying $\rho(\gamma)=-1$. Let $\overline M$ be the branched double cover $M$ ramified over $\gamma_1,\dots,\gamma_n$, which may be constructed from the cover of $M\setminus \bigcup_i \gamma_i$ corresponding with the kernel of $\rho$ by adding a curve over each of the $\gamma_i$'s.
\begin{proposition}\label{proposition:doublecovertorsion}
	Suppose $\overline M$ is a rational homology sphere. Then $\rho$ is an acyclic representation of $\pi_1(M\setminus \bigcup \Gamma^u)$ and $$|H_1(\overline M,\mathbb Z)| = \left|\det(I-H)\right|2^{-|\Gamma^s|-1}$$ where $H$ is the $\Delta \rho$-twisted directed adjacency matrix of our Markov partition.
\end{proposition}

\begin{proof}

Let $C(M)$ be the cell decomposition constructed in \cref{prop:rtorsionformula} and let $C(\overline M)$ be its lift to $\overline M$. Note that for $\gamma \in \Gamma^u$, the cells $v_\gamma$ and $p_\gamma$ each naturally lifts to a single cell $\overline{v_\gamma}$ and $\overline{p_\gamma}$ in $\overline M$. As in \cref{prop:rtorsionformula}, we use $B$ to denote the unstable branched surface in $M$.

Let $X$ be the $\rho$-twisted chain complex of $B$. Note that $X$ fits into the short exact sequence
$$X \xrightarrow{\iota} C(\overline M) \xrightarrow{\pi} C(M).$$

After flipping the signs of $\iota$ and $\pi$ in odd degree, we get the following chain complex:

\begin{equation}
\begin{tikzcd}
	X_3 \arrow[r] \arrow[d, "-\iota"]            & X_2 \arrow[r] \arrow[d, "\iota"]            & X_1 \arrow[r] \arrow[d, "-\iota"]            & X_0 \arrow[d, "\iota"]            \\
	C_3(\overline M) \arrow[r] \arrow[d, "-\pi"] & C_2(\overline M) \arrow[r] \arrow[d, "\pi"] & C_1(\overline M) \arrow[r] \arrow[d, "-\pi"] & C_0(\overline M) \arrow[d, "\pi"] \\
	C_3(M) \arrow[r]                             & C_2(M) \arrow[r]                            & C_1(M) \arrow[r]                             & C_0(M)                           
	\end{tikzcd}
\end{equation}

We can compute the Reidemeister torsion of this chain complex using the spectral sequence associated to either the vertical or horizontal filtration. Let's consider the first page of the spectral sequence associated to the horizontal filtration. The columns break into a direct sum over cells of $C(M)$. Each cell contributes a column of one of the following three types:

\begin{equation}
\begin{tikzcd}
	0 \arrow[d]                        &  & 0 \arrow[d] \arrow[d]              &  & \mathbb Z \arrow[d, "{(1,-1)}"] \\
	\overline{v_\gamma} \arrow[d, "2"] &  & \overline{p_\gamma} \arrow[d, "2"] &  & \mathbb Z^2 \arrow[d, "x+y"]    \\
	v_\gamma                           &  & p_\gamma                           &  & \mathbb Z                      
	\end{tikzcd}
\end{equation}

All the columns are acyclic. Therefore, the Reidemeister torsion of the double complex is a product of contributions from each column. The first two types of column have Reidemeister torsion 2, and appear in gradings 3 and 2 respectively. Therefore, their net contribution is 1. The contribution of the last type of column is 1. Therefore, the Reidemeister torsion of the double complex is 1.

On the other hand, the second page of the spectral sequence associated with the vertical filtration looks like
\begin{equation}
\begin{tikzcd}
	H_3(X) \arrow[d, "0"]              & H_2(X) \arrow[d] & H_1(X) \arrow[d]                              & H_0(X) \arrow[d, "0"]             \\
	H_3(\overline M) \arrow[d, "-\pi"] & 0_1 \arrow[d]    & {0_{|H_1(\overline M, \mathbb Z)|}} \arrow[d] & H_0(\overline M) \arrow[d, "\pi"] \\
	H_3(M)                             & 0_1              & {0_{|H_1(M,\mathbb Z)|}}                      & H_0(M)                           
	\end{tikzcd}
\end{equation}

Here, $0_{\omega}$ represents the 0 vector space with the volume form of magnitude $\omega$. For any choice of cell structure on a 3-manifold, the fundamental class has volume 1 and the point class has volume 1. Therefore, the map $\pi:H_0(\overline M) \to H_0(M)$ preserves the volume form and the map $\iota:H_3(\overline M) \to H_3(M)$ doubles the volume form. Therefore, we can compute the third page:

\begin{equation}
\begin{tikzcd}
	H_3(X) & H_2(X) \arrow[ldd] & H_1(X) \arrow[ldd]                  & H_0(X) \arrow[ldd] \\
	0_1    & 0_1                & {0_{|H_1(\overline M, \mathbb Z)|}} & 0_1                \\
	0_2    & 0_1                & {0_{|H_1(M,\mathbb Z)|}}            & 0_1               
	\end{tikzcd}
\end{equation}

The spectral sequence collapses at this stage, and $X$ must be acyclic. This proves the first statement in the lemma, that $\rho$ is an acyclic representation of $\pi_1(M\setminus \bigcup \Gamma^u)\cong \pi_1(B)$. Recall that we assumed that $|H_1(M,\Z)|=1$. Now the Reidemeister torsion of the complex above is the alternating product of the magnitudes of the volume forms in each grading:
\begin{equation}
	2\tau(X) |H_1(\overline M,\Z)|
\end{equation}
Equating our two computations of the Reidemeister torsion of the double complex, we get
\begin{equation}\label{eqn:vertvshor}
	2\tau(X) |H_1(\overline M,\Z)|=1
\end{equation}
Using \cref{prop:rtorsionformula} and the fact that $\rho(\gamma)=-1$ for each $\gamma\in \Gamma^s$, we find, 
\begin{align*}
	\tau(X)&=\left|\det(I-H)^{-1}\prod_{\gamma \in \Gamma^s} \det(I-\rho(\gamma))\right|\\
	&=\left|\det(I-H)^{-1}\right|2^{|\Gamma^s|}
\end{align*}
	
	Substituting into \cref{eqn:vertvshor} gives the desired formula.
	
\end{proof}

\begin{proof}[Proof of \Cref{theorem:dirichletclassnumber}]
	Let us compare the contributions of each orbit $\gamma$ of $\varphi$ to $\zeta_{\overline \varphi,1}(s)$ and $\zeta_{\varphi,1}(s)$. When $\gamma \not\in \set{\gamma_1,\dots,\gamma_n}$, $\gamma$ lifts to either one or two orbits depending on $\lk(\gamma, \cup_i \gamma_i)$. The contribution of $\gamma$ toward $\zeta_{\overline \varphi,1}(s)$ is
	
	\begin{equation*}
			\begin{cases} 
				(1-\Delta_\gamma e^{-sT_\gamma})^2 & \textup{ if $\rho(\gamma)=1$}\\ 
				1-\Delta_\gamma^2e^{-2sT_\gamma} & \textup{ if $\rho(\gamma)=-1$}\\ 			
			\end{cases}
	\end{equation*}
	Therefore, the contribution of $\gamma$ toward $\zeta_{\overline \varphi,1}(s)/\zeta_{\varphi,1}(s)$ is $$(1-\Delta_\gamma e^{-sT_\gamma}\rho(\gamma)).$$

	Now suppose $\gamma\in \set{\gamma_1,\dots,\gamma_n}$. Recall the notation $\lambda^u_1,\dots,\lambda^u_m$ and $r_1^u,\dots,r_m^u$ from \cref{eq:definition_local_zeta_function}. In fact, because our manifold is orientable, $r_1^u=\dots=r_m^u$ and the pushoffs of $\gamma^{r_j^u}$ into $\lambda_j^u$ are all homotopic in the complement of $\gamma$. There are two cases, depending on the linking number of $\bigcup_i \gamma_i$ with the pushoff of $\gamma^{r_j^u}$. This dictates whether $\lambda_j^u$ will have one or two lifts. The contribution of $\gamma$ to $\zeta_{\overline \varphi,1}(s)$ is

	\begin{equation*}
		\begin{cases} 
			\frac{\prod_{j=1}^m (1- e^{-sr_j^u T_\gamma})^2}{1-e^{-sT_\gamma}} & \textup{ if $\lambda^u_j$ has two lifts}\\ 
			\frac{\prod_{j=1}^m (1- e^{-2sr_j^u T_\gamma})}{1-e^{-sT_\gamma}} & \textup{ if $\lambda^u_j$ has one lift}\\ 			
		\end{cases}
\end{equation*}
Now the contribution of $\gamma$ to $\zeta_{\overline \varphi,1}(s)/\zeta_{\varphi,1}(s)$ is $$\prod_{j=1}^m (1-e^{-sr_j^u T_\gamma}\rho(\gamma^{r_j^u}))=\prod_{j=1}^m (1-\Delta_{\gamma^{r_j^u}}e^{-sr_j^u T_\gamma}\rho(\gamma^{r_j^u}))$$ where $\rho(\gamma^{r_j^u})$ is to be interpreted as $\rho$ evaluated on the pushoff of $\gamma^{r_j^u}$ into $\lambda^u_j$, and we have used \cref{remark:orientable half leaves} to see that $\Delta=1$ in this situation.

Now let us summarize the overcounting of orbits in the symbolic shift. An orbit $\gamma \not\in \set{\gamma_1,\dots,\gamma_n}$ is overcounted if and only if $\gamma \in \Gamma^s$. As in \cref{proposition:explicit_zeta_Markov3}, the overcounting is by a factor of $\xi_{\varphi,\rho,\gamma}(s)$.

For $\gamma \in \set{\gamma_1,\dots,\gamma_n}$, the orbits $\gamma^{r_j^u}$ are all counted exactly once in the Markov partition. Therefore, as in \cref{proposition:explicit_zeta_Markov3}, we have $$\frac{\zeta_{\overline \varphi,1}(s)} {\zeta_{\varphi,1}(s)}=\prod_{\mathcal{O} \in P_{\Sigma}} F_{\mathcal{O}}(s,1) \left(\prod_{\gamma \in \Gamma^s} \xi_{\varphi,\rho,\gamma}(s)\right)^{-1},$$ where we use notations from \cref{subsection:correcting}. By \cref{proposition:explicit_zeta_Markov2}, the first product on the right side is $\zeta_{\rho}^{\textup{Markov}}(s)$.

As proved in \cref{proposition:doublecovertorsion}, $\rho$ is an acyclic representation of $\pi_1(M\setminus \bigcup \Gamma^u)$. Also, $\varepsilon$ is the trivial representation. Therefore, \cref{lemma:denominator_non_zero} and \cref{lemma:denominator_non_zero_two} apply to show that $d_1(0,1)\neq 0$ and $d_2(0,1)\neq 0$. The hypotheses of \cref{lemma:zeta_Markov_at_zero} are satisfied. Now we can evaluate both sides at 0:
\begin{align*}
\left|\frac{\zeta_{\overline \varphi,1}(0)} {\zeta_{\varphi,1}(0)}\right|&=\left|\zeta^{\textup{Markov}}_{\rho}(0) \left(\prod_{\gamma \in \Gamma^s} \xi_{\varphi,\rho,\gamma}(0)\right)^{-1}\right|\\
&= \left|\det(I-H) 2^{-|\Gamma^s|}\right| \qquad \qquad\textup{by \cref{lemma:zeta_Markov_at_zero},}\\
&=2|H_1(\overline M,\Z)| \qquad \qquad \textup{by \cref{proposition:doublecovertorsion}.}
\end{align*}

\end{proof}

\bibliographystyle{alpha}
\bibliography{zeta_pseudo.bib}

\end{document}

%% file: pushoff.pdf_tex
\begingroup%
  \makeatletter%
  \providecommand\color[2][]{%
    \errmessage{(Inkscape) Color is used for the text in Inkscape, but the package 'color.sty' is not loaded}%
    \renewcommand\color[2][]{}%
  }%
  \providecommand\transparent[1]{%
    \errmessage{(Inkscape) Transparency is used (non-zero) for the text in Inkscape, but the package 'transparent.sty' is not loaded}%
    \renewcommand\transparent[1]{}%
  }%
  \providecommand\rotatebox[2]{#2}%
  \newcommand*\fsize{\dimexpr\f@size pt\relax}%
  \newcommand*\lineheight[1]{\fontsize{\fsize}{#1\fsize}\selectfont}%
  \ifx\svgwidth\undefined%
    \setlength{\unitlength}{117.43463802bp}%
    \ifx\svgscale\undefined%
      \relax%
    \else%
      \setlength{\unitlength}{\unitlength * \real{\svgscale}}%
    \fi%
  \else%
    \setlength{\unitlength}{\svgwidth}%
  \fi%
  \global\let\svgwidth\undefined%
  \global\let\svgscale\undefined%
  \makeatother%
  \begin{picture}(1,1.52790264)%
    \lineheight{1}%
    \setlength\tabcolsep{0pt}%
    \put(0,0){\includegraphics[width=\unitlength,page=1]{pushoff.pdf}}%
    \put(0.88155784,1.31308465){\color[rgb]{0,0,0}\makebox(0,0)[lt]{\lineheight{1.25}\smash{\begin{tabular}[t]{l}$\gamma$\end{tabular}}}}%
    \put(0,0){\includegraphics[width=\unitlength,page=2]{pushoff.pdf}}%
    \put(1.08004464,0.39924525){\color[rgb]{0,0,0}\makebox(0,0)[lt]{\lineheight{1.25}\smash{\begin{tabular}[t]{l}$\lambda_1^u$\end{tabular}}}}%
    \put(0,0){\includegraphics[width=\unitlength,page=3]{pushoff.pdf}}%
    \put(0.96998448,0.99158673){\color[rgb]{0,0,0}\makebox(0,0)[lt]{\lineheight{1.25}\smash{\begin{tabular}[t]{l}$\textup{pushoff of }\gamma^3$\end{tabular}}}}%
    \put(0,0){\includegraphics[width=\unitlength,page=4]{pushoff.pdf}}%
  \end{picture}%
\endgroup%

%% file: sectors.pdf_tex
\begingroup%
  \makeatletter%
  \providecommand\color[2][]{%
    \errmessage{(Inkscape) Color is used for the text in Inkscape, but the package 'color.sty' is not loaded}%
    \renewcommand\color[2][]{}%
  }%
  \providecommand\transparent[1]{%
    \errmessage{(Inkscape) Transparency is used (non-zero) for the text in Inkscape, but the package 'transparent.sty' is not loaded}%
    \renewcommand\transparent[1]{}%
  }%
  \providecommand\rotatebox[2]{#2}%
  \newcommand*\fsize{\dimexpr\f@size pt\relax}%
  \newcommand*\lineheight[1]{\fontsize{\fsize}{#1\fsize}\selectfont}%
  \ifx\svgwidth\undefined%
    \setlength{\unitlength}{493.45254135bp}%
    \ifx\svgscale\undefined%
      \relax%
    \else%
      \setlength{\unitlength}{\unitlength * \real{\svgscale}}%
    \fi%
  \else%
    \setlength{\unitlength}{\svgwidth}%
  \fi%
  \global\let\svgwidth\undefined%
  \global\let\svgscale\undefined%
  \makeatother%
  \begin{picture}(1,0.45757539)%
    \lineheight{1}%
    \setlength\tabcolsep{0pt}%
    \put(0,0){\includegraphics[width=\unitlength,page=1]{sectors.pdf}}%
    \put(0.20278888,0.10349278){\makebox(0,0)[lt]{\lineheight{1.25}\smash{\begin{tabular}[t]{l}$\overline{\mathbb{H}}$\end{tabular}}}}%
    \put(0.33636902,0.26539123){\makebox(0,0)[lt]{\lineheight{1.25}\smash{\begin{tabular}[t]{l}$\mathcal V$\end{tabular}}}}%
    \put(0,0){\includegraphics[width=\unitlength,page=2]{sectors.pdf}}%
    \put(0.48258799,0.3316727){\makebox(0,0)[lt]{\lineheight{1.25}\smash{\begin{tabular}[t]{l}$\psi_k$\end{tabular}}}}%
    \put(0.60524877,0.26372843){\makebox(0,0)[lt]{\lineheight{1.25}\smash{\begin{tabular}[t]{l}$\mathfrak{A}^3_k$\end{tabular}}}}%
    \put(0.60934837,0.16617914){\makebox(0,0)[lt]{\lineheight{1.25}\smash{\begin{tabular}[t]{l}$\mathfrak{A}^3_{\sigma(k)}$\end{tabular}}}}%
    \put(0.88991024,0.27420006){\color[rgb]{0.82745098,0.37254902,0.37254902}\makebox(0,0)[lt]{\lineheight{1.25}\smash{\begin{tabular}[t]{l}$\mathcal U \cap \mathfrak{A}^3_k$\end{tabular}}}}%
    \put(0,0){\includegraphics[width=\unitlength,page=3]{sectors.pdf}}%
    \put(0.91723847,0.19756912){\color[rgb]{0.35294118,0.62745098,0.17254902}\makebox(0,0)[lt]{\lineheight{1.25}\smash{\begin{tabular}[t]{l}$\mathfrak{L}^3$\end{tabular}}}}%
    \put(0.49162678,0.12465894){\makebox(0,0)[t]{\lineheight{1.25}\smash{\begin{tabular}[t]{c}$\psi_{\sigma(k)}$\end{tabular}}}}%
    \put(0.69846939,0.18938352){\makebox(0,0)[t]{\lineheight{1.25}\smash{\begin{tabular}[t]{c}$\phi$\end{tabular}}}}%
  \end{picture}%
\endgroup%

%% file: bowen_bracket.pdf_tex
\begingroup%
  \makeatletter%
  \providecommand\color[2][]{%
    \errmessage{(Inkscape) Color is used for the text in Inkscape, but the package 'color.sty' is not loaded}%
    \renewcommand\color[2][]{}%
  }%
  \providecommand\transparent[1]{%
    \errmessage{(Inkscape) Transparency is used (non-zero) for the text in Inkscape, but the package 'transparent.sty' is not loaded}%
    \renewcommand\transparent[1]{}%
  }%
  \providecommand\rotatebox[2]{#2}%
  \newcommand*\fsize{\dimexpr\f@size pt\relax}%
  \newcommand*\lineheight[1]{\fontsize{\fsize}{#1\fsize}\selectfont}%
  \ifx\svgwidth\undefined%
    \setlength{\unitlength}{204.88260841bp}%
    \ifx\svgscale\undefined%
      \relax%
    \else%
      \setlength{\unitlength}{\unitlength * \real{\svgscale}}%
    \fi%
  \else%
    \setlength{\unitlength}{\svgwidth}%
  \fi%
  \global\let\svgwidth\undefined%
  \global\let\svgscale\undefined%
  \makeatother%
  \begin{picture}(1,0.99763064)%
    \lineheight{1}%
    \setlength\tabcolsep{0pt}%
    \put(0,0){\includegraphics[width=\unitlength,page=1]{bowen_bracket.pdf}}%
    \put(0.31111296,0.59002374){\makebox(0,0)[lt]{\lineheight{1.25}\smash{\begin{tabular}[t]{l}$x$\end{tabular}}}}%
    \put(0.46985217,0.73015889){\makebox(0,0)[lt]{\lineheight{1.25}\smash{\begin{tabular}[t]{l}$y$\end{tabular}}}}%
    \put(0.64819423,0.60530152){\makebox(0,0)[lt]{\lineheight{1.25}\smash{\begin{tabular}[t]{l}$x'$\end{tabular}}}}%
    \put(0.478888,0.3070453){\makebox(0,0)[lt]{\lineheight{1.25}\smash{\begin{tabular}[t]{l}$y'$\end{tabular}}}}%
    \put(0.3167125,0.74276399){\makebox(0,0)[rt]{\lineheight{1.25}\smash{\begin{tabular}[t]{r}$[x,y]$\end{tabular}}}}%
  \end{picture}%
\endgroup%

%% file: transverse_disk.pdf_tex
\begingroup%
  \makeatletter%
  \providecommand\color[2][]{%
    \errmessage{(Inkscape) Color is used for the text in Inkscape, but the package 'color.sty' is not loaded}%
    \renewcommand\color[2][]{}%
  }%
  \providecommand\transparent[1]{%
    \errmessage{(Inkscape) Transparency is used (non-zero) for the text in Inkscape, but the package 'transparent.sty' is not loaded}%
    \renewcommand\transparent[1]{}%
  }%
  \providecommand\rotatebox[2]{#2}%
  \newcommand*\fsize{\dimexpr\f@size pt\relax}%
  \newcommand*\lineheight[1]{\fontsize{\fsize}{#1\fsize}\selectfont}%
  \ifx\svgwidth\undefined%
    \setlength{\unitlength}{272.49504662bp}%
    \ifx\svgscale\undefined%
      \relax%
    \else%
      \setlength{\unitlength}{\unitlength * \real{\svgscale}}%
    \fi%
  \else%
    \setlength{\unitlength}{\svgwidth}%
  \fi%
  \global\let\svgwidth\undefined%
  \global\let\svgscale\undefined%
  \makeatother%
  \begin{picture}(1,0.42611317)%
    \lineheight{1}%
    \setlength\tabcolsep{0pt}%
    \put(0,0){\includegraphics[width=\unitlength,page=1]{transverse_disk.pdf}}%
    \put(0.48926529,0.21246917){\makebox(0,0)[rt]{\lineheight{1.25}\smash{\begin{tabular}[t]{r}$x_0$\end{tabular}}}}%
    \put(0.62613081,0.10591342){\makebox(0,0)[t]{\lineheight{1.25}\smash{\begin{tabular}[t]{c}$y$\end{tabular}}}}%
    \put(0.54276983,0.33044825){\makebox(0,0)[rt]{\lineheight{1.25}\smash{\begin{tabular}[t]{r}$x$\end{tabular}}}}%
    \put(0.68872838,0.33004726){\makebox(0,0)[lt]{\lineheight{1.25}\smash{\begin{tabular}[t]{l}$[x,y]$\end{tabular}}}}%
    \put(0.69131616,0.25402166){\makebox(0,0)[lt]{\lineheight{1.25}\smash{\begin{tabular}[t]{l}$\varphi^{t_U([x,y])}([x,y])$\end{tabular}}}}%
    \put(0.24721074,0.29633557){\makebox(0,0)[rt]{\lineheight{1.25}\smash{\begin{tabular}[t]{r}$J_0$\end{tabular}}}}%
    \put(0.36292581,0.10914338){\makebox(0,0)[rt]{\lineheight{1.25}\smash{\begin{tabular}[t]{r}$I_0$\end{tabular}}}}%
    \put(0.25705946,0.34885174){\makebox(0,0)[rt]{\lineheight{1.25}\smash{\begin{tabular}[t]{r}$W^s_\epsilon(x_0)$\end{tabular}}}}%
    \put(0.38321457,0.0420489){\makebox(0,0)[rt]{\lineheight{1.25}\smash{\begin{tabular}[t]{r}$W^u_\epsilon(x_0)$\end{tabular}}}}%
    \put(0.15688129,0.1541443){\makebox(0,0)[rt]{\lineheight{1.25}\smash{\begin{tabular}[t]{r}$\Psi(I_0\times J_0)$\end{tabular}}}}%
    \put(0,0){\includegraphics[width=\unitlength,page=2]{transverse_disk.pdf}}%
    \put(0.88383532,0.04751065){\makebox(0,0)[lt]{\lineheight{1.25}\smash{\begin{tabular}[t]{l}$D$\end{tabular}}}}%
    \put(0,0){\includegraphics[width=\unitlength,page=3]{transverse_disk.pdf}}%
  \end{picture}%
\endgroup%

%% file: basepoint.pdf_tex
\begingroup%
  \makeatletter%
  \providecommand\color[2][]{%
    \errmessage{(Inkscape) Color is used for the text in Inkscape, but the package 'color.sty' is not loaded}%
    \renewcommand\color[2][]{}%
  }%
  \providecommand\transparent[1]{%
    \errmessage{(Inkscape) Transparency is used (non-zero) for the text in Inkscape, but the package 'transparent.sty' is not loaded}%
    \renewcommand\transparent[1]{}%
  }%
  \providecommand\rotatebox[2]{#2}%
  \newcommand*\fsize{\dimexpr\f@size pt\relax}%
  \newcommand*\lineheight[1]{\fontsize{\fsize}{#1\fsize}\selectfont}%
  \ifx\svgwidth\undefined%
    \setlength{\unitlength}{325.48562838bp}%
    \ifx\svgscale\undefined%
      \relax%
    \else%
      \setlength{\unitlength}{\unitlength * \real{\svgscale}}%
    \fi%
  \else%
    \setlength{\unitlength}{\svgwidth}%
  \fi%
  \global\let\svgwidth\undefined%
  \global\let\svgscale\undefined%
  \makeatother%
  \begin{picture}(1,0.35023034)%
    \lineheight{1}%
    \setlength\tabcolsep{0pt}%
    \put(0,0){\includegraphics[width=\unitlength,page=1]{basepoint.pdf}}%
    \put(0.9467512,0.1008095){\color[rgb]{0,0,0}\makebox(0,0)[lt]{\lineheight{1.25}\smash{\begin{tabular}[t]{l}$x_0$\end{tabular}}}}%
    \put(0,0){\includegraphics[width=\unitlength,page=2]{basepoint.pdf}}%
    \put(0.14083949,0.28016324){\color[rgb]{0,0,0}\makebox(0,0)[rt]{\lineheight{1.25}\smash{\begin{tabular}[t]{r}$\widetilde{R}_\vartheta$\end{tabular}}}}%
    \put(0.08333174,0.06353483){\color[rgb]{0,0,0}\makebox(0,0)[rt]{\lineheight{1.25}\smash{\begin{tabular}[t]{r}$\widetilde{R}_\theta$\end{tabular}}}}%
    \put(0.7432925,0.054367){\color[rgb]{0,0,0}\makebox(0,0)[lt]{\lineheight{1.25}\smash{\begin{tabular}[t]{l}$c_\theta$\end{tabular}}}}%
    \put(0.6918722,0.27422379){\color[rgb]{0,0,0}\makebox(0,0)[lt]{\lineheight{1.25}\smash{\begin{tabular}[t]{l}$c_\vartheta$\end{tabular}}}}%
    \put(0,0){\includegraphics[width=\unitlength,page=3]{basepoint.pdf}}%
  \end{picture}%
\endgroup%

%% file: stableface.pdf_tex
\begingroup%
  \makeatletter%
  \providecommand\color[2][]{%
    \errmessage{(Inkscape) Color is used for the text in Inkscape, but the package 'color.sty' is not loaded}%
    \renewcommand\color[2][]{}%
  }%
  \providecommand\transparent[1]{%
    \errmessage{(Inkscape) Transparency is used (non-zero) for the text in Inkscape, but the package 'transparent.sty' is not loaded}%
    \renewcommand\transparent[1]{}%
  }%
  \providecommand\rotatebox[2]{#2}%
  \newcommand*\fsize{\dimexpr\f@size pt\relax}%
  \newcommand*\lineheight[1]{\fontsize{\fsize}{#1\fsize}\selectfont}%
  \ifx\svgwidth\undefined%
    \setlength{\unitlength}{131.14518929bp}%
    \ifx\svgscale\undefined%
      \relax%
    \else%
      \setlength{\unitlength}{\unitlength * \real{\svgscale}}%
    \fi%
  \else%
    \setlength{\unitlength}{\svgwidth}%
  \fi%
  \global\let\svgwidth\undefined%
  \global\let\svgscale\undefined%
  \makeatother%
  \begin{picture}(1,1.26449505)%
    \lineheight{1}%
    \setlength\tabcolsep{0pt}%
    \put(0,0){\includegraphics[width=\unitlength,page=1]{stableface.pdf}}%
    \put(0.70711563,1.11033516){\color[rgb]{0,0,0}\makebox(0,0)[lt]{\lineheight{1.25}\smash{\begin{tabular}[t]{l}$\gamma_i$\end{tabular}}}}%
    \put(0,0){\includegraphics[width=\unitlength,page=2]{stableface.pdf}}%
  \end{picture}%
\endgroup%

%% file: complex.pdf_tex
\begingroup%
  \makeatletter%
  \providecommand\color[2][]{%
    \errmessage{(Inkscape) Color is used for the text in Inkscape, but the package 'color.sty' is not loaded}%
    \renewcommand\color[2][]{}%
  }%
  \providecommand\transparent[1]{%
    \errmessage{(Inkscape) Transparency is used (non-zero) for the text in Inkscape, but the package 'transparent.sty' is not loaded}%
    \renewcommand\transparent[1]{}%
  }%
  \providecommand\rotatebox[2]{#2}%
  \newcommand*\fsize{\dimexpr\f@size pt\relax}%
  \newcommand*\lineheight[1]{\fontsize{\fsize}{#1\fsize}\selectfont}%
  \ifx\svgwidth\undefined%
    \setlength{\unitlength}{503.26484299bp}%
    \ifx\svgscale\undefined%
      \relax%
    \else%
      \setlength{\unitlength}{\unitlength * \real{\svgscale}}%
    \fi%
  \else%
    \setlength{\unitlength}{\svgwidth}%
  \fi%
  \global\let\svgwidth\undefined%
  \global\let\svgscale\undefined%
  \makeatother%
  \begin{picture}(1,0.33944102)%
    \lineheight{1}%
    \setlength\tabcolsep{0pt}%
    \put(0,0){\includegraphics[width=\unitlength,page=1]{complex.pdf}}%
    \put(0.23747425,0.1544668){\color[rgb]{0,0,0}\makebox(0,0)[lt]{\lineheight{1.25}\smash{\begin{tabular}[t]{l}$\widetilde{R}_\theta$\end{tabular}}}}%
    \put(0,0){\includegraphics[width=\unitlength,page=2]{complex.pdf}}%
    \put(0.91678073,0.07945342){\color[rgb]{0,0,0}\makebox(0,0)[lt]{\lineheight{1.25}\smash{\begin{tabular}[t]{l}$e^{hor}_\theta$\\\end{tabular}}}}%
    \put(0.91672606,0.15107278){\color[rgb]{0,0,0}\makebox(0,0)[lt]{\lineheight{1.25}\smash{\begin{tabular}[t]{l}$v^1_\theta$\end{tabular}}}}%
    \put(0.71205337,0.14034354){\color[rgb]{0,0,0}\makebox(0,0)[rt]{\lineheight{1.25}\smash{\begin{tabular}[t]{r}$v^0_\theta$\end{tabular}}}}%
    \put(0,0){\includegraphics[width=\unitlength,page=3]{complex.pdf}}%
    \put(0.91536443,0.24044977){\color[rgb]{0,0,0}\makebox(0,0)[lt]{\lineheight{1.25}\smash{\begin{tabular}[t]{l}$f_\theta$\end{tabular}}}}%
    \put(0,0){\includegraphics[width=\unitlength,page=4]{complex.pdf}}%
    \put(0.71072126,0.18449765){\color[rgb]{0,0,0}\makebox(0,0)[rt]{\lineheight{1.25}\smash{\begin{tabular}[t]{r}$e^{vert,0}_\theta$\end{tabular}}}}%
    \put(0,0){\includegraphics[width=\unitlength,page=5]{complex.pdf}}%
    \put(0.91657377,0.19027897){\color[rgb]{0,0,0}\makebox(0,0)[lt]{\lineheight{1.25}\smash{\begin{tabular}[t]{l}$e^{vert,1}_\theta$\end{tabular}}}}%
  \end{picture}%
\endgroup%